%% file: main_document.tex
\documentclass[10pt,a4paper,notitlepage, oneside]{article}
\usepackage[utf8x]{inputenc} 
\usepackage{ucs} 
\usepackage[english]{babel}
\usepackage{amsmath}
\usepackage{amsthm}
\usepackage{amsfonts}
\usepackage{amssymb}
\usepackage{mathtools}
\usepackage{cite}
\usepackage{hyperref}
\hypersetup{
 colorlinks=true,  
 linkcolor=black,    
 urlcolor=red      
}

\usepackage{enumerate} 

\usepackage{tikz}
\usetikzlibrary{decorations.pathmorphing}
\usetikzlibrary{backgrounds}

\usepackage{tkz-euclide}
\usetkzobj{all}

\usepackage{xspace}  
\newcommand{\low}{loose odd wheel\xspace}

\newcommand{\subpath}[4]{  P_{#1, #2 - #3,#4} }

\usepackage[]{todonotes}  

\newtheorem{definition}{Definition}

\newtheorem{theorem}[definition]{Theorem}
\newtheorem{corollary}[definition]{Corollary}
\newtheorem{observation}[definition]{Observation}

\newtheorem{lemma}[definition]{Lemma}
\newtheorem{conjecture}[definition]{Conjecture}

\newtheorem{claim}[definition]{Claim}

\newcommand{\ssp}{\textrm{\rm SSP}}
\newcommand{\tstab}{\textrm{\rm TSTAB}}
\newcommand{\charf}[1]{\raisebox{\depth}{\(\chi\)}_{#1}}

\newcommand{\N}{\ensuremath{\mathbb{N}}} 
\newcommand{\Z}{\ensuremath{\mathbb{Z}}} 
\newcommand{\R}{\ensuremath{\mathbb{R}}} 

\usepackage{tikz}
\usetikzlibrary{arrows,positioning,intersections,}
\usetikzlibrary{backgrounds}
\usetikzlibrary{trees}

\tikzstyle{hvertex}=[circle,inner sep=0.cm, minimum size=.8mm, fill=black, draw=black]
\tikzstyle{novertex}=[rectangle]
\tikzstyle{invvertex}=[circle,inner sep=0.cm,fill=white]
\tikzstyle{hedge}=[thick]
\tikzstyle{ledge}=[dashed]
\tikzstyle{dedge}=[ultra thick, densely dotted]
\tikzstyle{pedge}=[ ->, thick]
\colorlet{hellgrau}{black!30!white}

\title{On $t$-perfect triangulations of the projective plane}
\author{Elke Fuchs and Laura Gellert}
\date{}
\begin{document}
\maketitle

\begin{abstract}
We prove that a triangulation of the projective plane is (strongly) $t$-perfect if and only if it is perfect and contains no $K_4$.
\end{abstract}

\section{Introduction} \label{sec:introduction}

\input{sec-introduction}

\section{ (Strong) $t$-perfection} \label{sec:t-perfection}

\input{sec-t-perfection}

\section{Triangulations} \label{sec:triangulations}

\input{sec-triangulations}

\section{Proof of Theorem~\ref{thm:mainthm_4_equivalences} }\label{sec:proof}

\input{sec-proof}

\section{Irreducible Triangulations}\label{sec:irreducible}
\input{sec-irred_triang}
\section{Even-contraction creating a loose odd wheel or an induced $\overline{C_7}$}\label{sec:even-contraction-low-barC7}

\input{sec-even_split_low_antihole}

\section{Even-contraction destroying an odd hole}\label{sec:even-contraction}

\input{sec-even_contraction}

\bibliographystyle{abbrv}

\bibliography{bib-pp_triang} 

\vfill

\noindent
Elke Fuchs
{\tt <elke.fuchs@uni-ulm.de>}\\
Laura Gellert
{\tt <laura.gellert@uni-ulm.de>}\\
Institut f\"ur Optimierung und Operations Research\\
Universit\"at Ulm, Ulm\\
Germany\\

\end{document}

%% file: sec-introduction.tex
Perfect graphs received considerous attention in graph theory. The purely combinatorial definition --- that a graph is perfect if and only if for each of its induced subgraphs, the chromatic number and the clique number coincide --- can be replaced by a polyhedral one~(Chv\`{a}tal~\cite{Chvatal75}):
A graph $G$ is perfect if and only if its \emph{stable set polytope $\ssp(G)$}, ie the convex hull of stable sets of~$G$, is determined by non-negativity and clique inequalities. This happens if and only if the polyhedron defined by the system of non-negativity and clique inequalities is integral. In this case, the system is also totally dual integral (see~eg~\cite[Ch.65.4]{LexBible}).

Modification of these inequalities leads to generalisations of perfection.
A graph~$G$ is \emph{$t$-perfect} if $\ssp(G)$ is determined by non-negativity-, edge- and odd-cycle inequalities. 
A graph $G$ is \emph{strongly $t$-perfect} if the described system is totally dual integral. Clearly, strong $t$-perfection implies $t$-perfection. It is not known whether the converse is also true.


We characterise (strongly) $t$-perfect triangulations of the projective plane.

\begin{theorem}\label{thm:mainthm_4_equivalences}
For every triangulation $G$ of the projective plane the following assertions are equivalent:
\begin{enumerate}[(a)]
\item $G$ is $t$-perfect \label{item:t-perfect}
\item $G$ is strongly  $t$-perfect \label{item:strongly_t-perfect}
\item $G$ is perfect and contains no $K_4$ \label{item:perfect_without_K4}
\item $G$ contains no \low and no $\overline{C_7}$ as an induced subgraph. \label{item:no_low_C7bar}
\end{enumerate}
\end{theorem}

The definition of a \low is given in Section~\ref{sec:t-perfection}.  

One of the main open questions about $t$-perfection is, whether a $t$-perfect graph can always be coloured with few colours.
Standard polyhedral methods assert that the fractional chromatic number of a
$t$-perfect graph is at most $3$.

Laurent and Seymour as well as Benchetrit found examples of a $t$-perfect graph that is not $3$-colourable (see~\cite[p.~1207]{Schri02} and~\cite{Benchetrit2015}).

\begin{conjecture}[Shepherd, Seb\H o]
Every $t$-perfect graph is $4$-colourable.
\end{conjecture}

The conjecture is known to hold in a number of graph classes, for instance in $P_6$-free graphs (Benchetrit~\cite{Benchetrit2016}), claw-free graphs (Bruhn and Stein~\cite{Bru_Stei12}), and in $P_5$-free graphs (Bruhn and Fuchs~\cite{Bru_Fu15}). In the last two classes, the graphs are even $3$-colourable. 

The verification of the conjecture for triangulations of the projective plane follows directly from Theorem~\ref{thm:mainthm_4_equivalences}: A $t$-perfect triangulation is perfect without $K_4$ and thus $3$-colourable.

\begin{corollary} \label{cor:4_colours->t-imp}
Every $t$-perfect triangulation of the projective plane is $3$-colourable.
\end{corollary}

In order to prove Theorem~\ref{thm:mainthm_4_equivalences}, we mainly study Eulerian triangulations, ie triangulations  where all vertices are of even degree.\footnote{Sometimes, an Eulerian triangulation is referred to as an \emph{even triangulation}.}
  Colourings of Eulerian triangulations were studied by Hutchinson et al.~\cite{HuRiSey02} and by Mohar~\cite{Mohar02}.  
 Suzuki and Watanabe~\cite{SuWa07} determined a family of Eulerian triangulations of the projective plane such that every Eulerian triangulation of the projective plane can be transformed into one of its members by application of two operations (see Section~\ref{sec:proof} for more details). 
  Barnette~\cite{Bar82} showed that one can obtain each triangulation of the projective plane from one of two minimal triangulations using two kinds of splitting operation.
 

A general treatment of (strong) $t$-perfection may be found 
in Schrijver~\cite[Vol.~B, Ch.~68]{LexBible}. 
The classes of $t$-perfect and strongly $t$-perfect graphs are closed under taking induced subgraphs and under simultaneous contractions of all  edges incident to  a vertex with stable neighbourhood. However, no  characterisation in terms of forbidden induced subgraphs is known. 
Bruhn and Benchetrit showed that plane triangulations are $t$-perfect if and only if they do not contain a certain subdivision of an odd wheel as an induced subgraph~\cite{Ben_Bru15}. 
We showed that quadrangulations of the projective plane are (strongly) $t$-perfect if and only if they are  bipartite~\cite{tpquad}. 
Boulala and Uhry~\cite{BouUhr79} established the $t$-perfection of series-parallel graphs.
It was shown by Gerards and Schrijver~\cite{Ge_Schri86} that each graph without an odd $K_4$ is $t$-perfect (an \emph{odd $K_4$} is a subdivision of a $K_4$ where each triangle has become an odd cycle).  Gerards~\cite{Gerards89} showed that such graphs are also strongly $t$-perfect. 
However, there exist odd $K_4$s which are not $t$-perfect 
--- the \emph{bad $K_4$s}.
Gerards and Shepherd~\cite{GS98} proved that graphs without bad $K_4$ can be recognised in polynomial time. Additionally, they showed that each graph without a bad $K_4$ as a subgraph is $t$-perfect. Schrijver~\cite{Schri02} extended this to strongly $t$-perfect graphs.




%


%% file: sec-t-perfection.tex
 All the graphs mentioned here are finite, simple. We follow the notation of Diestel~\cite{Diestel}.

Let $G=(V,E)$ be a graph.
The \emph{stable set polytope} $\ssp(G)\subseteq\mathbb R^{V}$ of $G$ 
is defined as the convex hull of the characteristic vectors of stable, ie independent, subsets of $G$.
The characteristic vector of a subset $S \subseteq V$ is the vector $\charf{S}~\in~\{0,1\}^{V}$ with $\charf{S}(v)=1$ if $v\in S$ and $0$ otherwise. 
We define a second polytope $\tstab(G)\subseteq\mathbb R^V$ for $G$, given by
\begin{eqnarray}
&&x\geq 0,\notag\\
&&x_u+x_v\leq 1\text{ for every edge }uv\in E, \label{eq_tstab}\\ 
&&\sum_{v\in V(C)}x_v\leq \left\lfloor ^{|C|}\slash_2\right\rfloor\text{ for every induced odd cycle }C 
\text{ in }G.\notag
\end{eqnarray}
These inequalities are respectively known as non-negativity, edge and
odd-cycle inequalities. Clearly, $\ssp(G)$ is contained in $ \tstab(G)$.
The graph $G$ is called \emph{$t$-perfect} if $\ssp(G)$ and
$\tstab(G)$ coincide. Equi\-va\-lently, $G$ is $t$-perfect if and only if
$\tstab(G)$ is an integral polytope, ie if all its vertices
are integral vectors.


A graph $G=(V,E)$ is \emph{strongly $t$-perfect} if the system \eqref{eq_tstab}  is \emph{totally dual integral}. That is, if for each weight vector $w \in \Z^{V}$, the linear program of maximizing $w^Tx$ over \eqref{eq_tstab} has an integer optimum dual solution. 
This property implies that $\tstab(G)$ is integral. Therefore, 
\begin{equation} \label{eq:strong_tp->tp}
\text{strong $t$-perfection implies $t$-perfection.}
\end{equation}
It is an open question whether every $t$-perfect graph is strongly $t$-perfect. 
The question is briefly discussed in Schrijver~\cite[Vol. B, Ch. 68]{LexBible}.

It is easy to verify that vertex deletion preserves $t$-perfection.
Another operation that keeps $t$-perfection was found by 
Gerards and Shepherd~\cite{GS98}: whenever there is a vertex $v$, so that its
neighbourhood is stable, we may contract all edges incident with $v$
simultaneously. We will call this operation a  \emph{$t$-contraction at~$v$}.
Any graph that is obtained from $G$ by a sequence of vertex deletions 
and $t$-contractions is a \emph{$t$-minor of $G$}.
Let us point out that any $t$-minor of a $t$-perfect graph is again $t$-perfect. 
The same holds for strong $t$-perfection (see eg~\cite[Vol.~B,~Ch.~68.4]{LexBible}).

A graph $G$ is \emph{minimally $t$-imperfect} if $G$ is $t$-imperfect, but all its proper $t$-minors are not $t$-perfect. A \emph{wheel} $W_p$ is a graph consisting of a $p$-cycle $w_1, \ldots, w_p, w_1$ (with $p \geq 3$) and a center vertex $v$ adjacent to $w_i$ for $i=1, \ldots, p$; see Figure~\ref{fig:wheels}. 
The wheel $W_p$ is an \emph{odd wheel} if $p$ is odd. 
\input{fig-wheels}
It is well-known (see eg~\cite[Vol. B, Ch. 68.4]{LexBible}) that 
\begin{equation} \label{eq:odd_wheels_t-imp}
\text{odd wheels are minimally $t$-imperfect.}
\end{equation}
Indeed, the vector $\left(^{1}\slash_3, \ldots, ^{1}\slash_3\right)$ is contained in $\tstab(W_{2k+1})$ but not in $\ssp(W_{2k+1})$ for $k \geq 1$.

%

For a cycle $C$ and three vertices $v_1,v_2,v_3 \in V(C)$ we denote by $\subpath{C}{v_1}{v_2}{v_3} $ the path connecting $v_1$ and $v_2$ along $C$ that does not contain $v_3$. Figure~\ref{fig-def_seg} illustrates this. Note that wiggly lines represent paths. 
\input{fig-def_seg}

Let $C$ be an odd cycle and let $v \notin V(C)$ be a vertex.
Three vertices $v_1,v_2,v_3$ are called three \emph{odd neighbours} of $v$ on $C$, if they are neighbours of $v$ and the paths $\subpath{C}{v_1}{v_2}{v_3}$, $\subpath{C}{v_1}{v_3}{v_2}$ and $\subpath{C}{v_2}{v_3}{v_1}$ on $C$ are odd.

A \emph{loose odd wheel}  consists of an odd cycle $C$ and a vertex $v \notin V(C)$ that has three odd neighbours on $C$.
Evidently, every odd wheel is a \low and every graph that contains an odd wheel as a subgraph also contains an odd wheel as an induced subgraph. 
Further, every loose odd wheel has an odd wheel as a $t$-minor. 
As vertex-deletion and $t$-contraction preserve $t$-perfection it follows directly from~\eqref{eq:odd_wheels_t-imp} that
\begin{equation}\label{eq:low->t-imp}
\text{a $t$-perfect graph contains no loose odd wheel.}
\end{equation}

The \emph{chromatic number} of a graph is the smallest number of colours needed to colour its vertices. A graph is \emph{perfect}, if for every induced subgraph the chromatic number and the size of a largest clique coincide. 

An \emph{odd hole} $C_k$ is an induced odd cycle on $k \geq 4$ vertices. An \emph{odd anti-hole} $\overline{C_k}$ is the complement of an odd hole. The anti-hole 
\begin{equation} \label{eq:7-antihole_t-imp}
\text{ $\overline{C_7}$ is minimally $t$-imperfect.}
\end{equation}

It is easy to check that odd holes and odd anti-holes are not perfect. The well-known Strong Perfect Graph Theorem 
states that the converse is also true.

\begin{theorem}[Chudnovsky, Robertson, Seymour, and Thomas~\cite{SPGT} ] \label{thm:SPGT} 
A simple graph $G$ is perfect if and only if $G$ contains no odd hole or odd anti-hole.
\end{theorem}

The incidence vector $x \in \R^{V(G)}$ of a stable set in a graph $G$ satisfies
\begin{eqnarray} 
&&x\geq 0,\notag\\
&&\sum_{v\in V(C)}x_v\leq 1\text{ for every clique }C \text{ in }G \label{eq:clique-inequ}.
\end{eqnarray}

Results of Lov{\'a}sz~\cite{Lovasz72}, Fulkerson~\cite{Fulkerson72}, and Chv\'atal~\cite{Chvatal75} showed that $G$ is perfect if and only if  \eqref{eq:clique-inequ} suffices to describe the stable set polytope $\ssp(G)$. This is the case if and only if the system~\eqref{eq:clique-inequ} is totally dual integral (see e.g.~\cite[[Vol.~B,~Ch.~65.4]{LexBible}). Thus, 

\begin{equation} \label{eq:perfect<->TDI}
\text{ $G$ is perfect if and only if  \eqref{eq:clique-inequ} is totally dual integral.}
\end{equation}

With this equivalence, it is easy to see one of the implications stated in Theorem~\ref{thm:mainthm_4_equivalences}.

\begin{lemma} \label{lem:perfect_no_K4->t-perfect}
Every perfect graph without $K_4$ is (strongly) $t$-perfect.
\end{lemma}

\begin{proof}
Let $G$ be perfect with no clique of size at least $4$. 
Then, \eqref{eq_tstab}  is equivalent to \eqref{eq:clique-inequ}. 
The system is totally dual integral as $G$ is perfect (see~\eqref{eq:perfect<->TDI}). This shows that $G$ is (strongly) $t$-perfect. 
\end{proof}

%% file: fig-wheels.tex
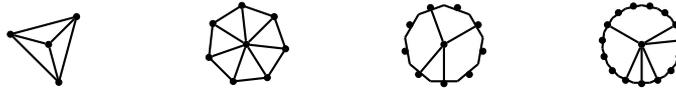
\begin{figure}[htb]
\def\krad{0.4cm}
\newcommand{\wheel}[1]{
\def\angle{360/#1}
\pgfmathtruncatemacro{\nminusone}{#1-1}
\node[hvertex] (c) at (0,0){};
\foreach \i in {0,...,\nminusone}{
  \begin{scope}[on background layer]
    \draw[hedge] (225-\angle/2+\i*\angle:\krad) -- (225+\angle/2+\i*\angle:\krad);
  \end{scope}
  \node[hvertex] (v\i) at (225-\angle/2+\i*\angle:\krad){};
  \draw[hedge] (c) -- (v\i);
}
}

\begin{center}
\begin{tikzpicture}[scale=1.3]

\wheel{3}

\begin{scope}[shift={(2,0)}]
\wheel{7}
\end{scope}

\begin{scope}[shift={(4,0)}]
\node[hvertex] (c) at (0,0){};
\foreach \i in {0,...,8}{
  \node[hvertex] (v\i) at (270-40+\i*40:\krad){};
  \draw[hedge] (270-20+\i*40:\krad) -- (270+20+\i*40:\krad);
}
\foreach \i in {0,1,4,6}{
  \begin{scope}[on background layer]
    \draw[hedge] (c) -- (v\i);
  \end{scope}
}
\end{scope}

\begin{scope}[shift={(6,0)}]
\node[hvertex] (c) at (0,0){};
\foreach \i in {0,...,14}{
  \node[hvertex] (v\i) at (270-24+\i*24:\krad){};
  \draw[hedge] (270-12+\i*24:\krad) -- (270+12+\i*24:\krad);
}
\foreach \i in {0,1,2,5,6,11}{
  \begin{scope}[on background layer]
    \draw[hedge] (c) -- (v\i);
  \end{scope}
}
\end{scope}

\end{tikzpicture}
\end{center}
\caption{The odd wheels $W_3=K_4$ and $W_7$ as well as two loose odd wheels}
\label{fig:wheels}

\end{figure}

%% file: fig-def_seg.tex
\begin{figure}

\begin{center}
\begin{tikzpicture}[scale=1]



%
\node[novertex] (l) at (-1.8,0.2){$\subpath{C}{v_1}{v_2}{v_3} $};
\node[novertex] (r) at (1.8,0.2){$\subpath{C}{v_1}{v_3}{v_2} $};
\node[novertex] (o) at (0,-1.2){$\subpath{C}{v_2}{v_3}{v_1} $};

\node[hvertex] (v1) at (0,.5){};
\node[novertex] (lv1) at (0,.7){$v_1$};
\node[hvertex] (v3) at (.5,-.5){};
\node[novertex] (lv3) at (.7,-.7){$v_3$};
\node[hvertex] (v2) at (-.5,-.5){};
\node[novertex] (lv2) at (-.7,-.7){$v_2$};

\draw [decorate, decoration={snake,amplitude=.4mm,segment length=2mm,post length=0mm}] (-0.5,-0.5) arc (-180:0:.5);
\draw [decorate, decoration={snake,amplitude=.4mm,segment length=2mm,post length=0mm}] (0.5,-0.5) arc (-60:116:.56);

\draw [decorate, decoration={snake,amplitude=.4mm,segment length=2mm,post length=0mm}] (0,0.5) arc (60:241:.56);


\end{tikzpicture}
\end{center}
\caption{The paths $\subpath{C}{v_1}{v_2}{v_3} $, $\subpath{C}{v_1}{v_3}{v_2}$ and $\subpath{C}{v_2}{v_3}{v_1}$ on a cycle $C$}

\label{fig-def_seg}

\end{figure}
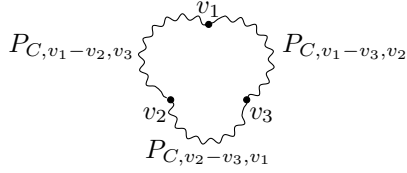

%% file: sec-triangulations.tex
A \emph{triangulation} $G$ of the projective plane is a finite simple graph embedded on the surface such that every face of $G$ is bounded by a $3$-cycle. 
A cycle $C$ in the projective plane is \emph{contractible} if $C$ separates the projective plane into two sets $S_C$ and $\overline{S_C}$ where $S_C$ is homeomorphic to an open disk in $\R^2$. We call $S_C$ the \emph{interior} of $C$. 
A triangulation is \emph{nice} if no vertex is contained in the interior of  a contractible $3$-cycle.

As a contractible cycle separates $G$, the following theorem of Chvátal implies that 
\begin{equation} \label{eq:nice_suffices}
\text{ it suffices to analyse $t$-(im)perfection of nice triangulations.}
\end{equation}


\begin{theorem} [Chvátal~\cite{Chvatal75}] \label{thm:clique-sum}
Let $G$ be a graph with a clique separator $X$, and let $C_1 ,\ldots , C_k$ be the components of $G − X$. Then, $G$ is  $t$-perfect respectively perfect if and only if $G [ C_1 ∪ X ],\ldots, G [ C_k ∪ X ]$ are $t$-perfect respectively perfect.
\end{theorem} 

Ringel~\cite{Ri78} pointed out that the neighbourhood of any vertex in a triangulation contains a Hamilton cycle:

\begin{theorem}[Ringel~\cite{Ri78}]\label{thm:locally_Hamiltonian} 
Let $G$ be a triangulation of a closed surface and let $v$ be a vertex with neighbourhood $ \lbrace v_0=v_d,v_1, \ldots, v_{d-1} \rbrace $ where $v,v_i,v_{i+1},v$ is a contractible triangle for $i=0, \ldots, d-1$. Then, the cycle $v_0,v_1, \ldots, v_{d-1},v_0$ is a contractible Hamilton cycle in $N_G(v)$.
\end{theorem}

We denote the contractible Hamilton cycle of $v$ described in Theorem~\ref{thm:locally_Hamiltonian} by~\emph{$HC(v)$} and note:

\begin{observation}\label{obs:induced_cycles}  
Let $G$ be a nice triangulation with a non-contractible cycle $C$. Let $u,v,w$ be three consecutive vertices on $C$. 
Then the two paths between $u$ and $w$ along the Hamilton cycle $HC(v)$ are induced.  
\end{observation}

\begin{proof}
Suppose that the path $P=v_1, v_2, \ldots, v_k$  with $u=v_1$ and $w=v_k$ on the Hamilton cycle of $v$ is not induced. Then, there are vertices  $v_i$ and $v_{i + \ell}$ of $P$ with $\ell \geq 2$ that are adjacent in $G$. As $C$ is non-contractible, the triangle $v,v_i,v_{i + \ell},v$ must be contractible and its interior contains the vertices $v_{i+1}, \ldots, v_{i + \ell-1}$. This shows that $G$ is not a nice triangulation.
\end{proof}

A triangulation is called~\emph{Eulerian} if each vertex has even degree. The next
observation implies that a $t$-perfect triangulation must be Eulerian.

\begin{observation} \label{obs:Eulerian}
Let $G$ be a triangulation of any surface and let $G$ contain a vertex whose neighbourhood is not bipartite. Then, $G$ does not satisfy any of the assertions given in Theorem~\ref{thm:mainthm_4_equivalences}. 
\end{observation} 

\begin{proof}
Let the neighbourhood of $v \in V(G)$ contain an odd induced cycle $C$. 
Then, $C$ and $v$ form an odd wheel $W$ and $G$ is not (strongly) $t$-perfect by~\eqref{eq:low->t-imp}. If $C$ is a triangle, $W$ is a $K_4$; otherwise, $C$ is an odd hole and $G$ is imperfect (see~Theorem~\ref{thm:SPGT}).
\end{proof}

In a triangulation, a vertex of odd degree has a non-bipartite neighbourhood (Theorem~\ref{thm:locally_Hamiltonian}). It thus follows from Observation~\ref{obs:Eulerian} that in order to prove Theorem~\ref{thm:mainthm_4_equivalences},
\begin{equation}\label{eq:Eulerian}
\text{ we can assume that the triangulation is Eulerian.}
\end{equation}

The next theorem follows directly from a characterisation of $t$-perfect triangulations of the plane given by  Bruhn and Benchetrit~\cite[Theorem 2]{Ben_Bru15}.

\begin{theorem}\label{thm:cycles_triang} 
Every triangulation $G$ of the projective plane that contains a contractible odd hole also contains a loose odd wheel. 
\end{theorem}

To prove Theorem~\ref{thm:mainthm_4_equivalences}, we thus can assume that  in our triangulation 
\begin{equation}\label{eq:odd_cycles_noncontractible}
\text{all odd cycles are non-contractible.}
\end{equation}

There is another useful property of a graph that forces a loose odd wheel.

\begin{observation} \label{obs:deg4->-loose_odd_wheel}
Let $G$ be a triangulation of the projective plane. If $G$ contains an odd hole with a vertex of degree $4$, then $G$ contains a loose odd wheel. 
\end{observation}

\begin{proof} 
Let $u,v,w$ be three consecutive vertices on an odd hole and let $v$ have neighbourhood $N_G(v)=\lbrace u,w,x,y \rbrace$. As $u$ and $w$ cannot be adjacent, the Hamilton cycle around $v$ equals $u,x,w,y,u$ and $C$ forms a loose odd wheel together with $x$.
\end{proof}


As we have seen in~\eqref{eq:Eulerian}, it suffices to analyse Eulerian triangulations of the projective plane to prove Theorem~\ref{thm:mainthm_4_equivalences}. Suzuki and Watanabe~\cite{SuWa07}  introduced the following operations that modify Eulerian triangulations.

\begin{definition} \label{def:operations}
Let $G$ be an Eulerian triangulation of the projective plane. Let~$x \in V(G)$ be a vertex with Hamilton cycle $a,b,a',b',a$ where the set of common neighbours of $b$ and $b'$ equals $ \lbrace a,a',x \rbrace$. An~\emph{even-contraction} at $x$ together with $b$ and $b'$ identifies the vertices $x,b,b'$ to a new vertex $y$ and removes loops as well as multiple edges. \\
The inverse operation is called an~\emph{even-splitting} at $y$. 
\end{definition}

\input{fig-def_operations}
Figure~\ref{fig:def_operations} shows an even-contraction. 
Note that the graph obtained from an Eulerian triangulation by an even-contraction or an even-splitting is again an Eulerian triangulation of the projective plane. 
An even-splitting is always unique: The embedding directly yields the partition of the neighbours of $y$ into neighbours of $b$ and of $b'$.

We note some useful observations: \\
If we can apply an even-contraction, then
\begin{equation}
N_G(b) \cap N_G(b')= \lbrace  a,a',x \rbrace.\label{eq:neighbours_of_b_b'}
\end{equation}
Further, not only the common neighbours of $b$ and $b'$ are restricted.
If  $bb' \in E(G)$, then $\lbrace b,b',a,x \rbrace$  induces a $K_4$, thus
\begin{equation}
\text{ we may assume that $bb'\notin E$.} \label{eq:b_b'_not_adjacent}
\end{equation}

Suzuki and Watanabe~\cite{SuWa07} pointed out that
even-contraction and even-splitting preserve $3$-colourability.
This leads to the following observation:

\begin{observation} \label{obs:3colours}
Let $G'$ be obtained from  $G$ by an even-contraction. If $G'$ is perfect without $K_4$, then $G$ is $3$-colourable.
\end{observation}

Suzuki and Watanabe~\cite{SuWa07} defined one more operation:

\begin{definition} 
Let $G$ be an Eulerian triangulation of the  projective plane. 
Let $u,v,w, u$ be  a contractible triangle in $G$ whose interior contains only the vertices $ x,y,z $ and the edges $xy, xz, yz, ux,uy,vy,vz,wx, wz$. The deletion of the vertices $x,y,z$ is called a \emph{deletion of an octahedron.}\\
The inverse operation is called an~\emph{attachment of an octahedron}. 
\end{definition}

Figure~\ref{fig:def_operations} shows a deletion of an octahedron. 

We call an Eulerian triangulation~\emph{irreducible} if no deletion of an octahedron and no even-contraction can be applied. Suzuki and Watanabe~\cite{SuWa07} listed all irreducible Eulerian triangulations. These graphs are treated in Section~\ref{sec:irreducible}.

%% file: fig-def_operations.tex
\begin{figure}
\begin{center}
\begin{tikzpicture} [ scale = .37]

\begin{scope} [shift={(0,0.5)}]
\node[hvertex] (x0) at (0,0){};
\node[hvertex] (x1) at (-1.5,2){};
\node[hvertex] (x2) at (1.5,2){};
\node[hvertex] (x3) at (1.5,-2){};
\node[hvertex] (x4) at (-1.5,-2){};

\node[novertex] (b2) at (-2.5,2.7){};
\node[novertex] (b3) at (-1.5,3.2){};
\draw[hedge] (b2) -- (x1);
\draw[hedge] (b3) -- (x1);

\node[novertex] (b'2) at (2.4,-2.6){};
\node[novertex] (b'3) at (1.7,-3){};
\draw[hedge] (b'2) -- (x3);
\draw[hedge] (b'3) -- (x3);

\node[novertex] (y0) at (0.8,0){$x$};
\node[novertex] (y1) at (-2.2,1.8){$b$};
\node[novertex] (y2) at (2,1.8){$a$};
\node[novertex] (y3) at (2.2,-1.5){$b'$};
\node[novertex] (y4) at (-2,-1.5){$a'$};

\draw[hedge] (x1) -- (x2);
\draw[hedge] (x2) -- (x3);
\draw[hedge] (x3) -- (x4);
\draw[hedge] (x4) -- (x1);

\draw[hedge] (x0) -- (x1);
\draw[hedge] (x0) -- (x2);
\draw[hedge] (x0) -- (x3);
\draw[hedge] (x0) -- (x4);
\end{scope}

\begin{scope} [shift={(5.5,1)}]

\node[novertex] (y1) at (-3,.3){};
\node[novertex] (y2) at (-3,-.2){};

\node[novertex] (y3) at (-1,.3){};
\node[novertex] (y4) at (-1,-.2){};

\draw[pedge] (y1) -- (y3);

\draw[pedge] (y4) -- (y2);

\end{scope}

\begin{scope} [shift={(7,0.5)}]

\node[hvertex] (x0) at (0,0){};
\node[hvertex] (x1) at (-1.5,-2){};
\node[hvertex] (x3) at (1.5,2){};

\node[novertex] (y0) at (1,0){$y$};
\node[novertex] (y2) at (2,1.8){$a$};
\node[novertex] (y4) at (-2,-1.5){$a'$};

\draw[hedge] (x0) -- (x1);
\draw[hedge] (x0) -- (x3);

\node[novertex] (b2) at (-1,.7){};
\node[novertex] (b3) at (0,1.2){};
\draw[hedge] (b2) -- (x0);
\draw[hedge] (b3) -- (x0);

\node[novertex] (b'2) at (0.9,-.6){};
\node[novertex] (b'3) at (.2,-1){};
\draw[hedge] (b'2) -- (x0);
\draw[hedge] (b'3) -- (x0);

\end{scope}

\begin{scope}  [shift={(15,0)}]


\node[hvertex] (x1) at (0,-1){};
\node[novertex] (lx1) at (0,-1.6){$y$};
\node[hvertex] (x2) at (-.6,-0.1){};
\node[novertex] (lx2) at (-1,0.1){$x$};
\node[hvertex] (x3) at (.6,-0.1){};
\node[novertex] (lx3) at (1,0.1){$z$};

\draw[hedge] (x1) -- (x2);
\draw[hedge] (x2) -- (x3);
\draw[hedge] (x3) -- (x1);

\node[hvertex] (x4) at (0,3){};
\node[novertex] (lx4) at (0.8,3){$w$};
\node[hvertex] (x5) at (-3,-2){};
\node[novertex] (lx5) at (-3.4,-2.5){$u$};
\node[hvertex] (x6) at (3,-2){};
\node[novertex] (lx6) at (3.4,-2.5){$v$};

\draw[hedge] (x4) -- (x5);
\draw[hedge] (x5) -- (x6);
\draw[hedge] (x6) -- (x4);

\draw[hedge] (x1) -- (x6);
\draw[hedge] (x1) -- (x5);

\draw[hedge] (x2) -- (x5);
\draw[hedge] (x2) -- (x4);

\draw[hedge] (x3) -- (x6);
\draw[hedge] (x3) -- (x4);

\end{scope}

\begin{scope} [shift={(21,1)}]

\node[novertex] (y1) at (-3,.3){};
\node[novertex] (y2) at (-3,-.2){};

\node[novertex] (y3) at (-1,.3){};
\node[novertex] (y4) at (-1,-.2){};

\draw[pedge] (y1) -- (y3);

\draw[pedge] (y4) -- (y2);

\end{scope}

\begin{scope} [shift={(23,0)}]

\node[hvertex] (x4) at (0,3){};
\node[hvertex] (x5) at (-3,-2){};
\node[hvertex] (x6) at (3,-2){};

\draw[hedge] (x4) -- (x5);
\draw[hedge] (x5) -- (x6);
\draw[hedge] (x6) -- (x4);

\node[novertex] (lx4) at (0.8,3){$w$};
\node[hvertex] (x5) at (-3,-2){};
\node[novertex] (lx5) at (-3.4,-2.5){$u$};
\node[hvertex] (x6) at (3,-2){};
\node[novertex] (lx6) at (3.4,-2.5){$v$};

\end{scope}

\end{tikzpicture}
\end{center}

\caption{Even-contraction and deletion of an octahedron}
\label{fig:def_operations}
\end{figure}
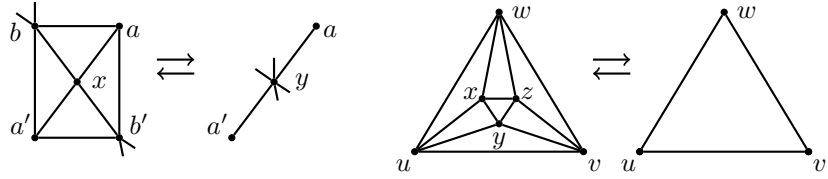

%% file: sec-proof.tex
The proof of Theorem~\ref{thm:mainthm_4_equivalences} is inductive. Lemma~\ref{lem:irreducible} provides the induction start. For the induction step, we consider a nice triangulation with an odd hole to which we apply an even-contraction. Lemma~\ref{lem:even-contraction}, Lemma~\ref{lem:even-split_low} and Lemma~\ref{lem:even-split_anti-hole} treat the different structures of the obtained graph.

\begin{lemma}
\label{lem:irreducible}
Every irreducible triangulation $G$ of the projective plane that contains no \low and no $\overline{C_7}$ as an induced subgraph is perfect and contains no $K_4$.
\end{lemma}
This lemma will be shown in Section~\ref{sec:irreducible}.

\begin{lemma} \label{lem:even-contraction}
Let $G$ be a nice Eulerian triangulation and let $G'$ be obtained from $G$ by an even-contraction. If  $G$ contains an odd hole and $G'$ is perfect, then $G$ contains a loose odd wheel.
\end{lemma}
The proof of this lemma can be found in~Section~\ref{sec:even-contraction}. The proofs of the following two lemmas appear in Section~\ref{sec:even-contraction-low-barC7}.

\begin{lemma} \label{lem:even-split_low}
Let $G$ be a nice Eulerian triangulation and let $G'$ be obtained from $G$ by an even-contraction. If $G'$ contains a loose odd wheel, then $G$ contains a loose odd wheel.
\end{lemma}

\begin{lemma} \label{lem:even-split_anti-hole}
Let $G$ be a nice Eulerian triangulation and let $G'$ be obtained from $G$ by an even-contraction. If $G'$ contains an induced $\overline{C_7}$, then $G$ contains a $\overline{C_7}$ or a loose odd wheel as an induced subgraph.
\end{lemma}

Now, we are able to prove the main theorem.

\begin{proof} [Proof of Theorem~\ref{thm:mainthm_4_equivalences}]

If a graph $G$ is perfect without $K_4$, then $G$ is strongly $t$-perfect and thus $t$-perfect by Lemma~\ref{lem:perfect_no_K4->t-perfect}. Therefore,~\eqref{item:perfect_without_K4} implies~\eqref{item:strongly_t-perfect} and~\eqref{item:strongly_t-perfect} implies~\eqref{item:t-perfect}.

If a graph $G$ is $t$-perfect, then $G$ contains no \low or $\overline{C_7}$; see~\eqref{eq:low->t-imp} and~\eqref{eq:7-antihole_t-imp}. Thus,~\eqref{item:t-perfect} implies~\eqref{item:no_low_C7bar}.

In order to show that~\eqref{item:no_low_C7bar} implies~\eqref{item:perfect_without_K4},
we now consider an imperfect triangulation $G$ of the projective plane and prove that $G$ contains $\overline{C_7}$ or a \low as an induced subgraph. Note that $K_4$ is a \low.

We can assume that $G$ is a nice triangulation (see~\eqref{eq:nice_suffices}). 
Further, as $G$ is imperfect, $G$ contains an odd hole or an odd anti-hole (Theorem~\ref{thm:SPGT}).
The Euler characteristic of the projective plane implies that $2|E(H)|\leq 6|V(H)| -6$ for every embeddable graph $H$. Thus, no odd anti-hole on nine or more vertices is embeddable.
As the anti-hole $\overline{C_5}$ is isomorphic to the hole $ C_5$, 
an imperfect graph in the projective plane contains an odd hole or the anti-hole $\overline{C_7}$.
If $G$ contains an anti-hole $\overline{C_7}$, $G$ is $t$-imperfect by~\eqref{eq:7-antihole_t-imp}.
Thus, suppose that $G$ contains an induced odd hole $C$.
If the odd cycle $C$ is contractible, $G$ contains a \low by~Theorem~\ref{thm:cycles_triang}. 
As $G$ is a nice triangulation,  no deletion of an octahedron can be applied.

If no even-contraction can be applied to $G$, then the graph is irreducible and the claim follows from Lemma~\ref{lem:irreducible}.
Assume that an even-contraction can be applied. 
If $G'$ is perfect, $G$ contains a \low by Lemma~\ref{lem:even-contraction}. 
If the obtained graph $G'$ contains an induced $\overline{C_7}$, then by Lemma~\ref{lem:even-split_anti-hole}, $G$ contains a loose odd wheel or an induced $\overline{C_7}$. 
If $G'$ contains an odd hole then by induction, $G'$ contains a loose odd wheel. Thus, $G$ also contains a \low; see Lemma~\ref{lem:even-split_low}. 
\end{proof}

%% file: sec-irred_triang.tex
In this section, we analyse irreducible Eulerian triangulations of the projective plane, ie Eulerian triangulations to which no deletion of an octahedron and no even-contraction can be applied. 

Following Suzuki and Watanabe~\cite{SuWa07} we define three families of graphs.

\input{fig-I_16}
The graph $I_{16}[s_1, \ldots, s_n]$ with $s_i \in \{1,2,3\}$ for $i \in [n]$ is obtained from the graph $D$ by inserting the \emph{pieces} $h_1$, $h_2$ and $h_3$ (see Figure~\ref{fig:Def_I_16}) into the hexagonal region of $D$ as follows: Insert $h_{s_1},\ldots, h_{s_n}$ one below the other into the hexagon $e_1bce_2ad$ (with $e_1=e_2$). Then, identify the paths between $e_1$ and $e_2$ in each pair of consecutive pieces. 
Further, identify the path $e_1bce_2$ in $D$ with the path connecting $e_1$ and $e_2$ in $h_{s_1}$ that has not been connected to another piece. Analogously, identify $e_1,d,a,e_2$ in $D$ with the path in $h_{s_n}$ that has not been connected to another piece. 
Figure~\ref{fig:I_16_18_19_def} shows $I_{16}[1,2,3]$ as an example where the identified paths are dotted. The graph $I_1$ in Figure~\ref{fig:irred_all} is isomorphic to $I_{16} [1]$. 
\input{fig-I_16_18_19_def}

The graph $I_{18}[n]$, with $n \in \N$, is obtained from the graph $E$ in Figure~\ref{fig:Def_I_16} by inserting the pieces $h_2$ and $h_3$ of Figure~\ref{fig:Def_I_16} one below the other a total of $n$ times alternatingly into the hexagonal region $a_1bc_1a_2dc_2$ (with $a_1=a_2$ and $c_1=c_2$). The paths between $e_1=c_1$ and $e_2=c_2$ in each pair of consecutive pieces are identified. Figure~\ref{fig:I_16_18_19_def} shows $I_{18}[1]$ as an example.

The graph $I_{19}[m]$ for $m \in \N$ is obtained from $E$ (see Figure~\ref{fig:Def_I_16}) by inserting $m$ copies of $h_2$ and $m$ copies of $h_3$ alternatingly --- starting with $h_2$ --- and identifying  $c_1$ and $e_1$ as well as $c_2$ and $e_2$. 
Each pair of consecutive pieces bounds a hexagonal region. This region may be triangulated. Note that this happens in a unique way as all vertices are required to be of even degree.  Figure~\ref{fig:I_16_18_19_def} shows $I_{19}[1]$ as an example.

\input{fig-irred_all}

\begin{theorem} [Suzuki and Watanabe~\cite{SuWa07}]
An Eulerian triangulation $G$ of the projective plane is irreducible if and only if $G$ is one of the graphs in Figure~\ref{fig:irred_all} or belongs to one of the families $I_{16}[s_1, \ldots, s_n], I_{18}[n]$, and $I_{19}[m]$.
\end{theorem}

This characterisation  enables us to prove Lemma~\ref{lem:irreducible}.

\begin{proof}[Proof of Lemma~\ref{lem:irreducible}]

The graphs $I_5,I_8,I_{11},I_{13}, I_{15}$ and $I_{20}$ contain no \low (and thus no $K_4$) and no $\overline{C_7}$. One can check that these graphs are perfect.

Figure~\ref{fig:irred_all} shows that $I_1,I_2,I_3,I_4,I_6, I_7,I_9, I_{10} I_{12},I_{14}$ and $I_{17}$ have a loose odd wheel as subgraph. 

Now consider $I_{16}[s_1, \ldots, s_n]$.
First, suppose that the graph contains an even number of copies of $h_1$. The graphs $D$, $h_1$, $h_2$ and $h_3$ are easily seen to be $3$-colourable. The colourings can be combined to obtain a $3$-colouring of the vertices of $I_{16}[s_1, \ldots, s_n]$. Thus, the graphs contains no $K_4$. One can check easily that every triangle-free subgraph of $G$ can be coloured with  two colours.
 Therefore, the clique number and the chromatic number coincide for each subgraph of $I_{16}[s_1, \ldots, s_n]$. This shows that the graph is perfect.
 
Second, let  $I_{16}[s_1, \ldots, s_n]$ contain an odd number of copies of $h_1$.
Consider the path $P$ from $b$ to $d$ consisting of the dotted edges in each copy of $h_1$ and $h_2$. This path is induced and odd and forms a \low together with $v$ and~$a$. 

Next, consider $I_{18}[n]$ for $n \in \N$.
Each consecutive pair of pieces has exactly one edge  in common that is not adjacent to $c_1=c_2$. The union of these edges forms an induced odd path between $b$ and $d$. (For an illustration 
see Figure~\ref{fig:lem_I_18_19}.) This path together with $x$ and $y$ gives a loose odd wheel.

Last, consider $I_{19}[m]$ for $m \in \N$.
An induced induced odd path from $b$ to $d$ together with the vertices $x$ and $y$ forms a loose odd wheel in $I_{19}[m]$. 
Such a path is eg dotted in Figure~\ref{fig:lem_I_18_19} for $I_{19}[2]$. For the general case take for the first pair of $h_1$ and $h_2$ the dotted path from $v_1$ to $v_2$ which is odd. For all other pairs take the dotted path from $v_2$ to $v_3$ which is even. Sticking these paths together gives an odd induced path from $b$ to $d$.  
\end{proof}

\input{fig-lem_I_18_I_19}

%% file: fig-I_16.tex
\begin{figure}[htb]
\begin{center}
\begin{tikzpicture}[scale = .17]

\begin{scope}[scale = .9]
\node[novertex] (x0) at (0,-11){$\mathbf{D}$};

\node[hvertex] (x1) at (0,-7){};
\node[novertex] (lx1) at (0,8.5){$v$};
\node[hvertex] (x2) at (0,-4){};
\node[novertex] (lx2) at (0,-2.6){$d$};
\node[hvertex] (x3) at (0,4){};
\node[novertex] (lx3) at (0,2.6){$b$};
\node[hvertex] (x4) at (0,7){};
\node[novertex] (lx4) at (0,-8.5){$v$};

\node[hvertex] (x5) at (-8,-4){};
\node[novertex] (lx5) at (-9.4,-4){$c$};
\node[hvertex] (x6) at (-8,4){};
\node[novertex] (lx5) at (-9.4,4){$a$};

\node[hvertex] (x7) at (8,-4){};
\node[novertex] (lx5) at (9.4,-4){$a$};
\node[hvertex] (x8) at (8,4){};
\node[novertex] (lx5) at (9.4,4){$c$};

\node[hvertex] (x11) at (-8,0){};
\node[novertex] (lx11) at (-9.4,0){$e_1$};

\node[hvertex] (x13) at (8,0){};
\node[novertex] (lx13) at (9.4,0){$e_2$};

\draw[hedge] (x1) -- (x2);
\draw[hedge] (x3) -- (x4);

\draw[hedge] (x5) -- (x11);
\draw[hedge] (x11) -- (x6);

\draw[hedge] (x7) -- (x13);
\draw[hedge] (x13) -- (x8);

\draw[hedge] (x5) -- (x1);
\draw[hedge] (x1) -- (x7);

\draw[hedge] (x5) -- (x2);
\draw[hedge] (x2) -- (x7);

\draw[hedge] (x6) -- (x3);
\draw[hedge] (x3) -- (x8);

\draw[hedge] (x6) -- (x4);
\draw[hedge] (x4) -- (x8);

\draw[hedge] (x11) -- (x2);
\draw[hedge] (x11) -- (x3);
\end{scope}

\begin{scope} [shift={(23,0)}, scale=0.9]
\node[novertex] (x0) at (0,-11){$\mathbf{E}$};

\node[hvertex] (x1) at (-2,-7){};
\node[novertex] (lx3) at (-2,-8.5){$y$};
\node[hvertex] (x2) at (2,-7){};
\node[novertex] (lx3) at (2,-8.5){$x$};

\node[hvertex] (x3) at (-2,7){};
\node[novertex] (lx3) at (-2,8){$x$};
\node[hvertex] (x4) at (2,7){};
\node[novertex] (lx4) at (2,8){$y$};

\node[hvertex] (x5) at (-8,-4){};
\node[novertex] (lx5) at (-9.5,-4){$c_1$};
\node[hvertex] (x6) at (-8,4){};
\node[novertex] (lx6) at (-9.5,4){$a_1$};
\node[hvertex] (x7) at (8,-4){};
\node[novertex] (lx7) at (9.5,-4){$a_2$};
\node[hvertex] (x8) at (8,4){};
\node[novertex] (lx8) at (9.5,4){$c_2$};

\node[hvertex] (x9) at (0,-4){};
\node[novertex] (lx9) at (0,-2.5){$d$};

\node[hvertex] (x10) at (0,4){};
\node[novertex] (lx10) at (0,2.5){$b$};

\draw[hedge] (x1) -- (x2);
\draw[hedge] (x2) -- (x7);
\draw[hedge] (x7) -- (x8);
\draw[hedge] (x8) -- (x4);
\draw[hedge] (x4) -- (x3);
\draw[hedge] (x3) -- (x6);
\draw[hedge] (x6) -- (x5);
\draw[hedge] (x5) -- (x1);

\draw[hedge] (x1) -- (x9);
\draw[hedge] (x2) -- (x9);
\draw[hedge] (x5) -- (x9);
\draw[hedge] (x7) -- (x9);

\draw[hedge] (x3) -- (x10);
\draw[hedge] (x4) -- (x10);
\draw[hedge] (x6) -- (x10);
\draw[hedge] (x8) -- (x10);

\end{scope}

\begin{scope}[shift={(46,7)}]
\node[novertex] (a0) at (5,-3){$\mathbf{h_1}$};
\node[hvertex] (a1) at (-9,0){};
\node[novertex] (la1) at (-10.5,0){$e_1$};
\node[hvertex] (a2) at (-4.5,1.5){};
\node[hvertex] (a3) at (0,1.5){};
\node[hvertex] (a4) at (4.5,0){};
\node[novertex] (la4) at (6,0){$e_2$};
\node[hvertex] (a5) at (-4.5,-1.5){};
\node[hvertex] (a6) at (0,-1.5){};

\node[hvertex] (a7) at (-2.25,0){};

\draw[hedge] (a1) -- (a2);
\draw[hedge] (a1) -- (a5);
\draw[dedge] (a2) -- (a5);
\draw[hedge] (a2) -- (a3);
\draw[hedge] (a5) -- (a6);
\draw[hedge] (a3) -- (a6);
\draw[hedge] (a4) -- (a6);
\draw[hedge] (a4) -- (a3);
\draw[hedge] (a2) -- (a7);
\draw[hedge] (a3) -- (a7);
\draw[hedge] (a5) -- (a7);
\draw[hedge] (a6) -- (a7);

\end{scope}

\begin{scope}[shift={(46,0)}]
\node[novertex] (y0) at (5,-3){$\mathbf{h_2}$};
\node[hvertex] (y1) at (0,0){};
\node[hvertex] (y2) at (4.5,0){};
\node[novertex] (ly2) at (6,0){$e_2$};
\node[hvertex] (y3) at (-3,0){};
\node[hvertex] (y4) at (-6,0){};
\node[hvertex] (y5) at (-9,0){};
\node[novertex] (ly5) at (-10.5,0){$e_1$};

\node[hvertex] (y6) at (-4.5,2){};
\node[hvertex] (y7) at (-4.5,-2){};

\draw[hedge] (y1) -- (y2);
\draw[hedge] (y1) -- (y3);
\draw[hedge] (y3) -- (y4);
\draw[hedge] (y4) -- (y5);

\draw[hedge] (y5) -- (y6);
\draw[hedge] (y5) -- (y7);
\draw[dedge] (y4) -- (y6);
\draw[dedge] (y4) -- (y7);
\draw[hedge] (y3) -- (y6);
\draw[hedge] (y3) -- (y7);
\draw[hedge] (y1) -- (y6);
\draw[hedge] (y1) -- (y7);

\end{scope}

\begin{scope}[shift={(46,-7)}]
\node[novertex] (z0) at (5,-3){$\mathbf{h_3}$};
\node[hvertex] (z1) at (-9,0){};
\node[novertex] (lz1) at (-10.5,0){$e_1$};
\node[hvertex] (z2) at (-4.5,0){};
\node[hvertex] (z3) at (-1.5,0){};
\node[hvertex] (z4) at (1.5,0){};
\node[hvertex] (z5) at (4.5,0){};
\node[novertex] (lz5) at (6,0){$e_2$};

\node[hvertex] (z6) at (0,2){};
\node[hvertex] (z7) at (0,-2){};

\draw[hedge] (z1) -- (z2);
\draw[hedge] (z2) -- (z3);
\draw[hedge] (z3) -- (z4);
\draw[hedge] (z4) -- (z5);

\draw[hedge] (z5) -- (z6);
\draw[hedge] (z5) -- (z7);
\draw[hedge] (z4) -- (z6);
\draw[hedge] (z4) -- (z7);
\draw[hedge] (z3) -- (z6);
\draw[hedge] (z3) -- (z7);
\draw[hedge] (z2) -- (z6);
\draw[hedge] (z2) -- (z7);

\end{scope}

\end{tikzpicture}

\end{center}

\caption{The pieces $D$, $E$, $h_1$, $h_2$ and $h_3$}
\label{fig:Def_I_16}
\end{figure}
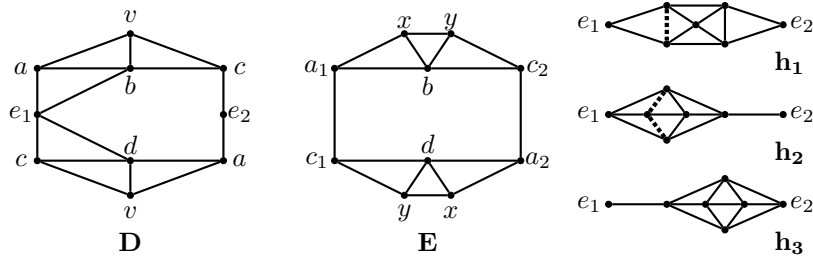

%% file: fig-I_16_18_19_def.tex
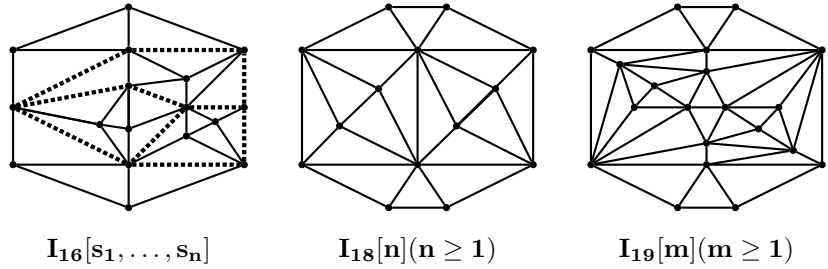
\begin{figure}[htb]
\begin{center}
\begin{tikzpicture}[scale = .19]

\begin{scope}[shift={(0,0)}]

\node[novertex] (x0) at (0,-10){$\mathbf{I_{16}[s_1, \ldots, s_n]}$};

\node[hvertex] (x1) at (0,-7){};
\node[hvertex] (x2) at (0,-4){};
\node[hvertex] (x3) at (0,4){};
\node[hvertex] (x4) at (0,7){};

\node[hvertex] (x5) at (-8,-4){};
\node[hvertex] (x6) at (-8,4){};

\node[hvertex] (x7) at (8,-4){};
\node[hvertex] (x8) at (8,4){};

\node[hvertex] (x9) at (-2,-1.2){};

\node[hvertex] (x11) at (-8,0){};

\node[hvertex] (x10) at (0,-1.5){};
\node[hvertex] (x12) at (0,1.5){};
\node[hvertex] (x13) at (8,0){};

\node[hvertex] (x14) at (4,-2){};
\node[hvertex] (x15) at (4,0){};
\node[hvertex] (x16) at (4,2){};

\node[hvertex] (x17) at (6,-1){};

\draw[hedge] (x1) -- (x2);
\draw[hedge] (x2) -- (x12);
\draw[hedge] (x10) -- (x12);
\draw[hedge] (x12) -- (x3);
\draw[hedge] (x3) -- (x4);

\draw[hedge] (x5) -- (x11);
\draw[hedge] (x11) -- (x6);

\draw[dedge] (x7) -- (x13);
\draw[dedge] (x13) -- (x8);

\draw[hedge] (x5) -- (x1);
\draw[hedge] (x1) -- (x7);

\draw[hedge] (x5) -- (x2);
\draw[dedge] (x2) -- (x7);

\draw[hedge] (x6) -- (x3);
\draw[dedge] (x3) -- (x8);

\draw[hedge] (x6) -- (x4);
\draw[hedge] (x4) -- (x8);

\draw[hedge] (x9) -- (x2);
\draw[hedge] (x9) -- (x10);
\draw[hedge] (x9) -- (x12);
\draw[hedge] (x9) -- (x11);

\draw[dedge] (x11) -- (x2);
\draw[hedge] (x11) -- (x9);
\draw[dedge] (x11) -- (x12);
\draw[dedge] (x11) -- (x3);

\draw[hedge] (x14) -- (x2);
\draw[hedge] (x14) -- (x7);
\draw[hedge] (x14) -- (x17);
\draw[hedge] (x14) -- (x15);

\draw[dedge] (x15) -- (x2);
\draw[hedge] (x15) -- (x10);
\draw[dedge] (x15) -- (x12);
\draw[hedge] (x15) -- (x16);
\draw[hedge] (x15) -- (x8);
\draw[dedge] (x15) -- (x13);
\draw[hedge] (x15) -- (x17);

\draw[hedge] (x16) -- (x8);
\draw[hedge] (x16) -- (x3);
\draw[hedge] (x16) -- (x12);

\draw[hedge] (x17) -- (x7);
\draw[hedge] (x17) -- (x13);

\end{scope}

\begin{scope}[shift={(20,0)}]

\node[novertex] (x0) at (0,-10){$\mathbf{I_{18}[n] (n \geq 1)}$};

\node[hvertex] (x1) at (-2,-7){};
\node[hvertex] (x2) at (2,-7){};

\node[hvertex] (x3) at (-2,7){};
\node[hvertex] (x4) at (2,7){};

\node[hvertex] (x5) at (-8,-4){};
\node[hvertex] (x6) at (-8,4){};

\node[hvertex] (x7) at (8,-4){};
\node[hvertex] (x8) at (8,4){};

\node[hvertex] (x9) at (0,-4){};

\node[hvertex] (x10) at (0,4){};

\node[hvertex] (x11) at (-5.4,-1.3){};
\node[hvertex] (x12) at (-2.7,1.3){};

\node[hvertex] (x14) at (2.7,-1.3){};
\node[hvertex] (x13) at (5.4,1.3){};

\draw[hedge] (x1) -- (x2);
\draw[hedge] (x2) -- (x7);
\draw[hedge] (x7) -- (x8);
\draw[hedge] (x8) -- (x4);
\draw[hedge] (x4) -- (x3);
\draw[hedge] (x3) -- (x6);
\draw[hedge] (x6) -- (x5);
\draw[hedge] (x5) -- (x1);

\draw[hedge] (x1) -- (x9);
\draw[hedge] (x2) -- (x9);
\draw[hedge] (x5) -- (x9);
\draw[hedge] (x7) -- (x9);
\draw[hedge] (x9) -- (x10);

\draw[hedge] (x3) -- (x10);
\draw[hedge] (x4) -- (x10);
\draw[hedge] (x6) -- (x10);
\draw[hedge] (x8) -- (x10);

\draw[hedge] (x11) -- (x12);
\draw[hedge] (x11) -- (x5);
\draw[hedge] (x11) -- (x9);
\draw[hedge] (x11) -- (x6);

\draw[hedge] (x12) -- (x10);
\draw[hedge] (x12) -- (x9);
\draw[hedge] (x12) -- (x6);

\draw[hedge] (x13) -- (x14);
\draw[hedge] (x13) -- (x7);
\draw[hedge] (x13) -- (x9);
\draw[hedge] (x13) -- (x10);

\draw[hedge] (x14) -- (x7);
\draw[hedge] (x14) -- (x8);
\draw[hedge] (x14) -- (x10);

\end{scope}

\begin{scope}[shift={(40,0)}]

\node[novertex] (x0) at (0,-10){$\mathbf{I_{19}[m](m \geq 1)}$};

\node[hvertex] (x1) at (-2,-7){};
\node[hvertex] (x2) at (2,-7){};

\node[hvertex] (x3) at (-2,7){};
\node[hvertex] (x4) at (2,7){};

\node[hvertex] (x5) at (-8,-4){};
\node[hvertex] (x6) at (-8,4){};

\node[hvertex] (x7) at (8,-4){};
\node[hvertex] (x8) at (8,4){};

\node[hvertex] (x9) at (0,-4){};
\node[hvertex] (x10) at (0,4){};

\node[hvertex] (x11) at (0,-2.5){};
\node[hvertex] (x12) at (0,2.5){};

\node[hvertex] (x13) at (-6,3){};
\node[hvertex] (x14) at (6,-3){};

\node[hvertex] (x15) at (-1.3,0){};
\node[hvertex] (x16) at (1.3,0){};

\node[hvertex] (x17) at (-3.6,1.5){};
\node[hvertex] (x18) at (3.6,-1.5){};

\node[hvertex] (x19) at (-5,0){};
\node[hvertex] (x20) at (5,0){};

\draw[hedge] (x1) -- (x2);
\draw[hedge] (x2) -- (x7);
\draw[hedge] (x7) -- (x8);
\draw[hedge] (x8) -- (x4);
\draw[hedge] (x4) -- (x3);
\draw[hedge] (x3) -- (x6);
\draw[hedge] (x6) -- (x5);
\draw[hedge] (x5) -- (x1);

\draw[hedge] (x1) -- (x9);
\draw[hedge] (x2) -- (x9);
\draw[hedge] (x5) -- (x9);
\draw[hedge] (x7) -- (x9);

\draw[hedge] (x3) -- (x10);
\draw[hedge] (x4) -- (x10);
\draw[hedge] (x6) -- (x10);
\draw[hedge] (x8) -- (x10);

\draw[hedge] (x9) -- (x11);
\draw[hedge] (x10) -- (x12);

\draw[hedge] (x13) -- (x6);
\draw[hedge] (x13) -- (x10);
\draw[hedge] (x13) -- (x5);
\draw[hedge] (x13) -- (x12);

\draw[hedge] (x8) -- (x12);
\draw[hedge] (x5) -- (x11);

\draw[hedge] (x14) -- (x9);
\draw[hedge] (x14) -- (x11);
\draw[hedge] (x14) -- (x7);
\draw[hedge] (x14) -- (x8);

\draw[hedge] (x12) -- (x15);
\draw[hedge] (x15) -- (x11);
\draw[hedge] (x11) -- (x16);
\draw[hedge] (x16) -- (x12);
\draw[hedge] (x16) -- (x15);

\draw[hedge] (x15) -- (x5);
\draw[hedge] (x16) -- (x8);

\draw[hedge] (x13) -- (x17);
\draw[hedge] (x17) -- (x15);

\draw[hedge] (x14) -- (x18);
\draw[hedge] (x18) -- (x16);

\draw[hedge] (x11) -- (x18);
\draw[hedge] (x12) -- (x17);

\draw[hedge] (x19) -- (x5);
\draw[hedge] (x19) -- (x15);
\draw[hedge] (x19) -- (x17);
\draw[hedge] (x19) -- (x13);

\draw[hedge] (x20) -- (x18);
\draw[hedge] (x20) -- (x14);
\draw[hedge] (x20) -- (x8);
\draw[hedge] (x20) -- (x16);

\end{scope}
\end{tikzpicture}
\end{center}
\caption{The three infinite families of irreducible graphs. Opposite points on the outer cycles are identified.} 

\label{fig:I_16_18_19_def}
\end{figure} 

%% file: fig-irred_all.tex
\begin{figure}[!t]
\begin{center}
\begin{tikzpicture}[scale = .16]

\begin{scope}[shift={(0,0)}]

\node[novertex] (x0) at (6,-7){$\mathbf{I_1}$};

\node[hvertex] (x1) at (0,-7){};
\node[hvertex] (x2) at (0,-4){};
\node[hvertex] (x3) at (0,4){};
\node[hvertex] (x4) at (0,7){};

\node[hvertex] (x5) at (-8,-4){};
\node[hvertex] (x6) at (-8,4){};

\node[hvertex] (x7) at (8,-4){};
\node[hvertex] (x8) at (8,4){};

\node[hvertex] (x9) at (-4,0){};
\node[hvertex] (x10) at (4,0){};

\draw[hedge] (x1) -- (x2);
\draw[hedge] (x2) -- (x3);
\draw[hedge] (x3) -- (x4);

\draw[dedge] (x5) -- (x6);
\draw[dedge] (x7) -- (x8);
\draw[hedge] (x7) -- (x1);

\draw[hedge] (x5) -- (x1);

\draw[hedge] (x5) -- (x2);
\draw[hedge] (x2) -- (x7);

\draw[dedge] (x6) -- (x3);
\draw[dedge] (x3) -- (x8);

\draw[hedge] (x6) -- (x4);
\draw[hedge] (x4) -- (x8);

\draw[dedge] (x6) -- (x9);
\draw[dedge] (x3) -- (x9);
\draw[dedge] (x5) -- (x9);
\draw[hedge] (x2) -- (x9);

\draw[hedge] (x8) -- (x10);
\draw[hedge] (x3) -- (x10);
\draw[hedge] (x7) -- (x10);
\draw[hedge] (x2) -- (x10);

\end{scope}
\begin{scope}[shift={(20,0)}]

\node[novertex] (x0) at (6,-7){$\mathbf{I_2}$};

\node[hvertex] (x1) at (0,-7){};
\node[hvertex] (x2) at (0,-4){};
\node[hvertex] (x3) at (0,4){};
\node[hvertex] (x4) at (0,7){};

\node[hvertex] (x5) at (-8,-4){};
\node[hvertex] (x6) at (-8,4){};

\node[hvertex] (x7) at (8,-4){};
\node[hvertex] (x8) at (8,4){};

\node[hvertex] (x9) at (-4,0){};

\node[hvertex] (x10) at (0,-1.3){};
\node[hvertex] (x11) at (0,1.3){};

\node[hvertex] (x12) at (5.4,-1.3){};

\node[hvertex] (x13) at (2.6,2){};

\draw[hedge] (x1) -- (x2);
\draw[hedge] (x2) -- (x10);
\draw[hedge] (x10) -- (x11);
\draw[hedge] (x11) -- (x3);
\draw[hedge] (x3) -- (x4);

\draw[dedge] (x5) -- (x6);
\draw[dedge] (x7) -- (x8);

\draw[hedge] (x5) -- (x1);
\draw[hedge] (x1) -- (x7);

\draw[hedge] (x5) -- (x2);
\draw[hedge] (x2) -- (x7);

\draw[dedge] (x6) -- (x3);
\draw[dedge] (x3) -- (x8);

\draw[hedge] (x6) -- (x4);
\draw[hedge] (x4) -- (x8);

\draw[hedge] (x10) -- (x8);

\draw[dedge] (x6) -- (x9);
\draw[dedge] (x3) -- (x9);
\draw[dedge] (x5) -- (x9);
\draw[hedge] (x2) -- (x9);
\draw[hedge] (x10) -- (x9);
\draw[hedge] (x11) -- (x9);

\draw[hedge] (x2) -- (x12);
\draw[hedge] (x7) -- (x12);
\draw[hedge] (x8) -- (x12);
\draw[hedge] (x10) -- (x12);

\draw[hedge] (x10) -- (x13);
\draw[hedge] (x11) -- (x13);
\draw[hedge] (x8) -- (x13);
\draw[hedge] (x3) -- (x13);

\end{scope}
\begin{scope}[shift={(40,0)}]

\node[novertex] (x0) at (6,-7){$\mathbf{I_3}$};

\node[hvertex] (x1) at (0,-7){};
\node[hvertex] (x2) at (0,-4){};
\node[hvertex] (x3) at (0,4){};
\node[hvertex] (x4) at (0,7){};

\node[hvertex] (x5) at (-8,-4){};
\node[hvertex] (x6) at (-8,4){};

\node[hvertex] (x7) at (8,-4){};
\node[hvertex] (x8) at (8,4){};

\node[hvertex] (x9) at (-4,0){};

\node[hvertex] (x10) at (.1,-2.1){};
\node[hvertex] (x11) at (1.6,-.9){};
\node[hvertex] (x12) at (1.6,.9){};
\node[hvertex] (x13) at (-.1,2.1){};

\node[hvertex] (x14) at (3,-3){};
\node[hvertex] (x15) at (5,0){};
\node[hvertex] (x16) at (3,3){};

\draw[hedge] (x1) -- (x2);
\draw[hedge] (x2) -- (x10);
\draw[hedge] (x10) -- (x11);
\draw[hedge] (x11) -- (x12);
\draw[hedge] (x12) -- (x13);
\draw[hedge] (x13) -- (x3);
\draw[hedge] (x3) -- (x4);

\draw[dedge] (x5) -- (x6);
\draw[dedge] (x7) -- (x8);

\draw[hedge] (x5) -- (x1);
\draw[hedge] (x1) -- (x7);

\draw[hedge] (x5) -- (x2);
\draw[hedge] (x2) -- (x7);

\draw[dedge] (x6) -- (x3);
\draw[dedge] (x3) -- (x8);

\draw[hedge] (x6) -- (x4);
\draw[hedge] (x4) -- (x8);

\draw[dedge] (x6) -- (x9);
\draw[dedge] (x3) -- (x9);
\draw[dedge] (x5) -- (x9);
\draw[hedge] (x2) -- (x9);
\draw[hedge] (x10) -- (x9);
\draw[hedge] (x11) -- (x9);
\draw[hedge] (x12) -- (x9);
\draw[hedge] (x13) -- (x9);

\draw[hedge] (x2) -- (x14);
\draw[hedge] (x10) -- (x14);
\draw[hedge] (x11) -- (x14);

\draw[hedge] (x7) -- (x14);

\draw[hedge] (x3) -- (x16);
\draw[hedge] (x12) -- (x16);
\draw[hedge] (x13) -- (x16);

\draw[hedge] (x8) -- (x16);

\draw[hedge] (x12) -- (x15);
\draw[hedge] (x11) -- (x15);

\draw[hedge] (x7) -- (x15);
\draw[hedge] (x8) -- (x15);

\draw[hedge] (x8) -- (x12);

\draw[hedge] (x7) -- (x11);

\end{scope}
\begin{scope}[shift={(0,-20)}]

\node[novertex] (x0) at (6,-7){$\mathbf{I_4}$};

\node[hvertex] (x1) at (0,-7){};
\node[hvertex] (x2) at (0,-4){};
\node[hvertex] (x3) at (0,4){};
\node[hvertex] (x4) at (0,7){};

\node[hvertex] (x5) at (-8,-4){};
\node[hvertex] (x6) at (-8,4){};

\node[hvertex] (x7) at (8,-4){};
\node[hvertex] (x8) at (8,4){};

\node[hvertex] (x9) at (0,0){};

\node[hvertex] (x10) at (-3.4,-1){};
\node[hvertex] (x11) at (-5.6,1.2){};

\node[hvertex] (x12) at (3,-2){};
\node[hvertex] (x13) at (5.6,-1){};
\node[hvertex] (x14) at (4,2){};

\draw[hedge] (x1) -- (x2);
\draw[dedge] (x2) -- (x9);
\draw[hedge] (x9) -- (x3);
\draw[hedge] (x3) -- (x4);

\draw[hedge] (x5) -- (x6);
\draw[hedge] (x7) -- (x8);

\draw[hedge] (x5) -- (x1);
\draw[hedge] (x1) -- (x7);

\draw[hedge] (x5) -- (x2);
\draw[dedge] (x2) -- (x7);

\draw[hedge] (x6) -- (x3);
\draw[hedge] (x3) -- (x8);

\draw[hedge] (x6) -- (x4);
\draw[hedge] (x4) -- (x8);

\draw[dedge] (x6) -- (x9);

\draw[hedge] (x10) -- (x11);

\draw[hedge] (x10) -- (x5);
\draw[hedge] (x10) -- (x2);
\draw[hedge] (x10) -- (x9);

\draw[hedge] (x11) -- (x5);
\draw[hedge] (x11) -- (x6);
\draw[hedge] (x11) -- (x9);

\draw[hedge] (x12) -- (x13);
\draw[hedge] (x13) -- (x14);
\draw[hedge] (x14) -- (x12);

\draw[dedge] (x12) -- (x7);
\draw[dedge] (x12) -- (x2);
\draw[dedge] (x12) -- (x9);
\draw[hedge] (x12) -- (x3);

\draw[hedge] (x13) -- (x7);
\draw[hedge] (x13) -- (x8);

\draw[hedge] (x14) -- (x8);
\draw[hedge] (x14) -- (x3);

\end{scope}
\begin{scope}[shift={(20,-20)}]

\node[novertex] (x0) at (6,-7){$\mathbf{I_5}$};

\node[hvertex] (x1) at (0,-7){};
\node[hvertex] (x2) at (0,-4){};
\node[hvertex] (x3) at (0,4){};
\node[hvertex] (x4) at (0,7){};

\node[hvertex] (x5) at (-8,-4){};
\node[hvertex] (x6) at (-8,4){};

\node[hvertex] (x7) at (8,-4){};
\node[hvertex] (x8) at (8,4){};

\node[hvertex] (x9) at (-4,0){};
\node[hvertex] (x10) at (4,0){};

\node[hvertex] (x11) at (-8,0){};
\node[hvertex] (x12) at (0,0){};
\node[hvertex] (x13) at (8,0){};

\draw[hedge] (x1) -- (x2);
\draw[hedge] (x2) -- (x12);
\draw[hedge] (x12) -- (x3);
\draw[hedge] (x3) -- (x4);

\draw[hedge] (x5) -- (x11);
\draw[hedge] (x11) -- (x6);

\draw[hedge] (x7) -- (x13);
\draw[hedge] (x13) -- (x8);

\draw[hedge] (x5) -- (x1);
\draw[hedge] (x1) -- (x7);

\draw[hedge] (x5) -- (x2);
\draw[hedge] (x2) -- (x7);

\draw[hedge] (x6) -- (x3);
\draw[hedge] (x3) -- (x8);

\draw[hedge] (x6) -- (x4);
\draw[hedge] (x4) -- (x8);

\draw[hedge] (x6) -- (x9);
\draw[hedge] (x3) -- (x9);
\draw[hedge] (x5) -- (x9);
\draw[hedge] (x2) -- (x9);

\draw[hedge] (x8) -- (x10);
\draw[hedge] (x3) -- (x10);
\draw[hedge] (x7) -- (x10);
\draw[hedge] (x2) -- (x10);

\draw[hedge] (x11) -- (x9);
\draw[hedge] (x9) -- (x12);
\draw[hedge] (x12) -- (x10);
\draw[hedge] (x10) -- (x13);
\end{scope}
\begin{scope}[shift={(40,-20)}]

\node[novertex] (x0) at (6,-7){$\mathbf{I_6}$};

\node[hvertex] (x1) at (0,-7){};
\node[hvertex] (x2) at (0,-4){};
\node[hvertex] (x3) at (0,4){};
\node[hvertex] (x4) at (0,7){};

\node[hvertex] (x5) at (-8,-4){};
\node[hvertex] (x6) at (-8,4){};

\node[hvertex] (x7) at (8,-4){};
\node[hvertex] (x8) at (8,4){};

\node[hvertex] (x9) at (0,0){};
\node[hvertex] (x10) at (-8,0){};
\node[hvertex] (x11) at (8,0){};

\node[hvertex] (x12) at (-5.4,0){};
\node[hvertex] (x13) at (-2.6,0){};

\node[hvertex] (x14) at (2.6,0){};
\node[hvertex] (x15) at (6,0){};

\node[hvertex] (x16) at (7,2){};

\draw[hedge] (x1) -- (x2);
\draw[hedge] (x2) -- (x9);
\draw[hedge] (x9) -- (x3);
\draw[hedge] (x3) -- (x4);

\draw[hedge] (x5) -- (x10);
\draw[dedge] (x10) -- (x6);

\draw[dedge] (x7) -- (x11);
\draw[hedge] (x11) -- (x8);

\draw[hedge] (x5) -- (x1);
\draw[hedge] (x1) -- (x7);

\draw[hedge] (x5) -- (x2);
\draw[hedge] (x2) -- (x7);

\draw[dedge] (x6) -- (x3);
\draw[hedge] (x3) -- (x8);

\draw[hedge] (x6) -- (x4);
\draw[hedge] (x4) -- (x8);

\draw[hedge] (x10) -- (x12);
\draw[hedge] (x12) -- (x13);
\draw[hedge] (x13) -- (x9);
\draw[hedge] (x9) -- (x14);
\draw[dedge] (x14) -- (x15);
\draw[dedge] (x15) -- (x11);

\draw[hedge] (x11) -- (x2);

\draw[hedge] (x12) -- (x6);
\draw[hedge] (x12) -- (x5);
\draw[hedge] (x12) -- (x13);

\draw[hedge] (x13) -- (x6);
\draw[hedge] (x13) -- (x5);
\draw[hedge] (x13) -- (x3);
\draw[hedge] (x13) -- (x2);

\draw[dedge] (x14) -- (x3);
\draw[hedge] (x14) -- (x2);
\draw[hedge] (x14) -- (x8);
\draw[dedge] (x14) -- (x16);

\draw[hedge] (x15) -- (x2);
\draw[dedge] (x15) -- (x16);
\draw[dedge] (x15) -- (x11);

\draw[hedge] (x16) -- (x8);
\draw[dedge] (x16) -- (x11);
\end{scope}
\begin{scope}[shift={(0,-40)}]

\node[novertex] (x0) at (6,-7){$\mathbf{I_7}$};

\node[hvertex] (x1) at (0,-7){};
\node[hvertex] (x2) at (0,-4){};
\node[hvertex] (x3) at (0,4){};
\node[hvertex] (x4) at (0,7){};

\node[hvertex] (x5) at (-8,-4){};
\node[hvertex] (x6) at (-8,4){};

\node[hvertex] (x7) at (8,-4){};
\node[hvertex] (x8) at (8,4){};

\node[hvertex] (x9) at (-4,0){};
\node[hvertex] (x10) at (4,0){};

\node[hvertex] (x11) at (-8,0){};
\node[hvertex] (x12) at (0,0){};
\node[hvertex] (x13) at (8,0){};

\node[hvertex] (x14) at (-5.4,1.7){};
\node[hvertex] (x15) at (5.4,1.7){};

\draw[hedge] (x1) -- (x2);
\draw[hedge] (x2) -- (x12);
\draw[dedge] (x12) -- (x3);
\draw[hedge] (x3) -- (x4);

\draw[hedge] (x5) -- (x11);
\draw[dedge] (x11) -- (x6);

\draw[dedge] (x7) -- (x13);
\draw[hedge] (x13) -- (x8);

\draw[hedge] (x5) -- (x1);
\draw[hedge] (x1) -- (x7);

\draw[hedge] (x5) -- (x2);
\draw[hedge] (x2) -- (x7);

\draw[dedge] (x6) -- (x3);
\draw[hedge] (x3) -- (x8);

\draw[hedge] (x6) -- (x4);
\draw[hedge] (x4) -- (x8);

\draw[hedge] (x14) -- (x9);
\draw[dedge] (x14) -- (x6);
\draw[hedge] (x3) -- (x9);
\draw[hedge] (x5) -- (x9);
\draw[hedge] (x2) -- (x9);

\draw[hedge] (x15) -- (x10);
\draw[hedge] (x15) -- (x8);
\draw[hedge] (x3) -- (x10);
\draw[hedge] (x7) -- (x10);
\draw[hedge] (x2) -- (x10);

\draw[hedge] (x11) -- (x9);
\draw[hedge] (x9) -- (x12);
\draw[dedge] (x12) -- (x10);
\draw[dedge] (x10) -- (x13);

\draw[dedge] (x14) -- (x11);
\draw[dedge] (x14) -- (x3);

\draw[hedge] (x15) -- (x13);
\draw[hedge] (x15) -- (x3);
\end{scope}
\begin{scope}[shift={(20,-40)}]

\node[novertex] (x0) at (6,-7){$\mathbf{I_8}$};

\node[hvertex] (x1) at (0,-7){};
\node[hvertex] (x2) at (0,-4){};
\node[hvertex] (x3) at (0,4){};
\node[hvertex] (x4) at (0,7){};

\node[hvertex] (x5) at (-8,-4){};
\node[hvertex] (x6) at (-8,4){};

\node[hvertex] (x7) at (8,-4){};
\node[hvertex] (x8) at (8,4){};

\node[hvertex] (x9) at (-4,0){};
\node[hvertex] (x10) at (4,0){};

\node[hvertex] (x11) at (-8,0){};
\node[hvertex] (x12) at (0,0){};
\node[hvertex] (x13) at (8,0){};

\node[hvertex] (x14) at (-5.4,1.7){};
\node[hvertex] (x15) at (5.4,1.7){};

\node[hvertex] (x16) at (-5.4,-1.7){};
\node[hvertex] (x17) at (5.4,-1.7){};

\draw[hedge] (x1) -- (x2);
\draw[hedge] (x2) -- (x12);
\draw[hedge] (x12) -- (x3);
\draw[hedge] (x3) -- (x4);

\draw[hedge] (x5) -- (x11);
\draw[hedge] (x11) -- (x6);

\draw[hedge] (x7) -- (x13);
\draw[hedge] (x13) -- (x8);

\draw[hedge] (x5) -- (x1);
\draw[hedge] (x1) -- (x7);

\draw[hedge] (x5) -- (x2);
\draw[hedge] (x2) -- (x7);

\draw[hedge] (x6) -- (x3);
\draw[hedge] (x3) -- (x8);

\draw[hedge] (x6) -- (x4);
\draw[hedge] (x4) -- (x8);

\draw[hedge] (x14) -- (x9);
\draw[hedge] (x14) -- (x6);
\draw[hedge] (x3) -- (x9);
\draw[hedge] (x2) -- (x9);

\draw[hedge] (x16) -- (x9);
\draw[hedge] (x5) -- (x16);

\draw[hedge] (x15) -- (x10);
\draw[hedge] (x15) -- (x8);
\draw[hedge] (x3) -- (x10);
\draw[hedge] (x2) -- (x10);

\draw[hedge] (x17) -- (x10);
\draw[hedge] (x7) -- (x17);

\draw[hedge] (x11) -- (x9);
\draw[hedge] (x9) -- (x12);
\draw[hedge] (x12) -- (x10);
\draw[hedge] (x10) -- (x13);

\draw[hedge] (x14) -- (x11);
\draw[hedge] (x14) -- (x3);

\draw[hedge] (x15) -- (x13);
\draw[hedge] (x15) -- (x3);

\draw[hedge] (x16) -- (x11);
\draw[hedge] (x16) -- (x2);

\draw[hedge] (x17) -- (x13);
\draw[hedge] (x17) -- (x2);

\end{scope}
\begin{scope}[shift={(40,-40)}]

\node[novertex] (x0) at (6,-7){$\mathbf{I_{9}}$};

\node[hvertex] (x1) at (0,-7){};
\node[hvertex] (x2) at (0,-4){};
\node[hvertex] (x3) at (0,4){};
\node[hvertex] (x4) at (0,7){};

\node[hvertex] (x5) at (-8,-4){};
\node[hvertex] (x6) at (-8,4){};

\node[hvertex] (x7) at (8,-4){};
\node[hvertex] (x8) at (8,4){};

\node[hvertex] (x9) at (-4,0){};

\node[hvertex] (x10) at (4,0){};

\node[hvertex] (x11) at (-8,0){};
\node[hvertex] (x12) at (0,0){};
\node[hvertex] (x13) at (8,0){};

\node[hvertex] (x14) at (-6.5,-1.5){};

\node[hvertex] (x15) at (-6.3,1.4){};
\node[hvertex] (x16) at (-5,2.6){};

\draw[hedge] (x1) -- (x2);
\draw[dedge] (x2) -- (x12);
\draw[hedge] (x12) -- (x3);
\draw[hedge] (x3) -- (x4);

\draw[hedge] (x5) -- (x11);
\draw[hedge] (x11) -- (x6);

\draw[hedge] (x7) -- (x13);
\draw[hedge] (x13) -- (x8);

\draw[hedge] (x5) -- (x1);
\draw[hedge] (x1) -- (x7);

\draw[hedge] (x5) -- (x2);
\draw[dedge] (x2) -- (x7);

\draw[hedge] (x6) -- (x3);
\draw[hedge] (x3) -- (x8);

\draw[hedge] (x6) -- (x4);
\draw[hedge] (x4) -- (x8);

\draw[hedge] (x9) -- (x11);
\draw[dedge] (x9) -- (x12);
\draw[hedge] (x9) -- (x3);

\draw[hedge] (x10) -- (x12);
\draw[hedge] (x10) -- (x13);
\draw[hedge] (x10) -- (x2);
\draw[hedge] (x10) -- (x3);

\draw[hedge] (x12) -- (x5);

\draw[hedge] (x13) -- (x2);
\draw[hedge] (x13) -- (x3);

\draw[hedge] (x14) -- (x11);
\draw[hedge] (x14) -- (x9);
\draw[hedge] (x14) -- (x12);
\draw[hedge] (x14) -- (x5);

\draw[dedge] (x15) -- (x16);
\draw[dedge] (x15) -- (x9);
\draw[hedge] (x15) -- (x11);
\draw[dedge] (x15) -- (x6);

\draw[dedge] (x16) -- (x9);
\draw[hedge] (x16) -- (x3);
\draw[dedge] (x16) -- (x6);

\end{scope}
\begin{scope}[shift={(0,-60)}]

\node[novertex] (x0) at (6,-7){$\mathbf{I_{10}}$};

\node[hvertex] (x1) at (0,-7){};
\node[hvertex] (x2) at (0,-4){};
\node[hvertex] (x3) at (0,4){};
\node[hvertex] (x4) at (0,7){};

\node[hvertex] (x5) at (-8,-4){};
\node[hvertex] (x6) at (-8,4){};

\node[hvertex] (x7) at (8,-4){};
\node[hvertex] (x8) at (8,4){};

\node[hvertex] (x9) at (-8,0){};
\node[hvertex] (x10) at (8,0){};

\node[hvertex] (x11) at (-5.3,0){};
\node[hvertex] (x12) at (-2.6,0){};
\node[hvertex] (x13) at (0,0){};
\node[hvertex] (x14) at (2.6,0){};
\node[hvertex] (x15) at (4.6,0){};

\node[hvertex] (x16) at (-6.6,-1.5){};
\node[hvertex] (x17) at (-6.6,1.5){};

\node[hvertex] (x18) at (-5.3,3){};

\node[hvertex] (x19) at (0,1.5){};

\draw[hedge] (x1) -- (x2);

\draw[hedge] (x3) -- (x4);

\draw[hedge] (x5) -- (x1);
\draw[hedge] (x1) -- (x7);

\draw[hedge] (x5) -- (x2);
\draw[hedge] (x2) -- (x7);

\draw[dedge] (x6) -- (x3);
\draw[hedge] (x3) -- (x8);

\draw[hedge] (x4) -- (x8);
\draw[hedge] (x4) -- (x6);

\draw[dedge] (x5) -- (x9);
\draw[dedge] (x9) -- (x6);

\draw[hedge] (x7) -- (x10);
\draw[hedge] (x10) -- (x8);

\draw[hedge] (x9) -- (x11);
\draw[hedge] (x11) -- (x12);
\draw[hedge] (x12) -- (x13);
\draw[hedge] (x13) -- (x14);
\draw[hedge] (x14) -- (x15);
\draw[hedge] (x15) -- (x10);

\draw[hedge] (x2) -- (x14);
\draw[hedge] (x2) -- (x15);
\draw[hedge] (x2) -- (x10);

\draw[hedge] (x16) -- (x11);
\draw[dedge] (x16) -- (x12);
\draw[dedge] (x16) -- (x9);
\draw[dedge] (x16) -- (x5);

\draw[dedge] (x5) -- (x12);
\draw[hedge] (x5) -- (x13);
\draw[hedge] (x5) -- (x14);

\draw[hedge] (x17) -- (x11);
\draw[hedge] (x17) -- (x6);
\draw[hedge] (x17) -- (x9);
\draw[hedge] (x17) -- (x18);

\draw[hedge] (x18) -- (x6);
\draw[hedge] (x18) -- (x11);
\draw[hedge] (x18) -- (x3);

\draw[hedge] (x3) -- (x11);
\draw[dedge] (x3) -- (x12);
\draw[hedge] (x3) -- (x14);
\draw[hedge] (x3) -- (x15);
\draw[hedge] (x3) -- (x10);

\draw[hedge] (x19) -- (x3);
\draw[hedge] (x19) -- (x12);
\draw[hedge] (x19) -- (x13);
\draw[hedge] (x19) -- (x14);

\end{scope}
\begin{scope}[shift={(20,-60)}]

\node[novertex] (x0) at (6,-7){$\mathbf{I_{11}}$};

\node[hvertex] (x1) at (0,-7){};
\node[hvertex] (x2) at (0,-4){};
\node[hvertex] (x3) at (0,4){};
\node[hvertex] (x4) at (0,7){};

\node[hvertex] (x5) at (-8,-4){};
\node[hvertex] (x6) at (-8,4){};

\node[hvertex] (x7) at (8,-4){};
\node[hvertex] (x8) at (8,4){};

\node[hvertex] (x9) at (-4,0){};

\node[hvertex] (x10) at (4,0){};

\node[hvertex] (x11) at (-8,0){};
\node[hvertex] (x12) at (0,0){};
\node[hvertex] (x13) at (8,0){};

\node[hvertex] (x15) at (-6.3,1.4){};
\node[hvertex] (x16) at (-5,2.6){};

\node[hvertex] (x17) at (-6.3,-1.4){};
\node[hvertex] (x18) at (-5,-2.6){};

\draw[hedge] (x1) -- (x2);
\draw[hedge] (x2) -- (x12);
\draw[hedge] (x12) -- (x3);
\draw[hedge] (x3) -- (x4);

\draw[hedge] (x5) -- (x11);
\draw[hedge] (x11) -- (x6);

\draw[hedge] (x7) -- (x13);
\draw[hedge] (x13) -- (x8);

\draw[hedge] (x5) -- (x1);
\draw[hedge] (x1) -- (x7);

\draw[hedge] (x5) -- (x2);
\draw[hedge] (x2) -- (x7);

\draw[hedge] (x6) -- (x3);
\draw[hedge] (x3) -- (x8);

\draw[hedge] (x6) -- (x4);
\draw[hedge] (x4) -- (x8);

\draw[hedge] (x9) -- (x2);
\draw[hedge] (x9) -- (x3);
\draw[hedge] (x9) -- (x11);
\draw[hedge] (x9) -- (x12);

\draw[hedge] (x10) -- (x12);
\draw[hedge] (x10) -- (x13);
\draw[hedge] (x10) -- (x2);
\draw[hedge] (x10) -- (x3);

\draw[hedge] (x13) -- (x2);
\draw[hedge] (x13) -- (x3);

\draw[hedge] (x15) -- (x16);
\draw[hedge] (x15) -- (x9);
\draw[hedge] (x15) -- (x11);
\draw[hedge] (x15) -- (x6);

\draw[hedge] (x16) -- (x9);
\draw[hedge] (x16) -- (x3);
\draw[hedge] (x16) -- (x6);

\draw[hedge] (x18) -- (x17);
\draw[hedge] (x18) -- (x9);
\draw[hedge] (x17) -- (x11);
\draw[hedge] (x18) -- (x5);

\draw[hedge] (x17) -- (x9);
\draw[hedge] (x18) -- (x2);
\draw[hedge] (x17) -- (x5);
\end{scope}
\begin{scope}[shift={(40,-60)}]

\node[novertex] (x0) at (6,-7){$\mathbf{I_{12}}$};

\node[hvertex] (x1) at (0,-7){};
\node[hvertex] (x2) at (0,-4){};
\node[hvertex] (x3) at (0,4){};
\node[hvertex] (x4) at (0,7){};

\node[hvertex] (x5) at (-8,-4){};
\node[hvertex] (x6) at (-8,4){};

\node[hvertex] (x7) at (8,-4){};
\node[hvertex] (x8) at (8,4){};

\node[hvertex] (x91) at (-4,2){};
\node[hvertex] (x92) at (-4,-2){};

\node[hvertex] (x101) at (4,2){};
\node[hvertex] (x102) at (4,-2){};

\node[hvertex] (x11) at (-8,0){};
\node[hvertex] (x12) at (0,0){};
\node[hvertex] (x13) at (8,0){};

\node[hvertex] (x14) at (-5.4,0){};
\node[hvertex] (x15) at (-2.7,0){};

\draw[hedge] (x1) -- (x2);
\draw[hedge] (x2) -- (x12);
\draw[hedge] (x12) -- (x3);
\draw[hedge] (x3) -- (x4);

\draw[hedge] (x5) -- (x11);
\draw[hedge] (x11) -- (x6);

\draw[hedge] (x7) -- (x13);
\draw[hedge] (x13) -- (x8);

\draw[hedge] (x5) -- (x1);
\draw[hedge] (x1) -- (x7);

\draw[hedge] (x5) -- (x2);
\draw[hedge] (x2) -- (x7);

\draw[hedge] (x6) -- (x3);
\draw[dedge] (x3) -- (x8);

\draw[hedge] (x6) -- (x4);
\draw[hedge] (x4) -- (x8);

\draw[hedge] (x6) -- (x91);
\draw[dedge] (x3) -- (x91);
\draw[hedge] (x11) -- (x91);
\draw[hedge] (x12) -- (x91);

\draw[dedge] (x5) -- (x92);
\draw[hedge] (x2) -- (x92);
\draw[hedge] (x11) -- (x92);
\draw[hedge] (x12) -- (x92);

\draw[hedge] (x8) -- (x101);
\draw[hedge] (x3) -- (x101);
\draw[hedge] (x7) -- (x102);
\draw[hedge] (x2) -- (x102);

\draw[hedge] (x12) -- (x101);
\draw[hedge] (x13) -- (x101);
\draw[hedge] (x12) -- (x102);
\draw[hedge] (x13) -- (x102);

\draw[hedge] (x11) -- (x14);
\draw[dedge] (x14) -- (x15);
\draw[hedge] (x15) -- (x12);
\draw[hedge] (x12) -- (x13);

\draw[dedge] (x14) -- (x91);
\draw[dedge] (x14) -- (x92);

\draw[dedge] (x15) -- (x91);
\draw[dedge] (x15) -- (x92);

\end{scope}
\begin{scope}[shift={(0,-80)}]

\node[novertex] (x0) at (6,-7){$\mathbf{I_{13}}$};

\node[hvertex] (x1) at (0,-7){};
\node[hvertex] (x2) at (0,-4){};
\node[hvertex] (x3) at (0,4){};
\node[hvertex] (x4) at (0,7){};

\node[hvertex] (x5) at (-8,-4){};
\node[hvertex] (x6) at (-8,4){};

\node[hvertex] (x7) at (8,-4){};
\node[hvertex] (x8) at (8,4){};

\node[hvertex] (x91) at (-4,2){};
\node[hvertex] (x92) at (-4,-2){};

\node[hvertex] (x11) at (-8,0){};
\node[hvertex] (x12) at (0,0){};
\node[hvertex] (x13) at (8,0){};

\node[hvertex] (x14) at (-5.4,0){};
\node[hvertex] (x15) at (-2.7,0){};

\node[hvertex] (x16) at (-2,-3){};
\node[hvertex] (x17) at (-2,3){};

\draw[hedge] (x1) -- (x2);
\draw[hedge] (x2) -- (x12);
\draw[hedge] (x12) -- (x3);
\draw[hedge] (x3) -- (x4);

\draw[hedge] (x5) -- (x11);
\draw[hedge] (x11) -- (x6);

\draw[hedge] (x7) -- (x13);
\draw[hedge] (x13) -- (x8);

\draw[hedge] (x5) -- (x1);
\draw[hedge] (x1) -- (x7);

\draw[hedge] (x5) -- (x2);
\draw[hedge] (x2) -- (x7);

\draw[hedge] (x6) -- (x3);
\draw[hedge] (x3) -- (x8);

\draw[hedge] (x6) -- (x4);
\draw[hedge] (x4) -- (x8);

\draw[hedge] (x6) -- (x91);
\draw[hedge] (x11) -- (x91);
\draw[hedge] (x12) -- (x91);

\draw[hedge] (x5) -- (x92);
\draw[hedge] (x11) -- (x92);
\draw[hedge] (x12) -- (x92);

\draw[hedge] (x11) -- (x14);
\draw[hedge] (x14) -- (x15);
\draw[hedge] (x15) -- (x12);
\draw[hedge] (x12) -- (x13);

\draw[hedge] (x14) -- (x91);
\draw[hedge] (x14) -- (x92);

\draw[hedge] (x15) -- (x91);
\draw[hedge] (x15) -- (x92);

\draw[hedge] (x16) -- (x2);
\draw[hedge] (x16) -- (x92);
\draw[hedge] (x16) -- (x12);
\draw[hedge] (x16) -- (x5);

\draw[hedge] (x17) -- (x91);
\draw[hedge] (x17) -- (x3);
\draw[hedge] (x17) -- (x6);
\draw[hedge] (x17) -- (x12);

\draw[hedge] (x13) -- (x2);
\draw[hedge] (x13) -- (x3);

\end{scope}
\begin{scope}[shift={(20,-80)}]

\node[novertex] (x0) at (6,-7){$\mathbf{I_{14}}$};

\node[hvertex] (x1) at (0,-7){};
\node[hvertex] (x2) at (0,-4){};
\node[hvertex] (x3) at (0,4){};
\node[hvertex] (x4) at (0,7){};

\node[hvertex] (x5) at (-8,-4){};
\node[hvertex] (x6) at (-8,4){};

\node[hvertex] (x7) at (8,-4){};
\node[hvertex] (x8) at (8,4){};

\node[hvertex] (x9) at (0,0){};
\node[hvertex] (x10) at (4,0){};

\node[hvertex] (x11) at (6,2){};

\node[hvertex] (x12) at (8,0){};

\node[hvertex] (x13) at (-4,-2){};

\node[hvertex] (x14) at (-8,0){};

\draw[hedge] (x1) -- (x2);
\draw[dedge] (x2) -- (x9);
\draw[hedge] (x9) -- (x3);
\draw[hedge] (x3) -- (x4);

\draw[hedge] (x5) -- (x14);
\draw[hedge] (x14) -- (x6);

\draw[hedge] (x7) -- (x12);
\draw[hedge] (x12) -- (x8);

\draw[hedge] (x5) -- (x1);
\draw[hedge] (x1) -- (x7);

\draw[hedge] (x5) -- (x2);
\draw[dedge] (x2) -- (x7);

\draw[hedge] (x6) -- (x3);
\draw[hedge] (x3) -- (x8);

\draw[hedge] (x6) -- (x4);
\draw[hedge] (x4) -- (x8);

\draw[hedge] (x3) -- (x10);
\draw[dedge] (x7) -- (x10);
\draw[dedge] (x2) -- (x10);

\draw[hedge] (x11) -- (x10);
\draw[hedge] (x11) -- (x8);
\draw[hedge] (x11) -- (x12);
\draw[hedge] (x11) -- (x3);

\draw[dedge] (x9) -- (x10);
\draw[hedge] (x10) -- (x12);

\draw[dedge] (x9) -- (x6);

\draw[hedge] (x9) -- (x14);

\draw[hedge] (x13) -- (x2);
\draw[hedge] (x13) -- (x5);
\draw[hedge] (x13) -- (x14);
\draw[hedge] (x13) -- (x9);

\end{scope}
\begin{scope}[shift={(40,-80)}]

\node[novertex] (x0) at (6,-7){$\mathbf{I_{15}}$};

\node[hvertex] (x1) at (0,-7){};
\node[hvertex] (x2) at (0,-4){};
\node[hvertex] (x3) at (0,4){};
\node[hvertex] (x4) at (0,7){};

\node[hvertex] (x5) at (-8,-4){};
\node[hvertex] (x6) at (-8,4){};

\node[hvertex] (x7) at (8,-4){};
\node[hvertex] (x8) at (8,4){};

\node[hvertex] (x91) at (-4,2){};
\node[hvertex] (x92) at (-4,-2){};

\node[hvertex] (x10) at (4,0){};

\node[hvertex] (x11) at (-8,0){};
\node[hvertex] (x12) at (0,0){};
\node[hvertex] (x13) at (8,0){};

\draw[hedge] (x1) -- (x2);
\draw[hedge] (x2) -- (x12);
\draw[hedge] (x12) -- (x3);
\draw[hedge] (x3) -- (x4);

\draw[hedge] (x5) -- (x11);
\draw[hedge] (x11) -- (x6);

\draw[hedge] (x7) -- (x13);
\draw[hedge] (x13) -- (x8);

\draw[hedge] (x5) -- (x1);
\draw[hedge] (x1) -- (x7);

\draw[hedge] (x5) -- (x2);
\draw[hedge] (x2) -- (x7);

\draw[hedge] (x6) -- (x3);
\draw[hedge] (x3) -- (x8);

\draw[hedge] (x6) -- (x4);
\draw[hedge] (x4) -- (x8);

\draw[hedge] (x6) -- (x91);
\draw[hedge] (x3) -- (x91);
\draw[hedge] (x11) -- (x91);
\draw[hedge] (x12) -- (x91);

\draw[hedge] (x5) -- (x92);
\draw[hedge] (x2) -- (x92);
\draw[hedge] (x11) -- (x92);
\draw[hedge] (x12) -- (x92);

\draw[hedge] (x8) -- (x10);
\draw[hedge] (x3) -- (x10);
\draw[hedge] (x7) -- (x10);
\draw[hedge] (x2) -- (x10);

\draw[hedge] (x11) -- (x12);
\draw[hedge] (x12) -- (x10);
\draw[hedge] (x10) -- (x13);

\end{scope}
\begin{scope}[shift={(10,-100)}]

\node[novertex] (x0) at (6,-7){$\mathbf{I_{17}}$};

\node[hvertex] (x1) at (0,-7){};
\node[hvertex] (x2) at (0,-4){};
\node[hvertex] (x3) at (0,4){};
\node[hvertex] (x4) at (0,7){};

\node[hvertex] (x5) at (-8,-4){};
\node[hvertex] (x6) at (-8,4){};

\node[hvertex] (x7) at (8,-4){};
\node[hvertex] (x8) at (8,4){};

\node[hvertex] (x9) at (-3,0){};

\node[hvertex] (x11) at (-8,0){};
\node[hvertex] (x12) at (0,0){};
\node[hvertex] (x13) at (8,0){};

\node[hvertex] (x14) at (0,-2.7){};
\node[hvertex] (x15) at (0,-1.3){};
\node[hvertex] (x16) at (1.5,-3){};

\node[hvertex] (x17) at (2,0){};
\node[hvertex] (x18) at (4,0){};
\node[hvertex] (x19) at (6,0){};

\node[hvertex] (x20) at (1.2,1.3){};

\node[hvertex] (x21) at (4.5,1.6){};
\node[hvertex] (x22) at (7,2){};

\node[hvertex] (x23) at (5.8,3){};

\draw[hedge] (x1) -- (x2);
\draw[dedge] (x2) -- (x14);
\draw[dedge] (x14) -- (x15);
\draw[dedge] (x15) -- (x12);
\draw[hedge] (x12) -- (x3);
\draw[hedge] (x3) -- (x4);

\draw[hedge] (x5) -- (x11);
\draw[hedge] (x11) -- (x6);

\draw[hedge] (x7) -- (x13);
\draw[hedge] (x13) -- (x8);

\draw[hedge] (x5) -- (x1);
\draw[hedge] (x1) -- (x7);

\draw[dedge] (x5) -- (x2);
\draw[hedge] (x2) -- (x7);

\draw[hedge] (x6) -- (x3);
\draw[hedge] (x3) -- (x8);

\draw[hedge] (x3) -- (x11);
\draw[hedge] (x3) -- (x9);

\draw[hedge] (x6) -- (x4);
\draw[hedge] (x4) -- (x8);

\draw[hedge] (x14) -- (x16);
\draw[hedge] (x15) -- (x16);
\draw[hedge] (x2) -- (x16);

\draw[hedge] (x14) -- (x11);
\draw[hedge] (x15) -- (x11);
\draw[hedge] (x15) -- (x9);
\draw[hedge] (x2) -- (x11);

\draw[hedge] (x11) -- (x9);
\draw[hedge] (x9) -- (x12);

\draw[hedge] (x15) -- (x7);
\draw[hedge] (x16) -- (x7);
\draw[hedge] (x2) -- (x7);
\draw[hedge] (x12) -- (x7);
\draw[hedge] (x17) -- (x7);
\draw[hedge] (x18) -- (x7);
\draw[hedge] (x19) -- (x7);

\draw[dedge] (x12) -- (x17);
\draw[dedge] (x17) -- (x18);
\draw[hedge] (x18) -- (x19);
\draw[hedge] (x19) -- (x13);

\draw[hedge] (x20) -- (x12);
\draw[hedge] (x20) -- (x17);
\draw[hedge] (x20) -- (x18);
\draw[hedge] (x20) -- (x3);
\draw[hedge] (x18) -- (x3);

\draw[hedge] (x21) -- (x3);
\draw[dedge] (x21) -- (x18);
\draw[hedge] (x21) -- (x19);
\draw[hedge] (x21) -- (x13);
\draw[dedge] (x21) -- (x22);
\draw[dedge] (x21) -- (x23);

\draw[dedge] (x22) -- (x23);
\draw[hedge] (x22) -- (x13);
\draw[dedge] (x22) -- (x8);

\draw[hedge] (x23) -- (x3);
\draw[dedge] (x23) -- (x8);

\end{scope}
\begin{scope}[shift={(30,-100)}]

\node[novertex] (x0) at (6,-7){$\mathbf{I_{20}}$};

\node[hvertex] (x1) at (-2,-7){};
\node[hvertex] (x2) at (2,-7){};

\node[hvertex] (x3) at (-2,7){};
\node[hvertex] (x4) at (2,7){};

\node[hvertex] (x5) at (-8,-4){};
\node[hvertex] (x6) at (-8,4){};

\node[hvertex] (x7) at (8,-4){};
\node[hvertex] (x8) at (8,4){};

\node[hvertex] (x9) at (0,-4){};

\node[hvertex] (x10) at (0,4){};

\node[hvertex] (x11) at (-1.5,-1){};
\node[hvertex] (x12) at (1.5,-1){};

\node[hvertex] (x13) at (-2,1.5){};
\node[hvertex] (x14) at (2,1.5){};

\node[hvertex] (x15) at (-5,-1){};
\node[hvertex] (x16) at (5,-1){};

\draw[hedge] (x1) -- (x2);
\draw[hedge] (x2) -- (x7);
\draw[hedge] (x7) -- (x8);
\draw[hedge] (x8) -- (x4);
\draw[hedge] (x4) -- (x3);
\draw[hedge] (x3) -- (x6);
\draw[hedge] (x6) -- (x5);
\draw[hedge] (x5) -- (x1);

\draw[hedge] (x1) -- (x9);
\draw[hedge] (x2) -- (x9);
\draw[hedge] (x5) -- (x9);
\draw[hedge] (x7) -- (x9);
\draw[hedge] (x11) -- (x9);
\draw[hedge] (x12) -- (x9);

\draw[hedge] (x3) -- (x10);
\draw[hedge] (x4) -- (x10);
\draw[hedge] (x6) -- (x10);
\draw[hedge] (x8) -- (x10);
\draw[hedge] (x11) -- (x10);
\draw[hedge] (x12) -- (x10);

\draw[hedge] (x11) -- (x12);
\draw[hedge] (x12) -- (x14);
\draw[hedge] (x14) -- (x10);
\draw[hedge] (x10) -- (x13);
\draw[hedge] (x13) -- (x11);

\draw[hedge] (x5) -- (x11);
\draw[hedge] (x7) -- (x12);

\draw[hedge] (x15) -- (x5);
\draw[hedge] (x15) -- (x11);
\draw[hedge] (x15) -- (x13);
\draw[hedge] (x15) -- (x6);

\draw[hedge] (x16) -- (x7);
\draw[hedge] (x16) -- (x8);
\draw[hedge] (x16) -- (x14);
\draw[hedge] (x16) -- (x12);

\draw[hedge] (x13) -- (x6);
\draw[hedge] (x14) -- (x8);

\end{scope}

\end{tikzpicture}
\end{center}

\caption{Some irreducible Eulerian triangulations of the projective plane. Opposite points on the outer cycles are identified. }
\label{fig:irred_all}
\end{figure}
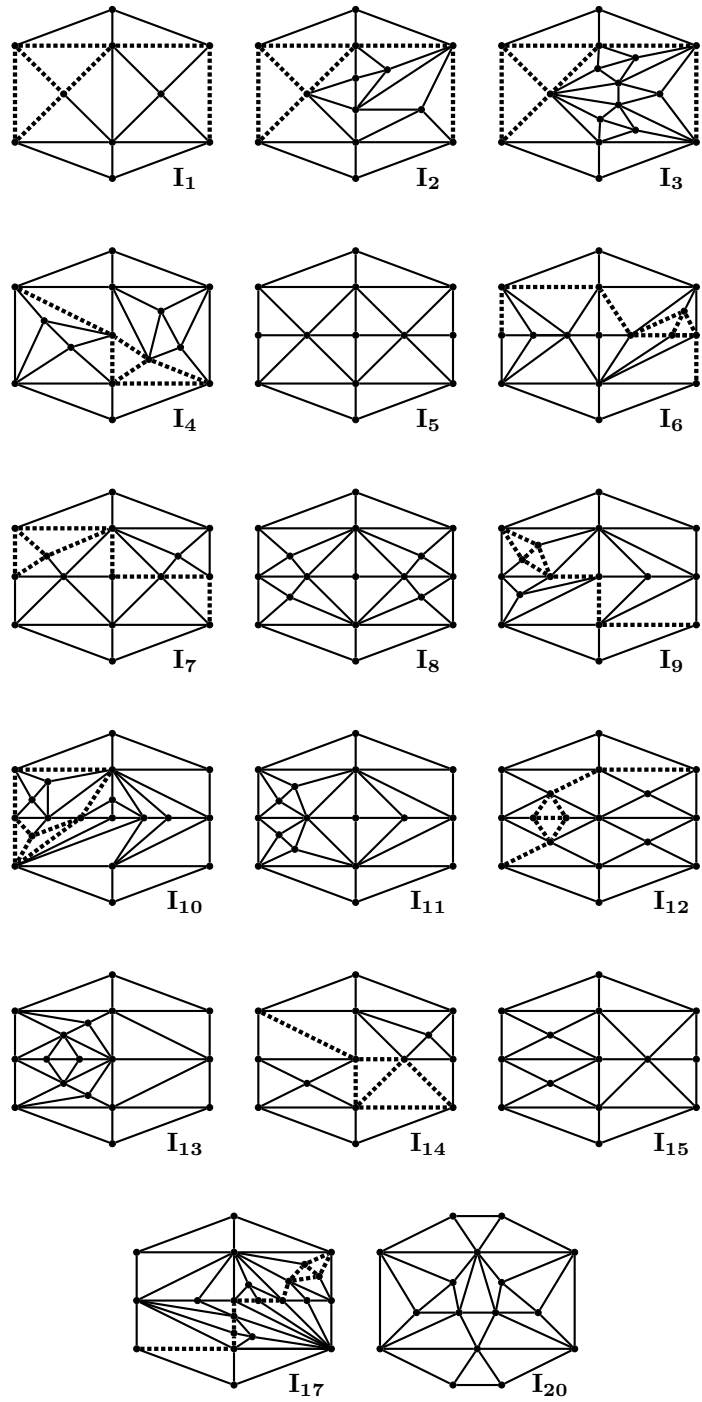

%% file: fig-lem_I_18_I_19.tex
\begin{figure}[htb]
\begin{center}
\begin{tikzpicture}[scale = .18]

\begin{scope}
\node[hvertex] (x1) at (0,-9){};
\node[hvertex] (x2) at (0,9){};
\node[hvertex] (x3) at (16,9){};
\node[hvertex] (x4) at (-16,9){};
\node[hvertex] (x5) at (-16,-9){};
\node[hvertex] (x6) at (16,-9){};

\node[hvertex] (x7) at (2,-3){};
\node[hvertex] (x8) at (-2,-3){};
\node[hvertex] (x9) at (2,3){};
\node[hvertex] (x10) at (-2,3){};

\node[hvertex] (x11) at (-7,6){};
\node[hvertex] (x12) at (-10,3){};

\node[hvertex] (x13) at (7,-6){};
\node[hvertex] (x14) at (10,-3){};

\node[hvertex] (x15) at (1,5){};
\node[hvertex] (x16) at (4,6){};

\node[hvertex] (x17) at (-1,-5){};
\node[hvertex] (x18) at (-4,-6){};

\node[hvertex] (x19) at (-0.5,1.5){};
\node[hvertex] (x20) at (-3,0){};

\node[hvertex] (x21) at (0.5,-1.5){};
\node[hvertex] (x22) at (3,0){};

\draw[hedge] (x1) -- (x5);
\draw[hedge] (x1) -- (x6);
\draw[hedge] (x5) -- (x4);
\draw[hedge] (x6) -- (x3);
\draw[hedge] (x2) -- (x4);
\draw[hedge] (x2) -- (x3);

\draw[dedge] (x1) -- (x7);
\draw[dedge] (x7) -- (x8);
\draw[dedge] (x8) -- (x9);
\draw[dedge] (x9) -- (x10);
\draw[dedge] (x10) -- (x2);

\draw[hedge] (x7) -- (x3);
\draw[hedge] (x9) -- (x3);
\draw[hedge] (x8) -- (x5);
\draw[hedge] (x10) -- (x5);

\draw[hedge] (x11) -- (x12);
\draw[hedge] (x11) -- (x2);
\draw[hedge] (x11) -- (x10);
\draw[hedge] (x11) -- (x4);
\draw[hedge] (x12) -- (x4);
\draw[hedge] (x12) -- (x10);
\draw[hedge] (x12) -- (x5);

\draw[hedge] (x13) -- (x14);
\draw[hedge] (x13) -- (x1);
\draw[hedge] (x13) -- (x7);
\draw[hedge] (x13) -- (x6);
\draw[hedge] (x14) -- (x6);
\draw[hedge] (x14) -- (x7);
\draw[hedge] (x14) -- (x3);

\draw[hedge] (x15) -- (x16);
\draw[hedge] (x15) -- (x2);
\draw[hedge] (x15) -- (x10);
\draw[hedge] (x15) -- (x9);
\draw[hedge] (x16) -- (x2);
\draw[hedge] (x16) -- (x9);
\draw[hedge] (x16) -- (x3);

\draw[hedge] (x17) -- (x18);
\draw[hedge] (x17) -- (x1);
\draw[hedge] (x17) -- (x7);
\draw[hedge] (x17) -- (x8);
\draw[hedge] (x18) -- (x1);
\draw[hedge] (x18) -- (x8);
\draw[hedge] (x18) -- (x5);

\draw[hedge] (x19) -- (x20);
\draw[hedge] (x19) -- (x9);
\draw[hedge] (x19) -- (x8);
\draw[hedge] (x19) -- (x10);
\draw[hedge] (x20) -- (x8);
\draw[hedge] (x20) -- (x10);
\draw[hedge] (x20) -- (x5);

\draw[hedge] (x21) -- (x22);
\draw[hedge] (x21) -- (x8);
\draw[hedge] (x21) -- (x9);
\draw[hedge] (x21) -- (x7);
\draw[hedge] (x22) -- (x9);
\draw[hedge] (x22) -- (x7);
\draw[hedge] (x22) -- (x3);

\end{scope}

\begin{scope}[shift={(35,0)}]

\node[hvertex] (y1) at (0,-13){};
\node[novertex] (ly1) at (0,-14.5){$v_3$};
\node[hvertex] (y2) at (0,11){};
\node[novertex] (ly2) at (0,12.5){$v_1$};
\node[hvertex] (y3) at (15,11){};
\node[hvertex] (y4) at (-15,11){};
\node[hvertex] (y5) at (-15,-13){};
\node[hvertex] (y6) at (15,-13){};

\node[hvertex] (y7) at (0,9){};
\node[hvertex] (y8) at (0,-11){};

\node[hvertex] (y9) at (-9,9){};

\node[hvertex] (y10) at (3,8){};
\node[hvertex] (y11) at (-3,6){};

\node[hvertex] (y12) at (0,5){};

\node[hvertex] (y13) at (-6,7){};
\node[hvertex] (y14) at (-5,8){};

\node[hvertex] (y15) at (0,-7){};
\node[hvertex] (y16) at (9,5){};

\node[hvertex] (y17) at (6,6){};
\node[hvertex] (y18) at (7,7){};

\node[hvertex] (y19) at (-9,-7){};
\node[novertex] (ly19) at (-5.7,-6.4){$v_2$};

\node[hvertex] (y23) at (3,-8){};
\node[hvertex] (y20) at (-3,-10){};

\node[hvertex] (y21) at (-6,-9){};
\node[hvertex] (y22) at (-5,-8){};

\node[hvertex] (y24) at (9,-11){};

\node[hvertex] (y25) at (6,-10){};
\node[hvertex] (y26) at (7,-9){};

\draw[hedge] (y1) -- (y5);
\draw[hedge] (y1) -- (y6);
\draw[hedge] (y5) -- (y4);
\draw[hedge] (y6) -- (y3);
\draw[hedge] (y2) -- (y4);
\draw[hedge] (y2) -- (y3);

\draw[dedge] (y1) -- (y8);
\draw[dedge] (y2) -- (y7);

\draw[hedge] (y9) -- (y4);
\draw[hedge] (y9) -- (y5);
\draw[hedge] (y9) -- (y2);
\draw[hedge] (y9) -- (y7);
\draw[hedge] (y7) -- (y3);

\draw[hedge] (y11) -- (y5);
\draw[hedge] (y11) -- (y7);
\draw[hedge] (y11) -- (y11);
\draw[hedge] (y11) -- (y12);
\draw[hedge] (y10) -- (y11);

\draw[hedge] (y12) -- (y5);
\draw[hedge] (y10) -- (y12);
\draw[dedge] (y10) -- (y7);
\draw[hedge] (y10) -- (y3);

\draw[hedge] (y13) -- (y14);
\draw[hedge] (y9) -- (y13);
\draw[hedge] (y11) -- (y13);
\draw[hedge] (y5) -- (y13);
\draw[hedge] (y14) -- (y11);
\draw[hedge] (y14) -- (y9);
\draw[hedge] (y14) -- (y7);

\draw[hedge] (y15) -- (y3);

\draw[hedge] (y16) -- (y12);
\draw[hedge] (y16) -- (y15);
\draw[hedge] (y16) -- (y3);
\draw[dedge] (y16) -- (y18);
\draw[hedge] (y16) -- (y17);
\draw[hedge] (y17) -- (y18);
\draw[hedge] (y17) -- (y12);
\draw[hedge] (y17) -- (y10);
\draw[dedge] (y18) -- (y10);
\draw[hedge] (y18) -- (y3);

\draw[hedge] (y19) -- (y12);
\draw[dedge] (y19) -- (y16);
\draw[hedge] (y19) -- (y15);
\draw[hedge] (y19) -- (y5);

\draw[hedge] (y20) -- (y5);
\draw[dedge] (y20) -- (y8);
\draw[hedge] (y20) -- (y15);

\draw[dedge] (y21) -- (y19);
\draw[hedge] (y21) -- (y5);
\draw[dedge] (y21) -- (y20);
\draw[hedge] (y21) -- (y22);
\draw[hedge] (y22) -- (y20);
\draw[hedge] (y22) -- (y19);
\draw[hedge] (y22) -- (y15);

\draw[hedge] (y23) -- (y8);
\draw[hedge] (y23) -- (y20);
\draw[hedge] (y23) -- (y15);
\draw[hedge] (y23) -- (y3);

\draw[hedge] (y24) -- (y1);
\draw[hedge] (y24) -- (y8);
\draw[hedge] (y24) -- (y6);
\draw[hedge] (y24) -- (y3);
\draw[hedge] (y25) -- (y24);
\draw[hedge] (y25) -- (y26);
\draw[hedge] (y25) -- (y8);
\draw[hedge] (y25) -- (y23);
\draw[hedge] (y26) -- (y23);
\draw[hedge] (y26) -- (y24);
\draw[hedge] (y26) -- (y3);

\draw[hedge] (y5) -- (y8);

\end{scope}

\end{tikzpicture}

\end{center}

\caption{Two odd paths described in the proof of Lemma~\ref{lem:irreducible}. Opposite points on the outer cycles are identified. }
\label{fig:lem_I_18_19}
\end{figure}
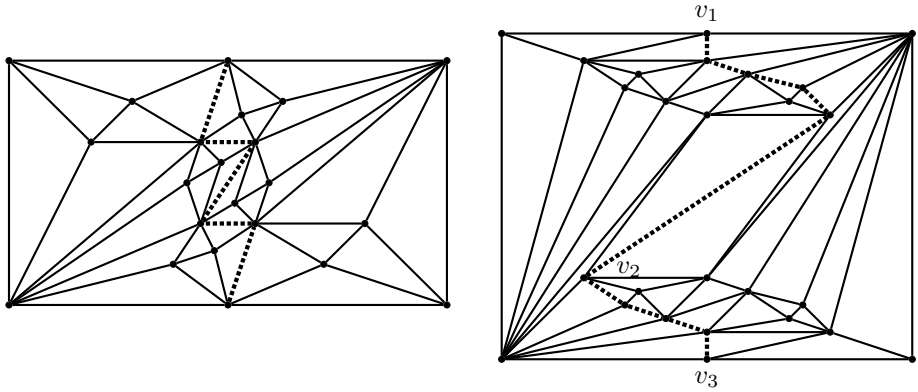

%% file: sec-even_split_low_antihole.tex
In this section, we prove Lemma~\ref{lem:even-split_low} and Lemma~\ref{lem:even-split_anti-hole}.

\begin{proof} [Proof of Lemma~\ref{lem:even-split_low}]
Let $W$ be a loose odd wheel in $G'$ with central vertex~$v$ and induced cycle~$C$. Let the even-contraction in $G$ be applied at $x$ together with $b$ and $b'$, let $N_G(x) = \lbrace  a,a,b,b' \rbrace$ and let $V(G')= V(G) \setminus \lbrace x, a,a,b,b' \rbrace \cup \lbrace  y \rbrace$. 

Evidently, $G$ still contains $W$ if $y \notin V(W)$.

Suppose that $y \in V(C)$. Let $u,u'$ be the two cycle vertices adjacent to $y$. 
Let $C'$ be the cycle obtained from $C$ by deletion of $y$ and addition of the edges $ub,bu'$ (if $ub,bu' \in E(G')$) or $ub,bx,xb',b'u'$ (if $ub, u'b' \in E(G)$). 
$C'$ is an odd induced cycle in $G'$. If $C'$ forms a loose odd wheel together with $v$, we are done. Otherwise, $b$ is adjacent to  $u$ and $u'$, but not to $v$. The vertex $b'$ is adjacent to $v$ and neither adjacent to $u$ nor to $u'$ (or vice versa). But then,  $\lbrace  u,u' \rbrace$ is different from $\lbrace  a,a' \rbrace$. Say $u \notin \lbrace  a,a' \rbrace$ . Since $C'$ together with $v$ does not form a loose odd wheel, there is a vertex $z$ such that the path $ \subpath{C'}{b}{z}{u'}$  together with $x$ and $v$ forms an odd hole. By Observation~\ref{obs:deg4->-loose_odd_wheel} we obtain a \low.

Next, suppose that $y$ equals the center vertex $v$ of $W$. 
Let $z_1$, $z_2$ and $z_3$ be three odd neighbours of $y$ in $W$. If $b$ respectively $b'$ is adjacent to all the three vertices in $G$, then $b$ respectively $b'$ forms a \low\ together with $C$. Assume that $bz_1, bz_2, b'z_3 \in E(G)$ and $b'z_1,bz_3 \notin E(G)$ (see Figure~\ref{fig:splitting}).
\input{fig-splitting_case2}
The paths $\subpath{C}{z_2}{z_3}{z_1}$  and $\subpath{C}{z_1}{z_3}{z_2}$ form odd cycles $C_{1}$ and $C_2$ together  with $b,x,b'$. By Observation~\ref{obs:deg4->-loose_odd_wheel}, $G$ contains a loose odd wheel if $C_{1}$ or $C_2$ is induced.  
Thus, assume that $C_1$ and $C_2$ have chords. 
Then, $C_{i}$ yields an induced odd subcycle $C'_{i}$. If $C'_i$ contains $b,x,b'$, it leads to a \low by Observation~\ref{obs:deg4->-loose_odd_wheel}.
Thus, $x$ is contained in a chord  of $C_i$ for $i=1,2$. As $x$ has only four neighbours, the chords are of the form $ax$ and $a'x$. Consequently, $b$ and $b'$ are also contained in chords of $C_1$ and $C_2$. (Note that it is possible that $a$ and $a'$ and coincide with vertices in $\lbrace  z_1, z_2, z_3 \rbrace$.) See Figure~\ref{fig:splitting} for an illustration.  If $ \subpath{C}{a'}{z_1}{z_2}$ or $ \subpath{C}{a}{z_2}{z_1}$ is odd, then $b$ has three odd neighbours on the cycle $C$ and yields a \low. Otherwise, the paths $ \subpath{C}{z_3}{a'}{z_1}$ and $ \subpath{C}{z_3}{a}{z_2}$ are odd and consequently, $b'$ has three odd neighbours on $C$.

This shows that $G$ contains a loose odd wheel in all cases.  
\end{proof}

\input{fig-odd_antihole_splitting}

\begin{proof}[Proof of Lemma~\ref{lem:even-split_anti-hole}]
Let the even-contraction in $G$ be applied at $x$ together with $b$ and $b'$, let $N_G(x) = \lbrace  a,a,b,b' \rbrace$ and let $V(G')= V(G) \setminus \lbrace x, a,a,b,b' \rbrace \cup \lbrace  y \rbrace$.  Assume that $G'$ contains an induced $\overline{C_7}$. Then, either $G$ also contains an induced $\overline{C_7}$ (and we are done) or $y$ is one of the vertices of $\overline{C_7}$ in $G'$. Let $u_1,u_2,u_3,v_3,v_2,v_1,y,u_1$  be a cycle of $\overline{C_7}$ in $G'$. Then, $N_{G'}(y)\cap V(\overline{C_7})= \{u_1,u_2,v_1,v_2\}$ (see also Figure~\ref{fig:odd_antihole_splitting}). In $G$, the vertices $u_1,u_2,v_1,v_2$ are neighbours of $b$ or $b'$. If a vertex is adjacent to both, $b$ and $b'$, then this vertex equals $a$ or $a'$. 
We now assume that $G$ contains no \low and no induced $\overline{C_7}$ and deduce some useful observations:
\begin{equation}\label{eq:4neighbours}
\text{ At most three of the  vertices in  $\lbrace u_1,u_2,v_1,v_2\rbrace$ are adjacent to $b$ in $G$.}
\end{equation}
Otherwise, $b$ induces a $\overline{C_7}$ in $G$ together with $ u_1,u_2,u_3,v_3,v_2,v_1$. 
\begin{equation}\label{eq:u1_in_N(b)_and_u2_in_N(b')}
\text{If $u_1b$ and $u_2b'$ are edges in $G$, then $\{u_1,u_2\}\cap\{a,a'\}\neq \emptyset$.}
\end{equation}
Otherwise, the edges $u_1b'$ and $u_2b$ are not contained in $G$. Then,  $b,x,b',u_1,u_2,b$ is an odd hole that forms a \low in $G$ together with $a$.
\begin{equation}\label{eq:u1_in_N(b)_and_v1_in_N(b')}
\text{If $u_1b$ and $v_1b'$ are edges in $G$, then $\{u_1,v_1\}\cap\{a,a'\}\neq \emptyset$.}
\end{equation}
Otherwise, the edges $u_1b'$ and $v_1b$ are not contained in $G$. Then, $b,x,b',v_1,u_1,b$ forms a \low  in $G$ together with $a$. 
Next, we note that
\begin{equation}\label{eq:u1v2_and_aa'}
\{u_1,v_2\}\neq\{a,a'\}.
\end{equation}
Otherwise,  $u_1,x,v_2,v_3,u_2,u_1$ is an odd hole that forms a \low in $G$ together with $b$.
\begin{equation}\label{eq:u1v2_in_N(b')}
\text{If $u_1$ and $v_2$ are adjacent to $b'$ in $G$, then $u_2$ is also a neighbour of $b'$.}
\end{equation} 
Otherwise, the vertices $b',u_1,u_2,u_3,v_3,v_2$ form a \low in $G$ with center $u_3$.

Note that in all observations, we can exchange $u$ and $v$ as well as $b$ and $b'$.
We now use our observations in order to get a contradiction: \\

First, we assume that $u_1$ and $v_2$ are neighbours of $b'$. Then, by~\eqref{eq:u1v2_in_N(b')}, $u_2$ is adjacent to $b'$ and by~\eqref{eq:4neighbours}, $v_1$ is not adjacent to $b'$. Further, because of~\eqref{eq:u1_in_N(b)_and_u2_in_N(b')} and~\eqref{eq:u1_in_N(b)_and_v1_in_N(b')}, $ \lbrace u_1,v_2 \rbrace =\lbrace a, a' \rbrace $. This is a contradiction to~\eqref{eq:u1v2_and_aa'}.
Consequently, 
\begin{equation}\label{eq:u1v2_no_common_neighbour_in_bb'}
\text{ $u_1$ and  $v_2$ do not have a common neighbour in $\{b,b'\}$. }
\end{equation}
Let $u_1b$ and $v_2b'$ be edges of $G$, and let $u_1b',v_2b \notin E(G)$. \\
First, assume that $v_1\in N_G(b')$. Then by \eqref{eq:u1_in_N(b)_and_v1_in_N(b')}, $v_1=a$. 
If $u_2$ is a neighbour of $b'$, then  $u_2$ equals $a'$ by~\eqref{eq:u1_in_N(b)_and_u2_in_N(b')}. This is a contradiction to~\eqref{eq:u1v2_and_aa'}. Thus, $u_2\in N_G(b)$ and $v_2b\in E(G)$ by~\eqref{eq:u1v2_in_N(b')}. This contradicts~\eqref{eq:u1v2_no_common_neighbour_in_bb'}. \\
Now let $v_1$ be adjacent to $b$ but not to $b'$. This contradicts~\eqref{eq:u1_in_N(b)_and_u2_in_N(b')} because of~\eqref{eq:u1v2_no_common_neighbour_in_bb'}.
\end{proof}

%% file: fig-splitting_case2.tex
\begin{figure}

\begin{center}
\begin{tikzpicture}[scale=1]

\def\radius{2.4}

\begin{scope}

\node[hvertex] (x) at (0,0){};
\node[novertex] (xl) at (0.09,-0.2){$x$};
\node[hvertex] (b) at (0.4,0){};
\node[novertex] (bl) at (0.6,-0.1){$b$};
\node[hvertex] (b2) at (-0.4,0){};
\node[novertex] (b2l) at (-0.6,-0.1){$b'$};
\node[hvertex] (z) at (0,-0.5*\radius){};
\node[novertex] (zl) at (0,-0.5*\radius-.2){$z_1$};
\node[hvertex] (f) at (-90+120:0.5*\radius){};
\node[novertex] (fl) at (-90+120:0.5*\radius+0.2){$z_2$};
\node[hvertex] (h) at (-90+240:0.5*\radius){};
\node[novertex] (hl) at (-90+240:0.5*\radius+0.3){$z_3$};

\begin{scope}[rotate=40]
  \draw[decorate,decoration={snake,amplitude=.4mm,segment length=2mm,post length=0mm}] (0,0) circle [radius=0.5*\radius];
\end{scope}
 \draw[hedge]  (x)--(b);
  \draw[hedge] (x)--(b2);
  \draw[hedge] (b)--(z);
  \draw[hedge] (b)--(f);
  \draw[hedge] (b2)--(h);


\end{scope}

\begin{scope}[shift={(4,0)}]

\node[hvertex] (x) at (0,0){};
\node[novertex] (xl) at (0.09,-0.17){$x$};
\node[hvertex] (b) at (0.4,0){};
\node[novertex] (bl) at (0.6,-0.1){$b$};
\node[hvertex] (b2) at (-0.4,0){};
\node[novertex] (b2l) at (-0.6,-0.1){$b'$};
\node[hvertex] (z) at (0,-0.5*\radius){};
\node[novertex] (zl) at (0,-0.5*\radius-.2){$z_1$};
\node[hvertex] (f) at (-90+120:0.5*\radius){};
\node[novertex] (fl) at (-90+120:0.5*\radius+0.2){$z_2$};
\node[hvertex] (h) at (-90+240:0.5*\radius){};
\node[novertex] (hl) at (-90+240:0.5*\radius+0.3){$z_3$};

\node[hvertex] (a) at (-90+180:0.5*\radius){};
\node[novertex] (al) at (-90+180:0.5*\radius+0.3){$a$};

\node[hvertex] (ap) at (-90+336:0.5*\radius){};
\node[novertex] (apl) at (-90+336:0.5*\radius+0.3){$a'$};

\begin{scope}[rotate=40]
  \draw[decorate,decoration={snake,amplitude=.4mm,segment length=2mm,post length=0mm}] (0,0) circle [radius=0.5*\radius];
\end{scope}
  \draw[hedge] (x)--(b);
  \draw[hedge] (x)--(b2);
  \draw[hedge] (x)--(a);
  \draw[hedge] (x)--(ap);
  \draw[hedge] (b)--(a);
  \draw[hedge] (b)--(ap);  
  \draw[hedge] (b2)--(a);
  \draw[hedge] (b2)--(ap);  
  \draw[hedge] (b)--(z);
  \draw[hedge] (b)--(f);
  \draw[hedge] (b2)--(h);

%

\end{scope}

\end{tikzpicture}
\end{center}
\caption{The situation after the splitting if $y$ equals the central vertex $v$}
\label{fig:splitting}

\end{figure}
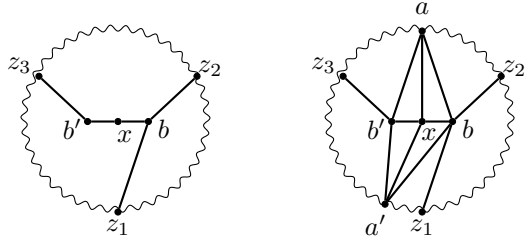

%% file: fig-odd_antihole_splitting.tex
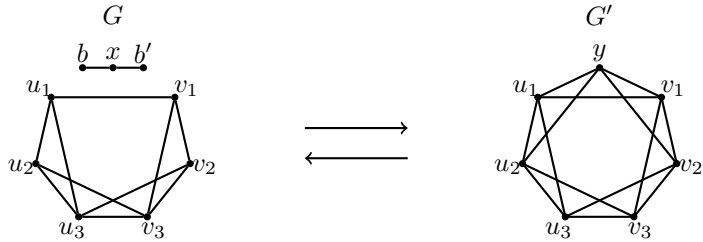
\begin{figure}[htb]
\begin{center}
\begin{tikzpicture}[scale = 0.8]

\def\radius{1.3cm}
\def\hshift{3.7cm}
\def\vshift{-3.7cm}

\begin{scope}[shift={(0,0)}]
\def\angle{360/7}
\foreach \i in {0,1,2,3,4,5,6}{
  \node[hvertex] (v\i) at (90+\angle*\i:\radius){};
}

 \draw[hedge] (v1)--(v2);
 \draw[hedge] (v2)--(v3);
 \draw[hedge] (v3)--(v4);
 \draw[hedge] (v4)--(v5);
 \draw[hedge] (v5)--(v6);
 \draw[hedge] (v1)--(v3);
 \draw[hedge] (v2)--(v4);
  \draw[hedge] (v4)--(v6);
 \draw[hedge] (v3)--(v5);
 \draw[hedge] (v6)--(v1);

\node[hvertex] (b) at (-.5,\radius){};
\node[hvertex] (b') at (.5,\radius){};
\draw[hedge] (v0)--(b);
\draw[hedge] (v0)--(b');

\node at (90+\angle*0:\radius+7){$x$};
\node at (90+\angle*1:\radius+7){$u_1$};
\node at (90+\angle*2:\radius+7){$u_2$};
\node at (90+\angle*3:\radius+7){$u_3$};
\node at (90+\angle*4:\radius+7){$v_3$};
\node at (90+\angle*5:\radius+7){$v_2$};
\node at (90+\angle*6:\radius+7){$v_1$};
\node at (-.5,\radius+7){$b$};
\node at (.5,\radius+7){$b'$};

\node at (0,\radius+25){$G$};

\end{scope}
\begin{scope} [shift={(6,0)}]

\node[novertex] (y1) at (-3,.3){};
\node[novertex] (y2) at (-3,-.2){};

\node[novertex] (y3) at (-1,.3){};
\node[novertex] (y4) at (-1,-.2){};

\draw[pedge] (y1) -- (y3);

\draw[pedge] (y4) -- (y2);

\end{scope}

\begin{scope} [shift={(8,0)}]
\def\angle{360/7}
\foreach \i in {0,1,2,3,4,5,6}{
  \begin{scope}[on background layer] 
    \draw[hedge] (90+\angle*\i:\radius) -- (90+\angle+\angle*\i:\radius);
    \draw[hedge] (90+\angle*\i:\radius) -- (90+2*\angle+\angle*\i:\radius);
  \end{scope}
  \node[hvertex] (v\i) at (90+\angle*\i:\radius){};
  }
\node at (90+\angle*0:\radius+7){$y$};
\node at (90+\angle*1:\radius+7){$u_1$};
\node at (90+\angle*2:\radius+7){$u_2$};
\node at (90+\angle*3:\radius+7){$u_3$};
\node at (90+\angle*4:\radius+7){$v_3$};
\node at (90+\angle*5:\radius+7){$v_2$};
\node at (90+\angle*6:\radius+7){$v_1$};

\node at (0,\radius+25){$G'$};

\end{scope}
\end{tikzpicture}
\end{center}
\caption{Even-splitting at $\bar{C_7}$}
\label{fig:odd_antihole_splitting}
\end{figure}

%% file: sec-even_contraction.tex
This section is dedicated to the proof of Lemma~\ref{lem:even-contraction}.

%
%

To prove Lemma~\ref{lem:even-contraction}, let $C=v_1v_2\ldots v_k,v_1$ be an odd hole in $G$ and let $G'$ be obtained by an even-contraction identifying $b,b'$ and $x$.
To the even-contraction applied at $x$ together with $b$ and $b'$ we associate the map 
\begin{equation*}
\gamma:V(G) \mapsto V(G)\setminus \{x,b,b'\}\cup \{y\}
\end{equation*}
with $\gamma(x)=\gamma(b)=\gamma(b')=y$ and $\gamma(v)=v$ for $v\in V(G)\setminus \{x,b,b'\}$.
We will abuse notation and will apply $\gamma$ also to subgraphs of $G$ and $G'$.

Our aim is to show that $G'$ contains an odd hole or $G$ contains a loose odd wheel. 
We split up this proof into several cases concerning $\lbrace a,b,a',b' \rbrace \cap V(C)$.

If $x \in V(C)$, then $G$ contains an odd wheel by Observation~\ref{obs:deg4->-loose_odd_wheel}.

If $\lbrace a,a',b,b' \rbrace \cap V(C) $ equals $\lbrace b,b' \rbrace$, then $\gamma(C)$ consists of an odd and an even induced cycle intersecting each other in the vertex $y$. $G'$ still contains an odd hole, if the odd cycle has length at least $5$. Otherwise, the cycle is a triangle $y,v,w,y$. Then, $b,v,w,b',x$ form a $5$-cycle. If this cycle is induced, $G$ contains a loose odd wheel by Observation~\ref{obs:deg4->-loose_odd_wheel}. Otherwise, $x$ has a neighbour among the cycle vertices $v,w$. Such a neighbour must equal $a$ or $a'$. This contradicts the assumption $V(C) \cap \lbrace a,a',b,b' \rbrace= \lbrace b,b' \rbrace$.

If $ \lbrace a,b,a',b' \rbrace V(C)$ is of size $3$, then $x$ has three odd neighbours on $C$.

If $ a,b,a',b' $ are all contained in $V(C)$, then $C$ is not an induced odd cycle. This is a contradiction.
\begin{equation} \label{eq:a_a'_on_C}
\text{
If  $\lbrace a,b,a',b' \rbrace \cap V(C) \subseteq \lbrace a,a' \rbrace$ then $C=\gamma(C)$ is an odd hole in $G'$.} 
\end{equation}

The cases where $\lbrace x,a,a',b,b' \rbrace \cap V(C)$ equals $\lbrace a',b' \rbrace$ or $\lbrace b'\rbrace$ are treated in Section~\ref{subsec:a'b'} and Section~\ref{subsec:b'} respectively.

\subsection{Even-contraction if $a',b'$ are contained in the odd hole} \label{subsec:a'b'}
\input{subsec-aprimebprime}

\subsection{Even-contraction if $b'$ is contained in the odd hole}
\label{subsec:b'}

\input{subsec-bprime}

%% file: subsec-aprimebprime.tex
In this subsection, we prove the following lemma.
\begin{lemma} \label{lem:a'b'}
Let $G$ be a nice Eulerian triangulation. Let $G'$ be obtained from $G$ by an even-contraction at $x$ and let $\{x,a,a',b,b'\}\cap V(C)$ equal $ \{a',b'\}$. If $G$ contains an odd hole $C$ and $G'$ is perfect, then $G$ contains a loose odd wheel.
\end{lemma}

\begin{proof}[Proof of Lemma~\ref{lem:a'b'}]
Let $G$ and $G'$ be as described. 
Appplication of the even-contraction translates the odd hole $C$ of $G$ into  a new odd cycle $K$ in $G'$ with vertex set $\gamma(V(C))=V(C) \setminus \lbrace b' \rbrace \cup \lbrace y \rbrace$.

Suppose that $G'$ does not contain an odd hole. Then, $K$ is not an odd hole, ie $K$ has chords. Further, all induced subcycles of $K$ are triangles or even cycles.

Note that each chord splits an odd cycle into an odd and an even subcycle. Therefore, 
\begin{equation} \label{eq:chord->odd_subcycle}
\text{every odd cycle has an induced odd subcycle.}
\end{equation}
If $G'$ does not contain an odd hole, the induced odd subcycle is a triangle. 

As $a'$ and $b'$ are adjacent and both contained in the induced cycle $C$, the vertices $a'$ and $b'$  consecutively appear on $C$. We define  $z,w,b',a'$ and $v$ to be five consecutive vertices of $C$.

As $C$ is induced, each chord of $K$ may have $y$ and a vertex that is adjacent to $b$ in $G$ as end vertices. Thus, the cycle $K$ may only have three different triangles as a subcycle: The vertex $y$ can be contained in a triangle together with $w,z$ (Case A) or with  $a',v$ (Case B) or with two other vertices 
that are adjacent on $K$ (Case C). Figure~\ref{fig:three_possible_triangles} shows these three cases and the associated configurations in $G$ and $G'$. Note that the cases are not exclusive.

\input{fig-three_possible_triangles}

In Case A, ie if $y,w,z,y$ is a triangle in $G'$, the vertices $b,z,w,b',a'$ form a $5$-cycle in $G$. This  cycle is induced as $C$ is induced and as $b$ can by assumption not be adjacent to $b'$ or $w$. Thus, the $5$-cycle forms a \low together with $x$.

For the two other cases, we denote by $v_0$ the vertex contained in a triangle with $y$ and a further vertex of $K$ for which 
$\subpath{K}{a'}{v_0}{y}$  is of shortest odd length.
This also implies that 

\begin{equation} \label{eq:C_a'_v0_b-shortest-odd-length}
\subpath{C}{a'}{v_0}{b'} \text{ is of shortest odd length.}
\end{equation}

This means that $v_0$ equals $v$ in Case B. 
In Case $C$, denote by $v_1$ the neighbour of $v_0$ on $C$ respectively $K$ that forms a triangle with $v_0$ and $y$. 

Note that $b$ has three odd neighbours on $C$ in $G$ if the neighbour of $v_0$ that is not contained in  $\subpath{C}{a'}{v_0}{b'}$ is adjacent to $y$ in $G'$. Thus, we may assume that $v_0$ forms a triangle together with $y$ and its neighbour $v_1$ that is contained in $\subpath{C}{a'}{v_0}{b'}$.

There are two paths connecting $w$ and $a'$ on the Hamilton cycle of $b'$. Both of these paths are induced (Observation~\ref{obs:induced_cycles}) and have the same parity (Observation~\ref{obs:Eulerian}).

Depending on parity and length of these paths, we prove the statement of this lemma independently. 
Claim~\ref{claim:nikolaus_even_hamilton} applies if the paths along $HC(b')$ connecting $a'$ and $w$ are even. If the paths are odd, see Claim~\ref{claim:nikolaus_odd_hamilton_triangle} and Claim~\ref{claim:nikolaus_odd_hamilton}.

The proof strategy of all the three claims is as follows: Starting from the subgraph of $G$ described in Case B and Case C, we analyse the surrounding vertices and edges until we find an odd cycle. We are done if such an odd cycle is an odd hole that is not destroyed by the even-contraction or if this cycle is part of a \low. If the cycle we found  is not induced, the chords lead to a new odd cycle in $G$. After some steps, we always find a \low or an odd hole not affected by the even-contraction in $G$. This means, we show that in every possible triangulation satisfying Case B or Case C, $G$ contains a \low if $G'$ contains no odd hole.

\begin{claim}\label{claim:nikolaus_even_hamilton}
Let the two paths along $HC(b')$ connecting $a'$ and $w$ be even. 
If $G'$ does not contain an odd hole, then $G$ contains a loose odd wheel.
\end{claim}

\begin{proof}
We first analyse $G$. The path $\subpath{HC(b')}{w}{a'}{a}$ forms an odd cycle together with the path $C-b'$ in $G$. If this cycle is induced, $G$ contains a loose odd wheel with center $b'$. 
If the cycle contains chords, these chords have one end vertex in  $\subpath{HC(b')}{w}{a'}{a}$ and one endvertex in $C-b'$. This comes from the fact that $C-b'$ and $\subpath{HC(b')}{w}{a'}{a}$ (see Observation~\ref{obs:induced_cycles}) are induced. The chords lead to an induced odd subcycle (see~\eqref{eq:chord->odd_subcycle}). If a longest such  cycle $\tilde{C}$  is not a triangle, $G$ contains an odd hole that is not affected by the even-contraction. Thus, suppose that $\tilde{C}$ is a triangle with either two vertices of $C-b'$ and one vertex of $\subpath{HC(b')}{w}{a'}{a}$ or vice versa. 

If $\tilde{C}$ is a triangle with one vertex $p$ of $\subpath{HC(b')}{w}{a'}{a}$ and two vertices $c,c'$ of $C-b'$, the vertex $p$ either has three odd neighbours on $C$ (namely the $c,c'$ and $b'$) or one of the two neighbours of $p$, say $c$,  is of the form $c\in V(C) \setminus \lbrace b' \rbrace$ where the path $\subpath{C}{v_0}{c}{a'} $ is even and does not contain $c'$. 
\begin{gather}
\text{ If } p \in \subpath{HC(b')}{w}{a'}{a} \text{ and } c \in C-b' \text{ and }pc \in E(G) \text{ and }  \nonumber\\
\subpath{C}{v_0}{c}{a'}  \text{ is even, then } G \mbox{ contains a loose odd wheel} \label{eq:chord}
\end{gather} 
by the following observations:\\
If $\subpath{C}{v_0}{c}{a'}$ is even, then the path $\lbrace c,p,b',x,b,v_0 \rbrace$ together with $\subpath{C}{v_0}{c}{a'}$ gives an odd cycle~$C'$.  
If $C'$ is induced, then $G$ contains a \low with center $a$. 
We will now see that every possible chord of $C'$ also yields a loose odd wheel:
Since $C$ is induced, there is no chord between two vertices of $C$. Further, no chord may have $x$ as an endvertex.  By~\eqref{eq:neighbours_of_b_b'} and~\eqref{eq:b_b'_not_adjacent}, the vertices $b$ and $p$ are not adjacent and $G$ contains a loose odd wheel if $b$ is adjacent to $b'$. 
If $b$ or $p$ is adjacent to a vertex of $\subpath{C}{v_0}{c}{a'}$, then either $b$ respectively $p$ has three odd neighbours on $C$, or the corresponding odd subcycle contains $x,b,b'$. In both cases, we obtain a \low with $a$.

Let $\tilde{C}$ now be a triangle with two vertices $p_1, p_2$ of $\subpath{HC(b')}{w}{a'}{a}$ and one vertex $c$ of $C-b'$. Without loss of generality, select $c$ such that $\subpath{C}{v_0}{c}{b'}$  is of minimal length.
By choosing $p_1=p$, it follows from~\eqref{eq:chord} that $G$ contains a loose odd wheel if $\subpath{C}{v_0}{c}{a'}$ is even. Otherwise, $p_1$ and $p_2$ form odd cycles $C_1, C_2$ together with the path $C-\subpath{C}{c}{b'}{a'}$; see~\eqref{eq:chord}. If one of these cycles is induced, $C-\subpath{C}{c}{b'}{a'}$, $p_1$ and $p_2$ form a loose odd wheel. 
Assume that $p_i$ ($i \in \lbrace1,2$) is adjacent to an inner vertex $u$ of $C-\subpath{C}{c}{b'}{a'}$. Choose $u$ in such a way that $\subpath{C}{u}{w}{b'}$ is of minimal length. Again apply~\eqref{eq:chord} where $p_i=p$ and $u=c$ to see that $G$ either contains a \low or the  
smallest induced odd subcycle of $C_i$ contains the vertices $w,b',p_i,u$. This odd cycle is induced and contains at least five vertices. As the cycle is not affected by the even-contraction, the graph $G'$ contains an odd hole. 
\end{proof}

Note that we can assume that $C$ is non-contractible (see~\eqref{eq:odd_cycles_noncontractible}), but do not know in which way the chords are embedded.
Figure~\ref{fig:Haus_vom_Nikolaus} shows all possible embeddings up to topological isomorphy. In Case B, there are two ways to embed the edge $bv$: the odd cycle $b,a',v,b$ may be contractible (Case I) or non-contractible (Case II). Similarly, there are four embeddings in Case C, differing in  the (non)-contractability of the cycles $b,a',\subpath{C}{a'}{v_1}{b'},v_1,b$ and $b,a',\subpath{C}{a'}{v_0}{b'},v_0,b$.

\input{fig-haus_vom_nikolaus}

\begin{claim}\label{claim:nikolaus_odd_hamilton_triangle}
Let the $a'$-$w$-path along $HC(b')$ that contains $a$ be odd and of length~$3$. If $G'$ does not contain an odd hole, then $G$ contains a loose odd wheel.
\end{claim}

\begin{proof} 
Note that in this case, the vertices $a$ and $w$ are adjacent. 

We first treat Case B, ie we suppose that $v_0=v$ is adjacent to $b$.
The path $\subpath{C}{w}{a'}{b'}$ forms a $z'$-$b'$-path together with the edges $a'b,bx,xb'$. This path and the path $V(C) - w$ form  cycles $C_{Q,x}$ and $C_{Q}$ of different parity together with $Q=\subpath{HC(w)}{b'}{z}{a}$. If the arising odd cycle is induced, $G$ contains a loose odd wheel with center $w$. 

We analyse the different types of chords:\\
First, note that $x$ cannot be an endvertex of a chord and that $b \notin V(Q)$  by~ \eqref{eq:neighbours_of_b_b'}.
If $b$ is the endvertex of a chord, the chord must be of the form $bq$ with $q\in V(Q)$. Then, $G$ contains the $5$-cycle $q,w,b',x,b,q$. As $bw, b'q \notin E(G)$  by~\eqref{eq:neighbours_of_b_b'}, either $b$ is adjacent to $b'$ or $q,w,b',x,b,q$ is induced. In both cases, $G$ contains a loose odd wheel; see~\eqref{eq:b_b'_not_adjacent} and Observation~\ref{obs:deg4->-loose_odd_wheel}.
If $a'$ is the endvertex of a chord, the chord must be of the form $a'q$ with $q\in V(Q)$. 
This gives the $5$-cycle $w,a,x,a',q,w$. If this cycle is induced, it forms a \low with $b$. We are done unless $aq$ is a chord. 

In embedding II (Figure~\ref{fig:Haus_vom_Nikolaus}), the cycle obtained from $C$ by replacing the edge $a'v$ with $a'b, bv$ separates $a$ from $q$. Thus, $a$ and $q$ cannot be adjacent. 
In embedding I, the edge $aq$ yields the contractible cycle $w,a,x,a',q,w$. 
The interior of this cycle contains the path $Q'= \subpath{HC(b')}{w}{a'}{a}$. Note that $Q'$ is odd and that we are done if $q \in V(Q')$; then $\lbrace w,b',a' \rbrace$ is a set of odd neighbours of $q$ on $C$.
Let $Q''$ be the path joining $w$ and the other neighbour of $z$ on $C$ along the Hamilton cycle of $z$ such that $q$ is not contained in $V(Q'')$.
Then, $V(C)\setminus \{z\}$ and  $Q' \cup V(C)\setminus \{z,b'\}$ form cycles $C'$ and $C''$ of different parity together with $Q''$.

 We will now consider possible chords of the associated odd cycle.
The vertices of $Q'$ are not adjacent to further vertices of $C$ or to a vertex of $Q''$ since they are contained in the interior of the  contractible cycle $w,q,a',b',w$. 
If a vertex of $Q''$ is adjacent to a vertex of $C$, take the smallest induced cycle in $C'$ and $C''$ that contains $w$. Then, again one of the cycles is odd. If the odd and induced cycle contains at least three vertices of $Q''$, then it forms a loose odd wheel with  $z$. Otherwise, the neighbour $q_1'' \in V(Q'')$ of $w$  is adjacent to a vertex of $C$ and forms an induced odd cycle. Then, either $q_1''$ has three odd neighbours on $C$ or the cycle using $Q'$ is odd. In both cases, we obtain a loose odd wheel with center vertex $b'$.

\bigskip

We now treat Case C, ie we assume that $v_0$ is not adjacent to $b$. Recall that $v_1$ is the vertex adjacent to $v_0$ on $\subpath{C}{a'}{v_0}{b'}$.
Let $\tilde{v} \not=a'$ be adjacent to $v$ on $C$.
There are two paths between $a'$ and $\tilde{v}$ along $HC(v)$. In all of the four possible embeddings of $bv_0$ and $bv_1$ (see Figure~\ref{fig:Haus_vom_Nikolaus}), one of the paths is contained in a region whose boundary does not contain $a$. Depending on the embedding, denote this path by $P=w_0, w_1, \ldots, w_k$ with $a'=w_0$ and $w_k=\tilde{v}$. 
 Let $j \in \lbrace  1, \ldots, k  \rbrace$ be the first index such that $w_j$ is adjacent to a vertex of $C-\{a',\tilde{v}\}$ (if such a chord exists, otherwise set $j=k$). Let $w'$ be the vertex closest to $v_1$ along $\subpath{C}{a'}{v_1}{v_0}$ that is adjacent to $w_j$. Note that  the cycles formed by $a'=w_0, w_1,\ldots, w_j$ together with $\subpath{C}{w'}{w}{\tilde{v}} \cup  \lbrace a,x  \rbrace$  respectively with $\subpath{C}{w'}{a'}{\tilde{v}}$ are of different parity. 
The only possible chords of the two cycles are edges between $a$ and $V(C)$.
If $ac$ is a chord with $c \in V(C)$, then $a$ has three odd neighbours on $C$ or the arising subcycle that contains $x, a'=w_0, w_1,\ldots, w_j$  is odd. This cycle yields an odd wheel with center vertex $\tilde{v}$ if $j\geq 2$. If $j=1$, the cycle formed by $a'$ and $w_1$ together with $\subpath{C}{w'}{w}{\tilde{v}} \cup  \lbrace a,x\rbrace$ is odd or $w_1$ has three odd neighbours on $C$. In both cases, $G$ contains a \low.
\end{proof}

\begin{claim}\label{claim:nikolaus_odd_hamilton}
Let the $a'$-$w$-path along $HC(b')$ that contains $a$ be odd and of length at least $5$. If $G'$ is perfect, then $G$ contains a loose odd wheel.
\end{claim}

\begin{proof}
The path $\subpath{HC(b')}{a}{w}{x}$ connecting $a$ and $w$ along $HC(b')$ is odd and has length at least $3$. 
First note that we can assume that $b$ is neither contained in $\subpath{HC(b')}{a}{w}{x}$ nor adjacent to a vertex of $HC(b')$ (see~\eqref{eq:b_b'_not_adjacent} and \eqref{eq:neighbours_of_b_b'}). 
The path $\subpath{HC(b')}{a}{w}{x}$ together with $b$ and $\subpath{C}{v_0}{w}{b'}$ forms an odd cycle $C'$. If $C'$ is induced in $G$, it forms  a loose odd wheel together with $b'$.

Thus, suppose that the cycle has chords. 
With the same arguments as for~\eqref{eq:chord} we can show the following.
Let  $p\in  V(\subpath{HC(b')}{a}{w}{x}) \setminus \lbrace a \rbrace$ and $c\in V(C) \setminus \lbrace b' \rbrace$ be adjacent in $G$ and let the path $\subpath{C}{v_0}{c}{a'}$  be even. Then, $G$ contains a loose odd wheel. 
As before we can conclude that 
\begin{equation}\label{eq:triangle_low}
\text{$G$ has no triangle with vertices of $V(\subpath{HC(b')}{a}{w}{x}) \setminus \{a\}$ and $V(C)$.}
\end{equation}
Otherwise, $G$ contains a \low.

\bigskip

Suppose there is a chord from $a$ to a vertex of $C$. 
Let $c$ be the vertex  adjacent to $a$ that is closest to $b'$ on $\subpath{C}{b'}{v_0}{a'}$. If $\subpath{C}{w}{c}{a'}$ is odd, we get an odd cycle with $a, V(\subpath{HC(b')}{a}{w}{x})$ and $V(\subpath{C}{w}{c}{a'})$. If the cycle is induced we obtain a loose odd wheel with center vertex $b'$. Otherwise, the smallest subgraph is an odd hole or a triangle, and we are done by~\eqref{eq:triangle_low}. 
Thus, we can assume that $\subpath{C}{w}{c}{a'}$ is even.

Let $R$ be the path connecting $b'$ and $z$ along $HC(w)$ such that an edge between a vertex of $R$ and $a$ always gives a contractible cycle. Let $C_{R,a}$ be the cycle formed by $V(\subpath{C}{b'}{c}{a'}) \setminus \{w\}$, $V(R)$ and $a$, and let $C_{R,x}$ be the cycle formed by $V(R), x,b$ and $\subpath{C}{v_0}{z}{a'}$. One of the cycles $C_{R,a}$ and $C_{R,x}$ is odd. If this cycle is induced, $G$ contains a loose odd wheel with $w$ as center.\\ 
We now consider possible chords in the two cycles. The vertex $b$ cannot be adjacent to a vertex of $R$ as $ac \in E(G)$. Further, $a$ cannot be adjacent to a vertex of $C\cap C_{R,a}$, by the definition of $ac$. The only possible chords are edges from $R$ to $C$, edges from $R$ to $a$ and edges from $C$ to $b$. If there is a chord between $C$ and $b$, then $b$ has three odd neighbours on $C$ (and $G$ has a \low) or there is still an induced odd cycle containing $b,x$ and $V(R)$. 
A chord from a vertex of $R$ to $C$ leads to an odd hole either in $C_{R,a}$ or in $C_{R,x}$. If the odd hole contains three vertices of $R$, we are done. Otherwise, the neighbour $r_1$ of $b'$ on $R$ is contained in a chord from $R$ to $C$. But then, either $r_1$ has three odd neighbours in $C$ or $G$ has an odd hole that contains $x$. In both cases we obtain a \low (see Observation~\ref{obs:deg4->-loose_odd_wheel}).\\ 
Assume that there is a chord from $a$ to a vertex $r$ of $R$. 
Recall that $G$ has a $3$-colouring by Observation~\ref{obs:3colours}. Thus, in $G$, the colours of the vertices on a Hamilton cycle $HC(u)$ alternate for every $u \in V(G)$. 
If $\subpath{HC(w)}{r}{b'}{z}$ is odd, this means that $r$ and $b'$ have different colours. As $a$ and $w$ are adjacent to $r$ and $b'$, the vertices $a$ and $w$ then  have the same colour. This contradicts our assumption that the $a'-w$-path along $HC(b')$ that contains $a$ is odd and shows that $\subpath{HC(w)}{r}{b'}{z}$ is even. Consequently, $C_{R,a}$ has an odd induced subcycle that contains $a$. This cycle is not affected by the even-contraction and we are done if this cycle is of length at least five. 

Assume that this odd cycle is a triangle.
Then, there is a vertex $r\in R$ with $ra \in E(G)$ and $rc \in E(G)$ where $c$ is the vertex adjacent to $a$ that is closest to $b'$ on $\subpath{C}{b'}{v_0}{a'}$ (as described in the beginning of the subcase). Choose $r$ such that its distance to $b'$ on the path $R$ is minimal. Since $\subpath{C}{b'}{c}{w}$ is even, the vertices of $\subpath{C}{b'}{c}{w}$ together with the vertices of $\subpath{R}{b'}{r}{z}$ form an odd cycle $C_{r,R}$. If this cycle is induced, it yields an odd wheel with center vertex $w$. The vertices of $C_{r,R}\cap C$ cannot form chords, since $C$ is induced. The vertices of $R-r$ cannot be adjacent to further vertices of $C_{r,R}\cap C$ since they lie in a contractible cycle which is closed by $ra$. If there is a chord from $r$ to a vertex  of $C_{r,R}\cap C$, then either this edge or the edge $ac$ form a contractible cycle that includes a part of the Hamilton cycle of $z$ --- the path $R'$. Thus, in that case no vertex of $R'$ is adjacent to $a$ or $b$. We obtain two cycles of different parity: the cycle with vertices $V(R')$, $b'$, $x$, $b$ and vertices of $C$, and the cycle with vertices $V(R')$, $b'$, $a$ and vertices of $C$. 
Both cycles can contain chords from $R'$ to $C$. But the, one of the induced cycles that includes $b'$ is odd and of length at least $5$. Consequently, we obtain a loose odd wheel.

\bigskip

Finally, if there is no chord from $a$ to a vertex of $C$, the only possible chords that can occur in $C'$ are edges from a vertex $r \in V(\subpath{HC(b')}{a}{w}{x})$ or from $b$ to a vertex of $C$. If there is a chord $bc$ with $c\in V(C)$, then $b$ has either three odd neighbours on $C$ or there is still an odd cycle of length at least $5$ containing $V(\subpath{HC(b')}{a}{w}{x}),a$ and $b$. 

Suppose there is a chord from $r \in V(\subpath{HC(b')}{a}{w}{x})$ to a vertex of $C$.
If the induced cycle in $C'$ that contains $a$ and $b$ is even, there is an odd cycle with vertices of $\subpath{HC(b')}{a}{w}{x}$ and $C$. Then, as we have seen  in~\eqref{eq:triangle_low}, $G$ contains a loose odd wheel. If the induced cycle in $C'$ that contains $a$ and $b$ is odd and contains at least three vertices of $\subpath{HC(b')}{a}{w}{x}$, then it forms a loose odd wheel with center vertex $b'$.
Otherwise, there is a vertex $r_1\in V(\subpath{HC(b')}{a}{w}{x})$ with $ar_1 \in E(G)$ that is adjacent to a vertex $c_r$ of $C$ such that $\subpath{C}{c_r}{w}{b'}$ is odd. Choose $c_r$ such that $\subpath{C}{c_r}{w}{b'}$ is of maximal length. The path $\subpath{C}{c_r}{v_0}{b'}$ is also odd and forms an odd cycle together with $b,a$ and $r_1$. The vertex $r_1$ cannot be contained in a chord of this cycle by choice of $c_r$ and by~\eqref{eq:neighbours_of_b_b'}. Further, there is no chord from $a$ to a vertex of $C$. The vertex $b$ can be adjacent to vertices of $C$. But then, $b$ either has three odd neighbours or there is an odd hole that contains $b,a$ and $r_1$. If the hole is contractible, Theorem~\ref{thm:cycles_triang}  assures that $G$ contains a loose odd wheel. Otherwise, the edge $r_1c_r$ closes a contractible cycle containing a part of the Hamilton cycle of $w$ --- the path $R''$. Using the fact that the vertices $b$ and $a'$ do not lie in the interior of this cycle, we get a loose odd wheel with $R''$, vertices of $C$ and $a'$ (respectively $\{x,b\}$) similar to the cases we have seen before.  
\end{proof}

This finishes the proof of Lemma~\ref{lem:a'b'}.

\end{proof}

%% file: fig-three_possible_triangles.tex
\begin{figure}[bht]
\begin{center}
\begin{tikzpicture}[scale = .9]

\begin{scope}[shift={(-6,0)}]
\node[novertex] (labelA) at (0,0){$\mathbf{G}$};
\end{scope}

\begin{scope}[shift={(-6,-3.5)}]
\node[novertex] (labelA) at (0,0){$\mathbf{G'}$};
\end{scope}

\begin{scope}[shift={(-3.5,0)}]
\node[invvertex] (cr) at (1.6,0){};
\node[invvertex] (cl) at (-1.6,0){};

\node[novertex] (labelA) at (0,-2){\textbf{Case A}};
\node[hvertex] (z) at (-1.2,-0.3){};
\node[novertex] (lz) at (-1.4,-0.5){$z$};
\node[hvertex] (w) at (- 0.8,0){};
\node[novertex] (lw) at (-0.75,-0.2){$w$};
\node[hvertex] (b2) at (-0.4,0){};
\node[novertex] (lb2) at (-0.4,-0.2){$b'$};
\node[hvertex] (a2) at (0.4,0){};
\node[novertex] (la2) at (0.65,-0.15){$a'$};

\node[novertex] (lC) at (-.2,-0.9){$C$};

\node[hvertex] (a) at (-0.4,0.9){};
\node[novertex] (la) at (-0.57,0.69){$a$};
\node[hvertex] (b) at (0.4,0.9){};
\node[novertex] (lb) at (0.6,0.9){$b$};

\node[hvertex] (x) at (0,0.45){};
\node[novertex] (lx) at (0.25,0.45){$x$};

\draw[hedge] (z) -- (w);
\draw[hedge] (w) -- (b2);
\draw[hedge] (b2) -- (a2);
\draw [decorate, decoration={snake,amplitude=.4mm,segment length=2mm,post length=0mm}] (-1.2,-0.3) arc (-180:20:0.8);

\draw[hedge] (a) -- (b);
\draw[hedge] (a) -- (b2);
\draw[hedge] (b) -- (a2);

\usetkzobj{all} 
\tkzDefPoint(0.4,0.9){B}
\tkzDefPoint(-0.7,0.9){H}
\tkzDefPoint(-1.2,-0.3){W}
\tkzCircumCenter(B,H,W)
\tkzGetPoint{0}
\tkzDrawArc[style=thick,color=black](0,B)(W)

\draw[hedge] (b) -- (x);
\draw[hedge] (b2) -- (x);
\draw[hedge] (a) -- (x);
\draw[hedge] (a2) -- (x);
\end{scope}

\begin{scope}[shift={(-3.5,-3.5)}]

\node[hvertex] (z) at (-1.2,-0.3){};
\node[novertex] (lz) at (-1.4,-0.5){$z$};
\node[hvertex] (w) at (- 0.8,0){};
\node[novertex] (lw) at (-0.75,-0.2){$w$};
\node[hvertex] (y) at (-0.4,0){};
\node[novertex] (ly) at (-0.4,-0.2){$y$};
\node[hvertex] (a2) at (0.4,0){};
\node[novertex] (la2) at (0.65,-0.15){$a'$};

\node[novertex] (lC) at (-.2,-0.9){$K$};

\node[hvertex] (a) at (-0.4,0.9){};
\node[novertex] (la) at (-0.57,0.69){$a$};

\draw[hedge] (z) -- (w);
\draw[hedge] (w) -- (y);
\draw[hedge] (a2) -- (y);
\draw[hedge] (a) -- (y);
\draw [decorate, decoration={snake,amplitude=.4mm,segment length=2mm,post length=0mm}] (-1.2,-0.3) arc (-180:20:0.8);

\usetkzobj{all} 
\tkzDefPoint(-.4,0){B}
\tkzDefPoint(-0.7,0.4){H}
\tkzDefPoint(-1.2,-0.3){W}
\tkzCircumCenter(B,H,W)
\tkzGetPoint{0}
\tkzDrawArc[style=thick,color=black](0,B)(W)

\end{scope}

\begin{scope}[shift={(0,0)}]
\node[invvertex] (cr) at (1.6,0){};
\node[invvertex] (cl) at (-1.6,0){};

\node[novertex] (labelA) at (0,-2){\textbf{Case B}};
\node[hvertex] (z) at (-1.2,-0.3){};
\node[novertex] (lz) at (-1.4,-0.5){$z$};
\node[hvertex] (w) at (- 0.8,0){};
\node[novertex] (lw) at (-0.75,-0.2){$w$};
\node[hvertex] (b2) at (-0.4,0){};
\node[novertex] (lb2) at (-0.4,-0.2){$b'$};
\node[hvertex] (a2) at (0.4,0){};
\node[novertex] (la2) at (0.39,-0.3){$a'$};
\node[hvertex] (v) at (0.8,-0.3){};
\node[novertex] (lv) at (1,-0.4){$v$};

\node[novertex] (lC) at (0,-1.1){$C$};

\node[hvertex] (a) at (-0.4,0.9){};
\node[novertex] (la) at (-0.4,1.08){$a$};
\node[hvertex] (b) at (0.4,0.9){};
\node[novertex] (lb) at (0.3,1.08){$b$};

\node[hvertex] (x) at (0,0.45){};
\node[novertex] (lx) at (0.25,0.45){$x$};

\draw[hedge] (z) -- (w);
\draw[hedge] (w) -- (b2);
\draw[hedge] (b2) -- (a2);
\draw[hedge] (a2) -- (v);
\draw [decorate, decoration={snake,amplitude=.4mm,segment length=2mm,post length=0mm}] (-1.2,-0.3) arc (-180:-3:1);

\draw[hedge] (a) -- (b);
\draw[hedge] (a) -- (b2);
\draw[hedge] (b) -- (a2);
\draw[hedge] (b) -- (v);

\draw[hedge] (b) -- (x);
\draw[hedge] (b2) -- (x);
\draw[hedge] (a) -- (x);
\draw[hedge] (a2) -- (x);
\end{scope}

\begin{scope}[shift={(0,-3.5)}]

\node[hvertex] (z) at (-1.2,-0.3){};
\node[novertex] (lz) at (-1.4,-0.5){$z$};
\node[hvertex] (w) at (- 0.8,0){};
\node[novertex] (lw) at (-0.75,-0.2){$w$};
\node[hvertex] (y) at (-0.4,0){};
\node[novertex] (ly) at (-0.4,-0.2){$y$};
\node[hvertex] (a2) at (0.4,0){};
\node[novertex] (la2) at (0.39,-0.3){$a'$};
\node[hvertex] (v) at (0.8,-0.3){};
\node[novertex] (lv) at (1,-0.4){$v$};

\node[novertex] (lC) at (0,-1.1){$K$};

\node[hvertex] (a) at (-0.4,0.9){};
\node[novertex] (la) at (-0.57,0.69){$a$};

\draw[hedge] (z) -- (w);
\draw[hedge] (w) -- (y);
\draw[hedge] (a2) -- (y);
\draw[hedge] (a2) -- (v);
\draw[hedge] (a) -- (y);
\draw [decorate, decoration={snake,amplitude=.4mm,segment length=2mm,post length=0mm}] (-1.2,-0.3) arc (-180:-3:1);

\usetkzobj{all} 
\tkzDefPoint(.8,-0.3){B}
\tkzDefPoint(0.7,0.4){H}
\tkzDefPoint(-.4,0){W}
\tkzCircumCenter(B,H,W)
\tkzGetPoint{0}
\tkzDrawArc[style=thick,color=black](0,B)(W)

\end{scope}

\begin{scope}[shift={(3.5,0)}]

\node[novertex] (labelC) at (0,-2){\textbf{Case C}};
\node[hvertex] (z) at (-1.2,-.3){};
\node[novertex] (lz) at (-1.4,-0.5){$w$};
\node[hvertex] (b2) at (- 0.8,0){};
\node[novertex] (lb2) at (-0.8,-0.25){$b'$};
\node[hvertex] (a2) at (0,0){};
\node[novertex] (la2) at (0,-0.25){$a'$};
\node[hvertex] (v1) at (0.8,-0.3){};
\node[hvertex] (v2) at (1.2,-0.5){};

\node[hvertex] (a) at (-0.8,0.9){};
\node[novertex] (la) at (-0.8,1.08){$a$};
\node[hvertex] (b) at (0,0.9){};
\node[novertex] (lb) at (-0.1,1.08){$b$};

\node[hvertex] (x) at (-0.4,0.45){};
\node[novertex] (lx) at (-0.15,0.45){$x$};

\draw[hedge] (z) -- (b2);
\draw[hedge] (b2) -- (a2);
\draw[decorate,decoration={snake,amplitude=.4mm,segment length=2mm,post length=0mm}] (a2) -- (v1);
\draw[hedge] (v1) -- (v2);

\draw[hedge] (a) -- (b);
\draw[hedge] (a) -- (b2);
\draw[hedge] (b) -- (a2);
\draw[hedge] (b) -- (v1);
\draw[hedge] (b) -- (v2);
\draw[hedge] (b) -- (x);
\draw[hedge] (b2) -- (x);
\draw[hedge] (a) -- (x);
\draw[hedge] (a2) -- (x);

\node[novertex] (lC) at (0,-1.3){$C$};
\draw [decorate, decoration={snake,amplitude=.4mm,segment length=2mm,post length=0mm}] (-1.2,-0.3) arc (-180:-10:1.2);

\end{scope}

\begin{scope}[shift={(3.5,-3.5)}]

\node[hvertex] (z) at (-1.2,-.3){};
\node[novertex] (lz) at (-1.4,-0.5){$w$};
\node[hvertex] (y) at (- 0.8,0){};
\node[novertex] (ly) at (-0.8,-0.25){$y$};
\node[hvertex] (a2) at (0,0){};
\node[novertex] (la2) at (0,-0.25){$a'$};
\node[hvertex] (v1) at (0.8,-0.3){};
\node[hvertex] (v2) at (1.2,-0.5){};

\node[hvertex] (a) at (-0.8,0.9){};
\node[novertex] (la) at (-0.97,0.69){$a$};

\draw[hedge] (z) -- (y);
\draw[hedge] (y) -- (a2);
\draw[hedge] (y) -- (a);
\draw[decorate,decoration={snake,amplitude=.4mm,segment length=2mm,post length=0mm}] (a2) -- (v1);
\draw[hedge] (v1) -- (v2);

\node[novertex] (lC) at (0,-1.3){$K$};

\draw [decorate, decoration={snake,amplitude=.4mm,segment length=2mm,post length=0mm}] (-1.2,-0.3) arc (-180:-10:1.2);

\usetkzobj{all} 
\tkzDefPoint(-.8,0){B}
\tkzDefPoint(0.4,0.4){H}
\tkzDefPoint(.8.,-.3){W}
\tkzCircumCenter(B,H,W)
\tkzGetPoint{0}
\tkzDrawArc[style=thick,color=black](0,W)(B)

\usetkzobj{all} 
\tkzDefPoint(-.8,0){B}
\tkzDefPoint(0.7,0.4){H}
\tkzDefPoint(1.2.,-.5){W}
\tkzCircumCenter(B,H,W)
\tkzGetPoint{0}
\tkzDrawArc[style=thick,color=black](0,W)(B)

\end{scope}

\end{tikzpicture}
\end{center}

\caption{The three possible constellations in $G$ which lead to triangles in $K$}
\label{fig:three_possible_triangles}
\end{figure}
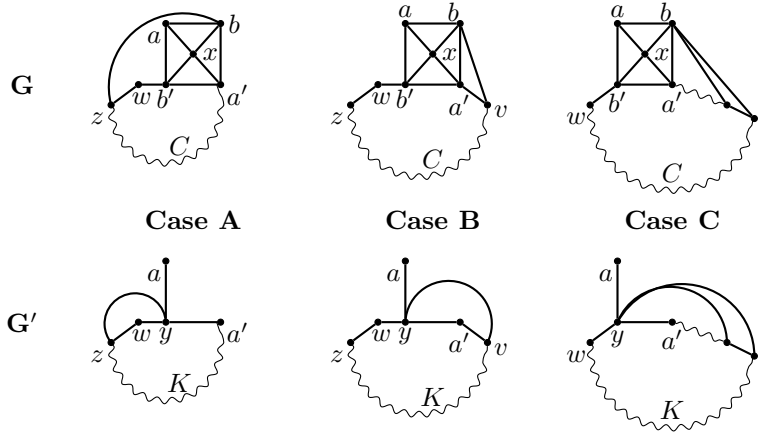

%% file: fig-haus_vom_nikolaus.tex
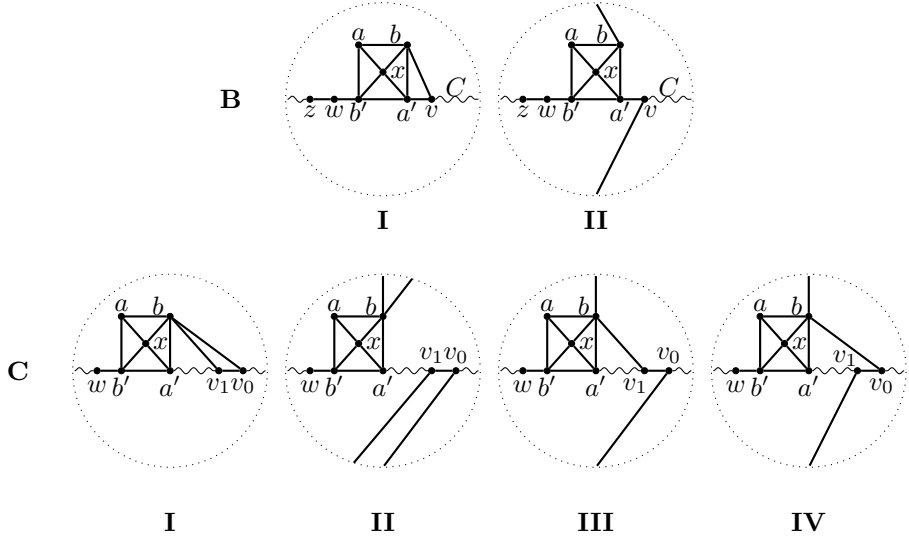
\begin{figure}[htb]
\begin{center}
\begin{tikzpicture}[scale = .8]

\begin{scope}[shift={(3.5,-4.5)}]

\draw[dotted] (0,0) circle (1.6cm);
\node[invvertex] (cr) at (1.6,0){};
\node[invvertex] (cl) at (-1.6,0){};

\node[novertex] (labelB) at (-2.5,0){\textbf{B}};
\node[hvertex] (z) at (-1.2,0){};
\node[novertex] (lz) at (-1.2,-0.2){$z$};
\node[hvertex] (w) at (- 0.8,0){};
\node[novertex] (lw) at (-0.8,-0.2){$w$};
\node[hvertex] (b2) at (-0.4,0){};
\node[novertex] (lb2) at (-0.4,-0.2){$b'$};
\node[hvertex] (a2) at (0.4,0){};
\node[novertex] (la2) at (0.4,-0.2){$a'$};
\node[hvertex] (v) at (0.8,0){};
\node[novertex] (lv) at (0.8,-0.2){$v$};

\node[novertex] (lC) at (1.2,0.2){$C$};

\node[hvertex] (a) at (-0.4,0.9){};
\node[novertex] (la) at (-0.4,1.08){$a$};
\node[hvertex] (b) at (0.4,0.9){};
\node[novertex] (lb) at (0.2,1.08){$b$};

\node[hvertex] (x) at (0,0.45){};
\node[novertex] (lx) at (0.25,0.45){$x$};

\draw[decorate,decoration={snake,amplitude=.4mm,segment length=2mm,post length=0mm}] (cl) -- (z);
\draw[hedge] (z) -- (w);
\draw[hedge] (w) -- (b2);
\draw[hedge] (b2) -- (a2);
\draw[hedge] (a2) -- (v);
\draw[decorate,decoration={snake,amplitude=.4mm,segment length=2mm,post length=0mm}] (v) -- (cr);

\draw[hedge] (a) -- (b);
\draw[hedge] (a) -- (b2);
\draw[hedge] (b) -- (a2);
\draw[hedge] (b) -- (v);
\draw[hedge] (b) -- (x);
\draw[hedge] (b2) -- (x);
\draw[hedge] (a) -- (x);
\draw[hedge] (a2) -- (x);
\node[novertex] (l) at (0,-2){\textbf{I}};

\end{scope}

\begin{scope}[shift={(7,-4.5)}]
\draw[dotted] (0,0) circle (1.6cm);
\node[invvertex] (cr) at (1.6,0){};
\node[invvertex] (cl) at (-1.6,0){};
\node[invvertex] (cu) at (0,1.6){};
\node[invvertex] (cd) at (0,-1.6){};

\node[hvertex] (z) at (-1.2,0){};
\node[novertex] (lz) at (-1.2,-0.2){$z$};
\node[hvertex] (w) at (- 0.8,0){};
\node[novertex] (lw) at (-0.8,-0.2){$w$};
\node[hvertex] (b2) at (-0.4,0){};
\node[novertex] (lb2) at (-0.4,-0.2){$b'$};
\node[hvertex] (a2) at (0.4,0){};
\node[novertex] (la2) at (0.4,-0.2){$a'$};
\node[hvertex] (v) at (0.8,0){};
\node[novertex] (lv) at (0.9,-0.2){$v$};
\node[novertex] (lC) at (1.2,0.2){$C$};

\node[hvertex] (a) at (-0.4,0.9){};
\node[novertex] (la) at (-0.4,1.08){$a$};
\node[hvertex] (b) at (0.4,0.9){};
\node[novertex] (lb) at (0.1,1.08){$b$};

\node[hvertex] (x) at (0,0.45){};
\node[novertex] (lx) at (0.25,0.45){$x$};
\draw[decorate,decoration={snake,amplitude=.4mm,segment length=2mm,post length=0mm}] (cl) -- (z);
\draw[hedge] (z) -- (w);
\draw[hedge] (w) -- (b2);
\draw[hedge] (b2) -- (a2);
\draw[hedge] (a2) -- (v);
\draw[decorate,decoration={snake,amplitude=.4mm,segment length=2mm,post length=0mm}] (v) -- (cr);

\draw[hedge] (a) -- (b);
\draw[hedge] (a) -- (b2);
\draw[hedge] (b) -- (a2);
\draw[hedge] (b) -- (cu);
\draw[hedge] (v) -- (cd);
\draw[hedge] (b) -- (x);
\draw[hedge] (b2) -- (x);
\draw[hedge] (a) -- (x);
\draw[hedge] (a2) -- (x);
\node[novertex] (l) at (0,-2){\textbf{II}};
\end{scope}

\begin{scope}[shift={(0,-9)}]

\tikzstyle{greyedge}=[thick, color=hellgrau]

\draw[dotted] (0,0) circle (1.6cm);
\node[invvertex] (cr) at (1.6,0){};
\node[invvertex] (cl) at (-1.6,0){};

\node[invvertex] (cu) at (0,1.6){};
\node[invvertex] (cd) at (0,-1.6){};
\node[invvertex] (cur) at (0.5,1.5){};
\node[invvertex] (cdl) at (-0.5,-1.5){};

\node[novertex] (labelC) at (-2.5,0){\textbf{C}};
\node[hvertex] (z) at (-1.2,0){};
\node[novertex] (lz) at (-1.2,-0.25){$w$};
\node[hvertex] (b2) at (- 0.8,0){};
\node[novertex] (lb2) at (-0.8,-0.25){$b'$};
\node[hvertex] (a2) at (0,0){};
\node[novertex] (la2) at (0,-0.25){$a'$};
\node[hvertex] (v1) at (0.8,0){};
\node[novertex] (lv1) at (0.8,-0.25){$v_1$};
\node[hvertex] (v2) at (1.2,0){};
\node[novertex] (lv2) at (1.2,-0.25){$v_0$};

\node[hvertex] (a) at (-0.8,0.9){};
\node[novertex] (la) at (-0.8,1.08){$a$};
\node[hvertex] (b) at (0,0.9){};
\node[novertex] (lb) at (-0.2,1.08){$b$};

\node[hvertex] (x) at (-0.4,0.45){};
\node[novertex] (lx) at (-0.15,0.45){$x$};

\draw[decorate,decoration={snake,amplitude=.4mm,segment length=2mm,post length=0mm}] (cl) -- (z);
\draw[hedge] (z) -- (b2);
\draw[hedge] (b2) -- (a2);
\draw[decorate,decoration={snake,amplitude=.4mm,segment length=2mm,post length=0mm}] (a2) -- (v1);
\draw[hedge] (v1) -- (v2);
\draw[decorate,decoration={snake,amplitude=.4mm,segment length=2mm,post length=0mm}] (cr) -- (v2);

\draw[hedge] (a) -- (b);
\draw[hedge] (a) -- (b2);
\draw[hedge] (b) -- (a2);
\draw[hedge] (b) -- (v1);
\draw[hedge] (b) -- (v2);
\draw[hedge] (b) -- (x);
\draw[hedge] (b2) -- (x);
\draw[hedge] (a) -- (x);
\draw[hedge] (a2) -- (x);

\node[novertex] (l) at (0,-2.5){\textbf{I}};

\end{scope}

\begin{scope}[shift={(3.5,-9)}]
\draw[dotted] (0,0) circle (1.6cm);
\node[invvertex] (cr) at (1.6,0){};
\node[invvertex] (cl) at (-1.6,0){};

\node[invvertex] (cu) at (0,1.6){};
\node[invvertex] (cd) at (0,-1.6){};
\node[invvertex] (cur) at (0.5,1.55){};
\node[invvertex] (cdl) at (-0.5,-1.55){};

\node[hvertex] (z) at (-1.2,0){};
\node[novertex] (lz) at (-1.2,-0.25){$w$};
\node[hvertex] (b2) at (- 0.8,0){};
\node[novertex] (lb2) at (-0.8,-0.25){$b'$};
\node[hvertex] (a2) at (0,0){};
\node[novertex] (la2) at (0,-0.25){$a'$};
\node[hvertex] (v1) at (0.8,0){};
\node[novertex] (lv1) at (0.8,.25){$v_1$};
\node[hvertex] (v2) at (1.2,0){};
\node[novertex] (lv2) at (1.2,.25){$v_0$};

\node[hvertex] (a) at (-0.8,0.9){};
\node[novertex] (la) at (-0.8,1.08){$a$};
\node[hvertex] (b) at (0,0.9){};
\node[novertex] (lb) at (-0.2,1.08){$b$};

\node[hvertex] (x) at (-0.4,0.45){};
\node[novertex] (lx) at (-0.15,0.45){$x$};

\draw[decorate,decoration={snake,amplitude=.4mm,segment length=2mm,post length=0mm}] (cl) -- (z);
\draw[hedge] (z) -- (b2);
\draw[hedge] (b2) -- (a2);
\draw[decorate,decoration={snake,amplitude=.4mm,segment length=2mm,post length=0mm}] (a2) -- (v1);
\draw[hedge] (v1) -- (v2);
\draw[decorate,decoration={snake,amplitude=.4mm,segment length=2mm,post length=0mm}] (cr) -- (v2);

\draw[hedge] (a) -- (b);
\draw[hedge] (a) -- (b2);
\draw[hedge] (b) -- (a2);
\draw[hedge] (b) -- (cu);
\draw[hedge] (b) -- (cur);
\draw[hedge] (v1) -- (cdl);
\draw[hedge] (v2) -- (cd);

\draw[hedge] (b) -- (x);
\draw[hedge] (b2) -- (x);
\draw[hedge] (a) -- (x);
\draw[hedge] (a2) -- (x);
\node[novertex] (l) at (0,-2.5){\textbf{II}};
\end{scope}

\begin{scope}[shift={(7,-9)}]
\draw[dotted] (0,0) circle (1.6cm);
\node[invvertex] (cr) at (1.6,0){};
\node[invvertex] (cl) at (-1.6,0){};

\node[invvertex] (cu) at (0,1.6){};
\node[invvertex] (cd) at (0,-1.6){};
\node[invvertex] (cur) at (0.5,1.5){};
\node[invvertex] (cdl) at (-0.5,-1.5){};

\node[hvertex] (z) at (-1.2,0){};
\node[novertex] (lz) at (-1.2,-0.25){$w$};
\node[hvertex] (b2) at (- 0.8,0){};
\node[novertex] (lb2) at (-0.8,-0.25){$b'$};
\node[hvertex] (a2) at (0,0){};
\node[novertex] (la2) at (0,-0.25){$a'$};
\node[hvertex] (v1) at (0.8,0){};
\node[novertex] (lv1) at (0.65,-0.25){$v_1$};
\node[hvertex] (v2) at (1.2,0){};
\node[novertex] (lv2) at (1.2,.25){$v_0$};

\node[hvertex] (a) at (-0.8,0.9){};
\node[novertex] (la) at (-0.8,1.08){$a$};
\node[hvertex] (b) at (0,0.9){};
\node[novertex] (lb) at (-0.2,1.08){$b$};

\node[hvertex] (x) at (-0.4,0.45){};
\node[novertex] (lx) at (-0.15,0.45){$x$};

\draw[decorate,decoration={snake,amplitude=.4mm,segment length=2mm,post length=0mm}] (cl) -- (z);
\draw[hedge] (z) -- (b2);
\draw[hedge] (b2) -- (a2);
\draw[decorate,decoration={snake,amplitude=.4mm,segment length=2mm,post length=0mm}] (a2) -- (v1);
\draw[hedge] (v1) -- (v2);
\draw[decorate,decoration={snake,amplitude=.4mm,segment length=2mm,post length=0mm}] (cr) -- (v2);

\draw[hedge] (a) -- (b);
\draw[hedge] (a) -- (b2);
\draw[hedge] (b) -- (a2);
\draw[hedge] (b) -- (v1);
\draw[hedge] (b) -- (cu);
\draw[hedge] (v2) -- (cd);

\draw[hedge] (b) -- (x);
\draw[hedge] (b2) -- (x);
\draw[hedge] (a) -- (x);
\draw[hedge] (a2) -- (x);
\node[novertex] (l) at (0,-2.5){\textbf{III}};
\end{scope}

\begin{scope}[shift={(10.5,-9)}]
\draw[dotted] (0,0) circle (1.6cm);
\node[invvertex] (cr) at (1.6,0){};
\node[invvertex] (cl) at (-1.6,0){};

\node[invvertex] (cu) at (0,1.6){};
\node[invvertex] (cd) at (0,-1.6){};
\node[invvertex] (cur) at (0.5,1.5){};
\node[invvertex] (cdl) at (-0.5,-1.5){};

\node[hvertex] (z) at (-1.2,0){};
\node[novertex] (lz) at (-1.2,-0.25){$w$};
\node[hvertex] (b2) at (- 0.8,0){};
\node[novertex] (lb2) at (-0.8,-0.25){$b'$};
\node[hvertex] (a2) at (0,0){};
\node[novertex] (la2) at (0,-0.25){$a'$};
\node[hvertex] (v1) at (0.8,0){};
\node[novertex] (lv1) at (0.6,0.2){$v_1$};
\node[hvertex] (v2) at (1.2,0){};
\node[novertex] (lv2) at (1.2,-0.25){$v_0$};

\node[hvertex] (a) at (-0.8,0.9){};
\node[novertex] (la) at (-0.8,1.08){$a$};
\node[hvertex] (b) at (0,0.9){};
\node[novertex] (lb) at (-0.2,1.08){$b$};

\node[hvertex] (x) at (-0.4,0.45){};
\node[novertex] (lx) at (-0.15,0.45){$x$};

\draw[decorate,decoration={snake,amplitude=.4mm,segment length=2mm,post length=0mm}] (cl) -- (z);
\draw[hedge] (z) -- (b2);
\draw[hedge] (b2) -- (a2);
\draw[decorate,decoration={snake,amplitude=.4mm,segment length=2mm,post length=0mm}] (a2) -- (v1);
\draw[hedge] (v1) -- (v2);
\draw[decorate,decoration={snake,amplitude=.4mm,segment length=2mm,post length=0mm}] (cr) -- (v2);

\draw[hedge] (a) -- (b);
\draw[hedge] (a) -- (b2);
\draw[hedge] (b) -- (a2);
\draw[hedge] (b) -- (cu);
\draw[hedge] (b) -- (v2);
\draw[hedge] (v1) -- (cd);

\draw[hedge] (b) -- (x);
\draw[hedge] (b2) -- (x);
\draw[hedge] (a) -- (x);
\draw[hedge] (a2) -- (x);

\node[novertex] (l) at (0,-2.5){\textbf{IV}};
\end{scope}

\end{tikzpicture}
\end{center}

\caption{Different ways of embedding the chords of $K$. (Opposite points on the dotted cycle are identified.)}
\label{fig:Haus_vom_Nikolaus}
\end{figure}

%% file: subsec-bprime.tex
\begin{lemma} \label{lem:b'}
Let $G$ be a nice Eulerian triangulation. Let $G'$ be obtained from $G$ by an even-contraction at $x$ and let $\{x,a,a',b,b'\}\cap V(C)$ equal $ \{b'\}$. If $G$ contains an odd hole $C$ and $G'$ is perfect, then $G$ contains a loose odd wheel.
\end{lemma}

\begin{proof}
Suppose that $G'$ is perfect. Then, every odd induced subcycle of $\gamma(C)=K$ in $G'$ is a triangle. As all chords of  $\gamma(C)$ contain $b$, there are two possibilities: the triangle may contain one or two chords.

If the triangle contains one chord, then $b$ is adjacent to a vertex $v_1$ of distance~$2$ from $b'$ on $C$. The vertices $y,v_1$ and the common neighbour $v_2$ of $v_1$ and $y$ on $K$ now give a triangle. In this case, the vertices $b',x,b,v_1$ and $v_2$ form an odd cycle in $G$. Since $C$ is induced and $b$ is not adjacent to $v_2$ (see~\eqref{eq:neighbours_of_b_b'}),  this cycle is induced. It forms a \low together with  $a$.

If the triangle is of the form $y,c_1,c_2,y$ where $yc_1$ and $yc_2$ are chords of $K$, then, $bc_1$ and $bc_2$ are edges in $G$. Without loss of generality, we choose $c_1$ and $c_2$ in such a way that $\subpath{C}{b'}{c_1}{c_2}$ is of minimal length. Note that this choice implies that 
\begin{equation} \label{eq:odd_segment_b'}
\subpath{C}{b'}{c_1}{c_2} \text{ is odd.}
\end{equation}
Otherwise, $\subpath{C}{b'}{c_1}{c_2}$ forms an odd cycle together with $b$ and $x$ which has an odd subcycle of length at least $5$.

\input{fig-bprime}

There are three ways of embedding the edges $bc_1$ and $bc_2$ (up to switching the vertices $c_1$ and $c_2$ and the vertices $a$ and $a'$ and up to topological isomorphy). They are shown in Figure~\ref{fig:bprime} and treat the (non-)contractabilty of the cycles $ \subpath{C}{b'}{c_1}{c_2} \cup \lbrace  b'a, ab, bc_1 \rbrace$ and  $ \subpath{C}{b'}{c_2}{c_1} \cup \lbrace  b'a, ab, bc_2 \rbrace$. Note that $C$ is non-contractible by~Theorem~\ref{thm:cycles_triang}.
If $a$ and $a'$ have no common neighbour besides $\lbrace b,b',x \rbrace$, then we can switch the roles of $b,b'$ with the roles of $a,a'$. As we have seen in~\eqref{eq:a_a'_on_C}, this means $C$ is not affected by the even-contraction. Thus, we can assume that $a$ and $a'$ have a common neighbour besides $b,x,b'$.  

\bigskip 
Suppose that the graph is embedded as in III (see Figure~\ref{fig:bprime}). The cycle $\subpath{C}{b'}{c_1}{c_2} \cup \lbrace  b'x, xb, bc_1 \rbrace$ is contractible and contains $a$ in its interior. The cycle $\subpath{C}{b'}{c_2}{c_1} \cup \lbrace  b'x, xb, bc_2 \rbrace$ is contractible and contains $a'$ in its interior. For this reason, a common neighbour of $a$ and $a'$ must be contained in the intersection of both cycles, ie in $\lbrace  b,b' \rbrace$. Therefore, it is not possible that  $a$ and $a'$ have a further common neighbour if the graph is embedded as in III.

\bigskip 
Suppose that the graph is embedded as in I (see Figure~\ref{fig:bprime}). 
Similar to the arguments before, one can see that all common neighbours of $a$ and $a'$ besides besides $b,x,b'$ lie on the path $\subpath{C}{b'}{c_1}{c_2}$. 

First assume that there is a vertex $v'$ in $\subpath{C}{b'}{c_1}{c_2}$ that is adjacent to $a$ and that there is a vertex $w$ on  $\subpath{C}{v'}{b'}{c_1}$ that is not adjacent to $a$.
There are two paths along $HC(w)$ that connect $v'$ and $z$. Let $P'$ be the path that lies in the contractible cycle formed by $\subpath{C}{b'}{v'}{c_1}$ and $a$. 
Consider the cycle $C_1$ formed by $P'$ together with $\subpath{C}{b'}{z}{w}$, the vertices $x,b$ and the path $\subpath{C}{c_1}{v'}{w}$. 
Further, consider the cycle $C_2$ formed by $P'$ together with $\subpath{C}{v'}{z}{w}$. 
Notice that either $C_1$ or $C_2$ is odd. If the odd cycle is induced, we get a loose odd wheel with center vertex $x$ respectively $w$. 
Assume that the odd cycle contains chords. If there is a chord from $b$ to a vertex of $\subpath{C}{b'}{c_1}{c_2}$ (in the case if $C_1$ is odd),  the vertex $b$ has three odd neighbours on $C$ or there still remains an induced odd cycle containing $b,x$. If there is a chord from a vertex $p$ of $P'$ to a vertex $q$ of $C$, consider the induced cycle $C_i'$ in $C_i$ which contains $v'$ for $i=1,2$. Either $C_1'$ or $C_2'$ is odd. If $C_1'$ is an odd induced cycle, we get a loose odd wheel by Observation~\ref{obs:deg4->-loose_odd_wheel}.
Otherwise, $C_2'$ is an odd cycle. If $w$ has at least three neighbours on $C_2'$ we get a loose odd wheel. If $w$ has only two neighbours on $C_2'$, it directly follows that $p$ has three odd neighbours on $C$, namely $v',w$ and $q$.

Now assume, that no vertex $v'$ as above exists. Then, either $a$ has three odd neighbours (and we get a \low) or the only vertex in the interior of $\subpath{C}{b'}{c_1}{c_2}$ that is adjacent to $a$ equals the neighbour $v$ of $b'$. 
But then $b, a$ and $\subpath{C}{v}{c_1}{c_2}$ form an odd cycle $C'$. Note that all chords of $C'$ have $b$ as an endvertex. Further, either $b$ has three odd neighbours on $C$ (and $G$ contains a \low) or all chords  $bw$  satisfy that $\subpath{C}{c_1}{w}{c_2}$ is of even length. In the second case, every odd induced subcycle of $C'$ contains $a$ and $b$. If the largest induced odd subcycle of $C'$ is of length at least $5$, then $G$ has a further odd hole which contains $a$ and $b$. Then, $G$ contains a \low by Lemma~\ref{lem:a'b'}. Otherwise, $b$ is adjacent to $v$  which  contradicts~\eqref{eq:neighbours_of_b_b'}.

\bigskip

Next, suppose that the graph is embedded as in II (see Figure~\ref{fig:bprime}).
In this case,  one can see that $c_1$ or $c_2$ are the only possible common neighbours of $a$ and $a'$ besides $b,x,b'$. 
If $c_1$ is adjacent to $a$ and $a'$, we obtain an odd cycle formed by $a'$ and $ \subpath{C}{b'}{c_1}{c_2}$. This means that $G$ has an odd hole which contains $a'$ and $b'$. 
If $c_2$ is the common neighbour of $a$ and $a'$, we obtain the cycle consisting of $\subpath{C}{b'}{c_2}{c_1}$ and $a$ that contains $a$ and $b'$. Both cycles are odd (see~\eqref{eq:odd_segment_b'}), of length at least $5$ (see~\eqref{eq:neighbours_of_b_b'}), and induced (as $C$ is induced and the embedding does not allow any chords).
With~Lemma~\ref{lem:a'b'}, the proof is finished.
\end{proof}

%% file: fig-bprime.tex
\begin{figure}[bht]
\begin{center}
\begin{tikzpicture}[scale = 1.3]
\begin{scope}
\draw[dotted] (0,0) circle (1.4cm);
\node[invvertex] (cr) at (1.4,0){};
\node[invvertex] (cl) at (-1.4,0){};
\node[novertex] (l) at (0,-1.7){\textbf{I}};

\node[hvertex] (b2) at (-0.6,0){};
\node[novertex] (lb2) at (-0.6,-0.2){$b'$};
\node[hvertex] (c1) at (0.6,0){};
\node[novertex] (lc1) at (0.68,-0.2){$c_1$};
\node[hvertex] (c2) at (1,0){};
\node[novertex] (lc2) at (1.15,-0.2){$c_2$};

\node[hvertex] (x) at (-0.6,0.4){};
\node[novertex] (lx) at (-0.52,0.48){$x$};
\node[hvertex] (a2) at (-0.9,0.4){};
\node[novertex] (la2) at (-1.07,0.4){$a'$};
\node[hvertex] (a) at (-0.3,0.4){};
\node[novertex] (la) at (-0.15,0.355){$a$};
\node[hvertex] (b) at (-0.6,0.8){};
\node[novertex] (lb) at (-0.75,0.85){$b$};

\draw[decorate,decoration={snake,amplitude=.4mm,segment length=2mm,post length=0mm}] (cr) -- (c2);
\draw[decorate,decoration={snake,amplitude=.4mm,segment length=2mm,post length=0mm}] (cl) -- (b2);
\draw[decorate,decoration={snake,amplitude=.4mm,segment length=2mm,post length=0mm}] (b2) -- (c1);

\draw[hedge] (c1) -- (c2);

\draw[hedge] (x) -- (b);
\draw[hedge] (x) -- (b2);
\draw[hedge] (x) -- (a);
\draw[hedge] (x) -- (a2);
\draw[hedge] (b2) -- (a);
\draw[hedge] (b2) -- (a2);
\draw[hedge] (a) -- (b);
\draw[hedge] (a2) -- (b);

\draw[hedge] (c1) -- (b);
\draw[hedge] (c2) -- (b);
\end{scope}

\begin{scope}[shift={(3.2,0)}]
\draw[dotted] (0,0) circle (1.4cm);
\node[invvertex] (cr) at (1.4,0){};
\node[invvertex] (cl) at (-1.4,0){};
\node[novertex] (cu) at (0,1.48){};
\node[novertex] (cd) at (0,-1.48){};
\node[novertex] (l) at (0,-1.7){\textbf{II}};

\node[hvertex] (b2) at (-0.6,0){};
\node[novertex] (lb2) at (-0.6,-0.2){$b'$};
\node[hvertex] (c1) at (0.6,0){};
\node[novertex] (lc1) at (0.68,-0.2){$c_1$};
\node[hvertex] (c2) at (1,0){};
\node[novertex] (lc2) at (1.15,-0.2){$c_2$};

\node[hvertex] (x) at (-0.6,0.4){};
\node[novertex] (lx) at (-0.52,0.48){$x$};
\node[hvertex] (a2) at (-0.9,0.4){};
\node[novertex] (la2) at (-1.07,0.4){$a'$};
\node[hvertex] (a) at (-0.3,0.4){};
\node[novertex] (la) at (-0.15,0.355){$a$};
\node[hvertex] (b) at (-0.6,0.8){};
\node[novertex] (lb) at (-0.75,0.85){$b$};

\draw[decorate,decoration={snake,amplitude=.4mm,segment length=2mm,post length=0mm}] (cr) -- (c2);
\draw[decorate,decoration={snake,amplitude=.4mm,segment length=2mm,post length=0mm}] (cl) -- (b2);
\draw[decorate,decoration={snake,amplitude=.4mm,segment length=2mm,post length=0mm}] (b2) -- (c1);
\draw[hedge] (c1) -- (c2);

\draw[hedge] (x) -- (b);
\draw[hedge] (x) -- (b2);
\draw[hedge] (x) -- (a);
\draw[hedge] (x) -- (a2);
\draw[hedge] (b2) -- (a);
\draw[hedge] (b2) -- (a2);
\draw[hedge] (a) -- (b);
\draw[hedge] (a2) -- (b);

\draw[hedge] (c2) -- (b);
\draw[hedge] (cu) -- (b);
\draw[hedge] (c1) -- (cd);

\end{scope}

\begin{scope} [shift={(6.4,0)}]
\draw[dotted] (0,0) circle (1.4cm);
\node[invvertex] (cr) at (1.4,0){};
\node[invvertex] (cl) at (-1.4,0){};
\node[novertex] (l) at (0,-1.7){\textbf{III}};
\node[novertex] (cu) at (0,1.48){};
\node[novertex] (cd) at (0,-1.48){};

\node[hvertex] (b2) at (-0.6,0){};
\node[novertex] (lb2) at (-0.6,-0.2){$b'$};
\node[hvertex] (c1) at (0.6,0){};
\node[novertex] (lc1) at (0.68,-0.2){$c_1$};
\node[hvertex] (c2) at (1,0){};
\node[novertex] (lc2) at (1.15,-0.2){$c_2$};

\node[hvertex] (x) at (-0.6,0.4){};
\node[novertex] (lx) at (-0.52,0.48){$x$};
\node[hvertex] (a2) at (-0.9,0.4){};
\node[novertex] (la2) at (-1.07,0.4){$a'$};
\node[hvertex] (a) at (-0.3,0.4){};
\node[novertex] (la) at (-0.15,0.355){$a$};
\node[hvertex] (b) at (-0.6,0.8){};
\node[novertex] (lb) at (-0.75,0.85){$b$};

\draw[decorate,decoration={snake,amplitude=.4mm,segment length=2mm,post length=0mm}] (cr) -- (c2);
\draw[decorate,decoration={snake,amplitude=.4mm,segment length=2mm,post length=0mm}] (cl) -- (b2);
\draw[decorate,decoration={snake,amplitude=.4mm,segment length=2mm,post length=0mm}] (b2) -- (c1);
\draw[hedge] (c1) -- (c2);

\draw[hedge] (x) -- (b);
\draw[hedge] (x) -- (b2);
\draw[hedge] (x) -- (a);
\draw[hedge] (x) -- (a2);
\draw[hedge] (b2) -- (a);
\draw[hedge] (b2) -- (a2);
\draw[hedge] (a) -- (b);
\draw[hedge] (a2) -- (b);

\draw[hedge] (c1) -- (b);
\draw[hedge] (cu) -- (b);
\draw[hedge] (cd) -- (c2);

\end{scope}

%
%
%
%
%
%
%
%
%

\end{tikzpicture}
\end{center}

\caption{The three possible embeddings of the edges $bc_1$ and $bc_2$ in $G$}
\label{fig:bprime}
\end{figure}
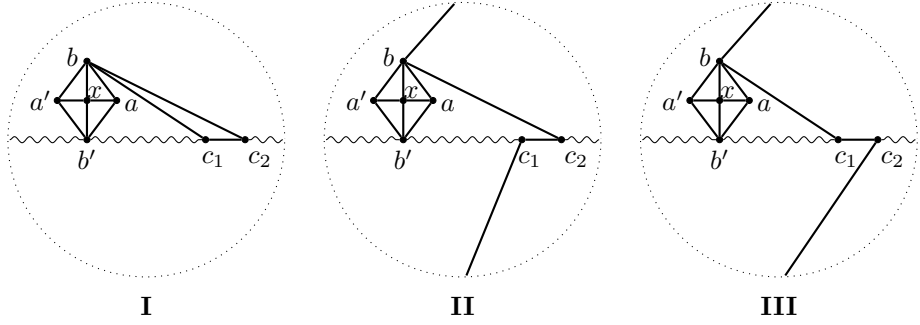

%% file: main_document.bbl
\begin{thebibliography}{10}
\bibitem{tpquad}
E.~Fuchs and L.~Gellert.
\newblock Quadrangulations of the projective plane are $t$-perfect if and only
  if they are bipartite.
\newblock {\em ArXiv e-prints}, June 2016.
\end{thebibliography}
